\documentclass{article}
\usepackage{amsmath,amssymb,amscd}
\usepackage{amsthm}
\usepackage{mathptmx}
\numberwithin{equation}{section}

\newtheorem{thm}{\indent \sc Theorem}[section]
\newtheorem{prop}[thm]{\indent \sc Proposition}
\newtheorem{cor}[thm]{\indent \sc Corollary}
\newtheorem{lem}[thm]{\indent \sc Lemma}

\newtheorem{defi}[thm]{\indent \sc Definition}

\newtheorem{rem}[thm]{\indent \sc Remark}

\newcommand{\spec}{\operatorname{Spec}}
\newcommand{\et}{\mathrm{\acute{e}t}}
\newcommand{\Nis}{\mathrm{Nis}}
\newcommand{\Zar}{\mathrm{Zar}}

\newcommand{\HO}{\operatorname{H}}

\newcommand{\HB}{\operatorname{H}_{\mathrm{B}}}

\newcommand{\BO}{\operatorname{C}}

\usepackage[all]{xy}

\begin{document}
\author{Makoto Sakagaito}
\title{On a generalized Brauer group in mixed characteristic cases, II}
\date{}
\maketitle
\begin{center}
Indian Institute of Technology Gandhinagar
\\
\textit{E-mail address}: sakagaito43@gmail.com
\end{center}
\begin{abstract}
As an extension of an author's previous paper (\cite[Theorem 1.1]{Sak}), 
we prove the Gersten-type conjecture for 
the mod $p$ \'{e}tale motivic cohomology
over a local ring of a smooth scheme over the spectrum of a discrete valuation ring 
of mixed characteristic $(0, p)$.
We also prove the $\mathbb{P}^{1}$-homotopy invariance 
for the generalized Brauer group
which is a generalization of a result due to
Saltman (\cite{Salt}) 
in mixed characteristic cases.    
\end{abstract}
\section{Introduction}

%
Let $A$ be a regular local ring with $\operatorname{dim}(A)\leq 1$, 
$X$ an equi-dimensional scheme over $\spec A$,
$\mathbb{Z}(n)_{\et}$ and $\mathbb{Z}(n)_{\Zar}$ be 
Bloch’s cycle complex 
for the \'{e}tale and Zariski topology, respectively (cf.\cite[p.779]{Ge}).
Denote
$\mathbb{Z}/m(n)_{\et}=\mathbb{Z}(n)_{\et}\otimes\mathbb{Z}/m\mathbb{Z}$
for a positive integer $m$.

\vspace{2.0mm}

Suppose that $X$ is smooth over $\spec(A)$. 
Let $R$ be the local ring $\mathcal{O}_{X, y}$ of $X$ at $y$ and
$\operatorname{Spec}(R)^{(i)}$ the set of points $x\in \spec(R)$ of
codimension $i$. Let us denote
\begin{equation*}
\HO^{t}_{x}(
R_{\et}, \mathbb{Z}/m(n)
)
:=
\displaystyle\lim_{\substack{\to \\ x\in U}}
\HO^{t}_{\bar{\{x\}}\cap U}
(U,
\mathbb{Z}/m(n)
)
\end{equation*}
where $U$ runs through open neighborhoods of $x$ in $\spec(R)$.
Then there is a sequence
\begin{align}\label{InBO}
0&\to 
\HO^{s}_{\et}(R, \mathbb{Z}/m(n))
\to
\bigoplus_{x\in \spec(R)^{(0)}}
\HO^{s}_{x}(R_{\et}, \mathbb{Z}/m(n))
\nonumber 
\\
&\to \bigoplus_{x\in \spec(R)^{(1)}}
\HO^{s+1}_{x}(R_{\et}, \mathbb{Z}/m(n)) 
\to\cdots
\end{align}
(cf. \cite[Part 1, (1.3)]{C-H-K} and Definition \ref{DefBO}) and it is known (or it is directly proved by known results) 
that the sequence (\ref{InBO}) is exact
in the following cases:
\begin{thm}\upshape\label{Intref}
Let the notations be the same as above. 
Let $D^{b}(X, \mathbb{Z}/m\mathbb{Z})$ be the derived category of bounded complexes of
\'{e}tale $\mathbb{Z}/m\mathbb{Z}$-sheaves on $X_{\et}$.
Then,
\begin{itemize}
\item[(a)] (cf. \cite[Theorem 1.5]{Ge-L2}, \cite[Corollary 2.2.2]{C-H-K}) In the case where $A$ is a field and $l$
is prime to $\operatorname{char}(A)$, then we
have an isomorphism
\begin{equation*}
\mathbb{Z}/l(n)_{\et}
\simeq
\mu_{l}^{\otimes n}
\end{equation*}
in $D^{b}(X_{\et}, \mathbb{Z}/l\mathbb{Z})$ and the sequence (\ref{InBO}) is exact for any $s$. Here $\mu_{l}$ is the sheaf of $l$-th roots of unity.
\item[(b)] (cf. \cite[p.787, \S 5, (12)]{Ge}, \cite[p.600, Theorem 4.1]{Sh}) In the case where $A$ is a field of positive
characteristic $p>0$,
then we have an isomorphism
\begin{equation*}
\mathbb{Z}/p^{r}(n)_{\et} 
\simeq
W_{r}\Omega^{n}_{X, \log}[-n]
\end{equation*}
in $D^{b}(X_{\et}, \mathbb{Z}/p^{r}\mathbb{Z})$ 
and the sequence (\ref{InBO}) is exact 
for $m=p^{r}$.
Here $W_{r}\Omega^{n}_{X, \log}$ 
is the logarithmic de Rham-Witt sheaf of $X$
(cf. \cite{I}, \cite[p.575, Definition 2.6]{Sh}).
\item[(c)] (cf. \cite[p.524, Conjecture 1.4.1 (1)]{SaP},
Proposition \ref{PTT}) 
In the case where $A$ is a discrete valuation ring of mixed-characteristic $(0, p)$, 
then we have an isomorphism 
\begin{equation}\label{qiTT}
\mathbb{Z}/p^{r}(n)_{\et}
\simeq  
\mathfrak{T}_{r}(n)_{X}
\end{equation}
in $D^{b}(X_{\et}, \mathbb{Z}/p^{r}\mathbb{Z})$ 
and
the sequence (\ref{InBO}) is exact for $s\leq n$ and $m=p^{r}$.
Here $\mathfrak{T}_{r}(n)_{X}$ is the $p$-adic Tate twist of $X$
(cf. \cite[p.537--538, Definition 4.2.4]{SaP},
\cite[p.187, Remark 3.7]{SaR}).
\item[(d)] (cf. \cite[p.30, Theorem 1.1]{Sak}) In the case where $A$ is a discrete valuation ring
of mixed characteristic $(0, p)$,
the sequence (\ref{InBO}) is exact at the first two
terms for $s=n+1$ and $m=p^{r}$.
\end{itemize}
\end{thm}
The exactness of the sequence (\ref{InBO}) in the case of Theorem \ref{Intref} (c)
directly follows from
the Gersten-type conjecture for the motivic cohomology (\cite[p.774, Theorem 1.2.5]{Ge}),
the Beilinson-Lichtenbaum conjecture (\cite[p.774, Theorem 1.2.2]{Ge}, \cite{V})
and the purity for $\mathfrak{T}_{r}(n)$ 
(cf. \cite[p.540, Theorem 4.4.7]{SaP}, \cite[p.187, Remark 3.7]{SaR}).
See Proposition \ref{PTT}.

\vspace{2.0mm}

One of objective of this paper is to extend Theorem \ref{Intref} (d). More precisely, we show that
the sequence (\ref{InBO}) is exact for $s\geq n+1$ and $m=p$ in the case  where
$A$ is a discrete valuation ring of mixed characteristic $(0, p)$
and $A$ contains $p$-th roots of unity.
In this paper, we call the exactness of the sequence (\ref{InBO}) for $m=p$ as the Gersten-type
conjecture for the mod $p$ \'{e}tale motivic cohomology.

In order to prove this result, it is effective to know about the object
$\HO_{x}^{t}(R_{\et}, \mathbb{Z}/p(n))$ 
for a point $x$ of the closed fiber of $\spec(R)$. 
Since we have the isomorphism (\ref{qiTT}), it suffices to know about the object
\begin{math}
\HO_{x}^{t}(X_{\et}, \mathfrak{T}_{r}(n))
\end{math}
for a point $x$ of the closed fiber of $X$.

In \S\ref{Ttexact}, we investigate the object
$\HO^{t}_{x}(\mathfrak{X}_{\et}, \mathfrak{T}_{r}(n))$
in the case where $\mathfrak{X}$ is a regular semistable family over 
$\spec(A)$ (cf. \cite[p.526, 1.11]{SaP}), 
that is, $\mathfrak{X}$ is  everywhere \'{e}tale locally isomorphic to
\begin{equation*}
\spec(
A[T_{0}, T_{1}, \cdots, T_{N}]/
(T_{0}T_{1}\cdots T_{a}-\pi)
)    
\end{equation*}
for some integer $a$ with 
$0\leq a\leq N=\operatorname{dim}(\mathfrak{X}/A)$,
where $\pi$ is a prime element of $A$.
In \S\ref{Ttexact}, we show the followings:

\begin{thm}\label{IntTh1}\upshape(Theorem \ref{iW})
Let $A$ be a discrete valuation ring of mixed characteristic $(0, p)$, 
$\mathfrak{X}$ a regular semistable family
over $\spec(A)$
and
$\iota: Y\to \mathfrak{X}$ 
the inclusion of
the closed fiber of $\mathfrak{X}$.
Let 
$n\geq 0$,
$s\geq 1$ 
be integers,
$\mathfrak{X}^{(s)}$ the set of points on $\mathfrak{X}$ of codimension $s$
and $x\in Y\cap \mathfrak{X}^{(s)}$.

Then we have an isomorphism
\begin{equation*}
\HO_{x}^{n+s}(
Y_{\et}, 
\iota^{*}\mathfrak{T}_{r}(n))   
\xrightarrow{\simeq}
\HO_{x}^{s}(
Y_{\et},
\lambda_{r}^{n}
)
\end{equation*}
for any integer $r>0$.
Here 
\begin{equation*}
\lambda_{Y, r}^{n}
:=
\operatorname{Im}\Bigl(
d\log:
(\mathbb{G}_{m, Y})^{\otimes n}
\to
\bigoplus_{y\in Y^{(0)}}i_{x *}
W_{r}\Omega^{n-1}_{y, \log}
\Bigr)
\end{equation*}
and $i_{y}$ denotes the canonical map $y\hookrightarrow Y$
for a point $y\in Y$. 
%
Moreover, we have an isomorphism
\begin{equation}
\HO_{x}^{n+t}(
Y_{\et},
\iota^{*}\mathfrak{T}_{r}(n)
)
\simeq 
0
\end{equation}
for any integers $r>0$ and $t>s$.
\end{thm}

In the case where $Y$ is a smooth variety, $\lambda^{n}_{r}$ agrees with $W_{r}\Omega^{n}_{Y, \log}$ 
(\cite[p.575, Definition 2.6]{Sh}). 
Let $j: U\hookrightarrow \mathfrak{X}$ be the open complement
$\mathfrak{X}\setminus Y$.
By using Theorem \ref{IntTh1}, the distinguished triangle
\begin{equation*}
\cdots\to 
j_{!}\mu_{p^{r}}^{\otimes n}
\to
\mathfrak{T}_{r}(n)
\to
\iota_{*}\iota^{*}\mathfrak{T}_{r}(n)
\to\cdots
\end{equation*}
(cf. \cite[pp.75--76, II, Remark 3.13]{M} and \cite[p.537, Definition 4.2.4]{SaP})
induces the following exact sequences:
\begin{thm}\upshape(Theorem \ref{exijS} and Corollary \ref{exmut})\label{Intext2}
Let the notations be the same as in Theorem \ref{IntTh1} and
$j: U\hookrightarrow \mathfrak{X}$ the open complement
$\mathfrak{X}\setminus Y$. 
Then we have an exact sequence
\begin{align*}
0
\to
\HO_{x}^{s}(
Y_{\et},
\lambda_{r}^{n}
)
\to
\HO_{x}^{n+s+1}(
\mathfrak{X}_{\et},
j_{!}\mu_{p^{r}}^{\otimes n}
)
\to
\HO_{x}^{n+s+1}(
\mathfrak{X}_{\et},
\mathfrak{T}_{r}(n)
)
\to 0
\end{align*}
and an isomorphism
\begin{equation*}
\HO_{x}^{n+t}(
\mathfrak{X}_{\et}, j_{!}\mu_{p^{r}}^{\otimes n}
)
\xrightarrow{\simeq}
\HO_{x}^{n+t}(
\mathfrak{X}_{\et},
\mathfrak{T}_{r}(n)
)
\end{equation*}
for $t\geq s+2$. Moreover, we have an exact sequence
\begin{align*}
0\to 
\HO_{x}^{s}(
Y_{\et},
\lambda_{1}^{n}
)
\to
\HO_{x}^{n+s+1}(
\mathfrak{X}_{\et},
\mathfrak{T}_{1}(n-1)
)
\to
\HO_{x}^{n+s+1}(
\mathfrak{X}_{\et},
\mathfrak{T}_{1}(n)
)
\to 0
\end{align*}
and an isomorphism
\begin{equation*}
\HO_{x}^{n+t}(
\mathfrak{X}_{\et}, j_{!}\mu_{p}^{\otimes n}
)
\xrightarrow{\simeq}
\HO_{x}^{n+t}(
\mathfrak{X}_{\et},
\mathfrak{T}_{1}(n)
)
\end{equation*}
for $t\geq s+2$
in the case where $A$ contains $p$-th roots of unity.
\end{thm}
By using the coniveau spectral sequence (cf. Theorem \ref{sp}), Theorem \ref{Intref} and Theorem \ref{Intext2}, we are able to
prove the exactness of the sequence (\ref{Intref}) in the case where $R$ is a henselian local ring as follows.
\begin{thm}\label{Intpre}\upshape (Theorem \ref{premain}) 
Let $A$ be a discrete valuation ring of mixed characteristic $(0, p)$ and  
$\mathfrak{X}$ a smooth scheme over $\spec(A)$.
Suppose that
$A$ contains $p$-th roots of unity. Let $R$ be the henselization of the local ring 
$\mathcal{O}_{\mathfrak{X}, x}$ of $\mathfrak{X}$ at $x$. 
Then the sequence (\ref{InBO}) is exact for $s\geq n+1$ and $m=p$.
\end{thm}

In \S\ref{SKH}, we prepare to extend Theorem \ref{Intpre}.
In this section, we prove the following:
\begin{prop}\label{Intsur}\upshape(Proposition \ref{Bsur})
Let $A$ be a discrete valuation ring of mixed characteristic $(0, p)$ and $\pi$ a prime element of $A$.
Let $\mathfrak{X}$ be a smooth scheme over $\spec(A)$ and 
$R$ the local ring $\mathcal{O}_{\mathfrak{X}, x}$ of $\mathfrak{X}$ at
a point $x$ of the closed fiber of $\mathfrak{X}$. Then the homomorphism
\begin{equation*}
\HO^{n+1}_{\et}(
R,
\mathbb{Z}/p(n)
) 
\to
\HO^{n+1}_{\et}(
R/(\pi),
\mathbb{Z}/p(n)
)
\end{equation*}
is surjective for integers $n\geq 0$.
\end{prop}

In \S\ref{Sahe}, we extend Theorem \ref{Intpre} as follows.

\begin{thm}\label{IntGCL}\upshape(Theorem \ref{mainGC})
Let $A$ be a discrete valuation ring of mixed characteristic $(0, p)$,  
$\mathfrak{X}$ a smooth scheme over $\spec(A)$
and $R$ the local ring $\mathcal{O}_{\mathfrak{X}, x}$ of $\mathfrak{X}$ at $x\in \mathfrak{X}$.
Suppose that $A$ contains $p$-th roots of unity. Then
the sequence (\ref{InBO}) is exact for $s\geq n+1$ and $m=p$.
\end{thm}

In \S\ref{Sahe}, we first prove that Theorem \ref{IntGCL} holds if the homomorphism
\begin{equation}\label{InRk}
\HO^{n+k}_{\et}(R, \mathbb{Z}/p(n))
\to
\HO^{n+k}_{\et}(k(R), \mathbb{Z}/p(n))
\end{equation}
is injective for $k\geq 2$ (see Proposition \ref{Asin}). Here $k(R)$ is the fraction field of $R$. 
At the end of this section, we prove the injectivity of the homomorphism (\ref{InRk}) by using
the theory of the $p$-adic vanishing cycle (cf. \S \ref{RVT}) and Proposition \ref{Intsur}.

In \S\ref{SP1h}, we consider global cases. In this section, we prove the following:
\begin{thm}\upshape(Theorem \ref{P1h} and Corollary \ref{HBP1})\label{IntHp}
Let $A$ be a discrete valuation ring 
of mixed characteristic $(0, p)$.
Then we have an isomorphism
\begin{equation*}
\HO_{\Zar}^{s}
\left(
\mathbb{P}^{m}_{A},
\tau_{\geq n+1}
R\epsilon_{*}\mathbb{Z}/p^{r}(n)
\right)
=
\bigoplus_{j=0}^{s-n-1}
\HO_{\et}^{s-2j}
(
A, 
\mathbb{Z}/p^{r}(n-j)
)
\end{equation*}
for $n+1\leq s\leq n+m+1$ and any $r>0$ where
$\epsilon: X_{\et}\to X_{\Zar}$ is the canonical map of sites
and $\epsilon_{*}$ is the forgetful functor.

In particular, we have an isomorphism
\begin{equation}\label{Ppbt}
\HB^{n}(\mathbb{P}^{s}_{A})_{p^{r}}
\simeq
\HB^{n}(A)_{p^{r}}
=
\HO^{n}_{\et}(A, \mathbb{Z}/p^{r}(n-1))
\end{equation}
for any positive integers $n$, $s$ and $r$. 
Here $\HB^{*}(X)$ is a generalized Brauer group of
a smooth scheme $X$ (cf. \cite[p.40, Definition 3.1]{Sak}) and $\HB^{*}(X)_{m}$
is the $m$-torsion subgroup of
$\HB^{*}(X)$.
\end{thm}
In this paper, we call  the isomorphism (\ref{Ppbt}) 
the $\mathbb{P}^{1}$-homotopy invariance for 
the generalized Brauer group.
Let $X$ be a smooth scheme over the spectrum of a regular local
ring $A$ with $\operatorname{dim}(A)\leq 1$
and $k(X)$ 
the ring of rational functions on $X$.
By \cite[p.51, Theorem 4.6]{Sak}, we have the exact sequence
\begin{equation*}
0\to
\HB^{n}(X)
\to
\operatorname{Ker}\left(
\HB^{n}(k(X))
\to
\prod_{x\in X^{(1)}}
\HB^{n}(k(\mathcal{O}_{X, x}^{sh}))
\right)
\to
\bigoplus_{x\in X^{(1)}}
\HB^{n-1}(\kappa(x))
\end{equation*}
where $\mathcal{O}_{X, x}^{sh}$ is the strict
henselization of a local ring $\mathcal{O}_{X, x}$, 
$\kappa(x)$ is the residue field of $x\in X$
and 
$X^{(1)}$ is the set of points of codimension $1$. 
Moreover, $\HB^{2}(X)$ is the Brauer group
$\operatorname{Br}(X)$ of $X$ by \cite[p.41, Proposition 3.3]{Sak}. 
So we can regard the isomorphism (\ref{Ppbt}) as a generalization of a result due to
Saltman
(\cite[p.55, Proposition 1.1]{Salt} and \cite[p.57, Proposition 1.7]{Salt}) 
in mixed characteristic cases.
In the equi-characteristic cases,
Auel-Bigazzi–B\"{o}hning–Graf von Bothmer (\cite[p.2480, Theorem 1.1]{A-B-B-G}) proved 
that a smooth proper universally
$\operatorname{CH}_{0}$-trivial variety over a field has
universally trivial Brauer group.
Moreover, this result is generalized for the mod $p$ unramified cohomology by Otabe (\cite[p.216, Corollary 1.3]{O}).

As an application of Theorem \ref{Intpre} 
and Theorem \ref{IntHp}, we have the following:
\begin{cor}\upshape(Corollary \ref{Ppbm2})
Let $A$ be a henselian discrete valuation ring
of mixed characteristic $(0, p)$. 
Suppose that $A$ contains $p$-th roots of unity.
Then we have an isomorphism
\begin{equation*}
\HO_{\Nis}^{s}\left(
\mathbb{P}^{m}_{A}, R^{n+1}\alpha_{*}\mathbb{Z}/p(n)
\right)
\simeq
\HO_{\et}^{n-s+1}
(A, \mathbb{Z}/p(n-s))
\end{equation*}
for $0\leq s\leq m$ where 
$\alpha: (\mathbb{P}_{A}^{m})_{\et}\to(\mathbb{P}_{A}^{m})_{\Nis}$ 
is the canonical map of sites
and $\alpha_{*}$ is the forgetful functor.

Moreover, the sequence
\begin{align*}
0\to&    
\bigoplus_{x\in (\mathbb{P}_{A}^{m})^{(0)}}
\HO^{n+2}_{x}\left(
(\mathbb{P}^{m}_{A})_{\et}, 
\mathbb{Z}/p(n)
\right)
\to
\bigoplus_{x\in (\mathbb{P}_{A}^{m})^{(1)}}
\HO^{n+3}_{x}\left(
(\mathbb{P}^{m}_{A})_{\et}, \mathbb{Z}/p(n)
\right) \nonumber \\
\to&
\bigoplus_{x\in (\mathbb{P}_{A}^{m})^{(2)}}
\HO^{n+4}_{x}\left(
(\mathbb{P}^{m}_{A})_{\et}, \mathbb{Z}/p(n)
\right)
\to\cdots
\end{align*}
is exact.
\end{cor}
\subsection*{Notations}
For a scheme $X$,
$X^{(i)}$ denotes the set of points
of codimension $i$ and
$k(X)$ denotes the ring of rational functions on $X$.
For a point $x\in X$, $\kappa(x)$ denotes the residue field
of $x$, $\mathcal{O}_{X, x}$ denotes the local ring of $X$ at $x\in X$
and
$\mathcal{O}^{h}_{X, x}$ denotes the henselization of $\mathcal{O}_{X, x}$.
$X_{\et}$, $X_{\Nis}$ and $X_{\Zar}$
denote the category of \'{e}tale schemes over a scheme $X$
equipped with
the \'{e}tale, Nisnevich and Zariski topology, respectively.
Let $\epsilon: X_{\et}\to X_{\Zar}$ be the canonical map of sites
and $\epsilon_{*}$ the forgetful functor.

Let $\mathbb{Z}(n)_{\et}$ (resp. $\mathbb{Z}(n)_{\Zar}$)
be Bloch’s cycle complex for
\'{e}tale (resp. Zariski) topology and 
$\mathbb{Z}/m(n)_{t}=\mathbb{Z}(n)_{t}\otimes\mathbb{Z}/m\mathbb{Z}$
for $t\in \{\et, \Zar\}$ and a positive integer $m$.
For an integer $m>0$, $\mu_{m}$ denotes the sheaf of $m$-th
roots of unity.
For a smooth scheme $Y$ over a field of positive characteristic
$p>0$, $\nu^{n}_{Y, r}=W_{r}\Omega_{Y, \log}^{n}$ denotes
the logarithmic de Rham-Witt sheaf
(cf. \cite{I}, \cite[p.575, Definition 2.6]{Sh}).
\section{Preliminary}\label{Prel}
In this section, we review results which is used in this paper.

\subsection{The coniveau spectral sequence}
In \cite{C-H-K}, the coniveau spectral sequence 
\begin{align*}
E^{s, t}_{1}
=\bigoplus_{x\in X^{(s)}}\HO^{s+t}_{x}(
X, \mathcal{F}
)
\Rightarrow
\HO^{s+t}_{\et}(
X, \mathcal{F}
)
\end{align*}
is constructed in the case where
$X$ is equidimensional and $\mathcal{F}$ is a sheaf of abelian groups on $X_{\et}$.
By using the same method as in the construction of the above coniveau spectral sequence, we have the following:
\begin{thm}\upshape(cf. \cite[Part 1, \S 1]{C-H-K})\label{sp}
Let $X$ be an equidimensional scheme and noetherian of dimension $d$.
Let $\mathcal{F}^{\bullet}$ be a bounded below complex of
sheaves of abelian groups on $X_{\et}$.
Then we have a spectral sequence
\begin{equation}\label{sp1}
E_{1}^{s, t}(X, \mathcal{F}^{\bullet})
=
\bigoplus_{x\in X^{(s)}}
\HO^{s+t}_{x}
(
X_{\et}, 
\mathcal{F}^{\bullet}
)
\Rightarrow
E^{s+t}(X, \mathcal{F}^{\bullet})
=
\HO_{\et}^{s+t}(X, \mathcal{F}^{\bullet}
).
\end{equation}
\end{thm}
\begin{proof}\upshape
The proof of the statement is the same as \cite[Part 1, \S 1]{C-H-K}.
We review the proof of \cite[Part 1, \S 1]{C-H-K} for convenience.
Consider a chain of closed subsets of $X$
\begin{equation*}
\overset{\to}{Z}:
\emptyset
\subset
Z_{d}\subset Z_{d-1}\subset
\cdots
\subset Z_{0}=X
\end{equation*}
where
$Z_{i}$ is
a closed subset of $X$ with
$\operatorname{codim}_{X}Z_{i}=i$. 
Take the convention that
$Z_{i}=0$ for $i>d$ and $Z_{i}=X$ for $i<0$.
For a pair $(Z_{p+1}\subset Z_{p})$,
the sequence
\begin{align*}
\cdots
&\to
\HO^{p+q}_{Z_{p+1}}(
X_{\et}, \mathcal{F}^{\bullet})
\xrightarrow{i^{p+1, q-1}}
\HO^{p+q}_{Z_{p}}(
X_{\et}, \mathcal{F}^{\bullet})
\xrightarrow{j^{p, q}}
\HO^{p+q}_{Z_{p}\setminus Z_{p+1}}(
(X\setminus Z_{p+1})_{\et},
\mathcal{F}^{\bullet})  \\
&\xrightarrow{k^{p, q}}
\HO^{p+q+1}_{Z_{p+1}}(
X_{\et}, \mathcal{F}^{\bullet})
\to
\cdots
\end{align*}
is exact 
(cf. \cite[p.92, III, Remark 1.26]{M}).

So we can construct an exact couple 
\begin{math}
C_{\underset{Z}{\to}}(D, E, i, j, k)
\end{math}
by setting
\begin{equation*}
D^{p, q}=\HO^{p+q}_{Z_{p}}(X_{\et}, 
\mathcal{F}^{\bullet}
)    
~~
\mathrm{and}
~~
E^{p, q}
=
\HO^{p+q}_{Z_{p}\setminus Z_{p+1}}(
(X\setminus Z_{p+1})_{\et}, 
\mathcal{F}^{\bullet}
):
\end{equation*}
\begin{equation*}
\xymatrix@C=10pt{
D^{p+1, q-1} \ar[rr]^-{i^{p+1, q-1}} && D^{p, q} \ar[ld]^-{j^{p, q}}
\\
& E^{p, q}\ar[lu]^-{k^{p, q}}_{(0, +1)}. & 
}   
\end{equation*}

By \cite[p.39, Theorem 2.8]{Mc},
this exact couple $C_{\underset{Z}{\to}}$ yields a spectral sequence of cohomological type. 
Put
\begin{equation*}
F^{p}\HO^{p+q}=
\operatorname{Im}\left(
D^{p, q}
\xrightarrow{i^{1, p+q-1}\circ\cdots\circ i^{p, q}}
D^{0, p+q}
\right)
~~\textrm{and}~~
K^{p, q}=
\operatorname{Ker}(i^{1, p+q-1}\circ\cdots\circ i^{p, q}).
\end{equation*}
Then we have an exact sequence
\begin{equation*}
0\to j^{p, q}(K^{p, q})
\to j^{p, q}(D^{p, q})
\to F^{p}\HO^{p+q}/F^{p+1}\HO^{p+q}
\to 0
\end{equation*}
by applying the snake lemma to the diagram
\footnotesize
\begin{equation*}
\xymatrix{
0 \ar[r]
& K^{p+1, q-1} \ar[r]\ar[d]
& D^{p+1, q-1} \ar[r]\ar[d]^{i^{p+1, q-1}}
& F^{p+1}\HO^{p+q} \ar[r]\ar@{^{(}-_>}[d]
& 0 
\\
0 \ar[r]
& K^{p, q} \ar[r]
& D^{p, q} \ar[r]
& F^{p}\HO^{p+q} \ar[r]
& 0.
}    
\end{equation*}
\normalsize

Consequently, the spectral sequence associated to 
the exact couple $C_{\underset{Z}{\to}}$ 
converges to 
$\HO^{*}_{\et}(X, \mathcal{F}^{\bullet})$ 
with the filtration $F^{p}\HO^{*}$
by \cite[p.40, Proposition 2.9]{Mc}.

Order the set of $(d+1)$-tuples
$\overset{\to}{Z}$ by
$\overset{\to}{Z}\leq \overset{\to}{Z^{\prime}}$ 
if
$Z_{q}\subset Z^{\prime}_{q}$
for all $q$.
Passing to the limit now defines a new couple $\underset{\to}{C}$
and the spectral sequence associated to the exact couple $\underset{\to}{C}$
still converges to $\HO^{*}_{\et}(X, \mathcal{F}^{\bullet})$.
If 
$T_{1}, \cdots, T_{r}$
are pairwise distinct closed subsets of $X$, then
we have an isomorphism
\begin{equation}\label{PCS}
\bigoplus\HO^{*}_{T_{i}}(
X, \mathcal{F}^{\bullet}
)
\simeq
\HO^{*}_{\cup T_{i}}(
X, \mathcal{F}^{\bullet}
)
\end{equation}
(cf. \cite[Part I, \S 1.2, The proof of Lemma 1.2.1]{C-H-K}).
Hence we have an isomorphism
\begin{equation*}
\underset{\to}{\lim}_{
\overset{\to}{Z}}  
\HO^{q+r}_{Z_{q}\setminus Z_{q+1}}(
(X\setminus Z_{q+1})_{\et}, 
\mathcal{F}^{\bullet}
)
\simeq
\bigoplus_{x\in X^{(q)}}
\HO^{q+r}_{x}(
X_{\et}, \mathcal{F}^{\bullet}
)
\end{equation*}
by (\ref{PCS}). This completes the proof.
\end{proof}
\begin{defi}\upshape\label{DefBO}
Let $m$ be a positive integer. 
Let $X$ be an essentially smooth scheme
over the spectrum of a regular local ring $A$ with $\operatorname{dim}(A)\leq 1$.
We call the following complex 
the Cousin complex
$(\BO^{q, n}_{m}(X)^{\bullet})$:
\begin{align*}
0
\to    
\bigoplus_{x\in X^{(0)}}
\HO^{n+q}_{x}
(
X_{\et}, 
\mathbb{Z}/m(n)
)
\xrightarrow{d_{1}^{0, n+q}(X, n)}
\bigoplus_{x\in X^{(1)}}
\HO^{n+q+1}_{x}
(
X_{\et}, \mathbb{Z}/m(n)
)
\\
\xrightarrow{d_{1}^{1, n+q}(X, n)}
\cdots 
\xrightarrow{d_{1}^{j-1, n+q}(X, n)}
\bigoplus_{x\in X^{(j)}}
\HO^{n+q+j}_{x}
(
X_{\et}, \mathbb{Z}/m(n)
)
\xrightarrow{d_{1}^{j, n+q}(X, n)}
\cdots,
\end{align*}
which is defined as the complex
$E_{1}^{\bullet, n+q}$ -terms of the coniveau spectral sequence
(\ref{sp1})
for $\mathcal{F}^{\bullet}=\mathbb{Z}/m(n)_{\et}$.
\end{defi}
\begin{lem}\upshape\label{spsqc}
Let
$X$ be an equidimensional scheme and noetherian of dimension $d$. Let 
\begin{math}
\alpha^{\bullet}: \mathcal{F}^{\bullet}
\to
\mathcal{G}^{\bullet}
\end{math}
be a homomorphism of bounded below complexes of sheaves of
abelian groups on $X_{\et}$.
Then the diagram
\begin{align*}
\xymatrix@C=44pt{
\displaystyle
\bigoplus_{x\in X^{(i)}}
\HO_{x}^{n+q+i}(X_{\et}, \mathcal{F}^{\bullet})
\ar[r]^-{d_{1}^{i,n+q}(
X, \mathcal{F}^{\bullet})}\ar[d]_{\alpha^{i, n+q}}
& \displaystyle
\bigoplus_{x\in X^{(i+1)}}
\HO_{x}^{n+q+i+1}(
X_{\et}, 
\mathcal{F}^{\bullet})
\ar[d]^{\alpha^{i, n+q+1}}   
\\
\displaystyle
\bigoplus_{x\in X^{(i)}} 
\HO_{x}^{n+q+i}(
X_{\et}, 
\mathcal{G}^{\bullet}
)
\ar[r]^-{d_{1}^{i, n+q}(X, \mathcal{G}^{\bullet})}
& \displaystyle
\bigoplus_{x\in X^{(i+1)}}
\HO_{x}^{n+q+i+1}(
X_{\et}, \mathcal{G}^{\bullet}
)
}    
\end{align*}
is commutative where the vertical arrows are induced by $\alpha^{\bullet}$.
\end{lem}
\begin{proof}\upshape
Since the composite map
\footnotesize
\begin{align*}
\xymatrix{
\HO^{q+r}_{
Z_{q}\setminus Z_{q+1}}(
(X\setminus Z_{q})_{\et},
\mathcal{F}^{\bullet}
)
\ar[r]^-{k^{q, r}}
&
\HO^{q+r+1}_{Z_{q+1}}
(X_{\et}, \mathcal{F}^{\bullet}
)
\ar[r]^-{j^{q+1, r}}
&
\HO^{q+r+1}_{
Z_{q+1}\setminus Z_{q+2}}(
(X\setminus Z_{q+2})_{\et},
\mathcal{F}^{\bullet}
)
}
\end{align*}
\normalsize
induces the differential
\begin{equation*}
d_{1}^{q, r}(X, \mathcal{F}^{\bullet}):   
\bigoplus_{x\in X^{(q)}}
\HO^{q+r}_{x}(
X_{\et}, 
\mathcal{F}^{\bullet}
)
\to
\bigoplus_{x\in X^{(q+1)}}
\HO^{q+r+1}_{x}(
X_{\et}, \mathcal{F}^{\bullet}
)
\end{equation*}
by passing to the limit, the statement directly follows.
\end{proof}
\subsection{Logarithmic Hodge-Witt sheaves}
Let $Y$ be a normal crossing variety over 
the spectrum of
a field $k$ of positive characteristic $p$, that is,
a pure-dimensional scheme of finite type over $\spec(k)$ which 
is everywhere \'{e}tale locally isomorphic to
\begin{equation*}
\spec(k[T_{0}, \cdots, T_{d}]/(T_{0}T_{1}\cdots T_{r}))    
\end{equation*}
for some integer $r$ with $0\leq r\leq d=\operatorname{dim}(Y)$.

Let us denote the set of points on $Y$ of codimension $s$ by $Y^{(s)}$.
For a point $y\in Y$, let $i_{y}$ be the canonical map
$y\hookrightarrow Y$. Then
we have logarithmic Hodge-Witt sheaves:%
\begin{equation*}
\lambda^{n}_{Y, r}
:=\operatorname{Im}
\left(
d\log:
(\mathbb{G}_{m, Y})^{\otimes n}
\to
\bigoplus_{y\in Y^{(0)}}
i_{y *}W_{r}\Omega^{n}_{y, \log}
\right)    
\end{equation*}
(cf. \cite[p.726, Definition 3.1.1]{SaL}) and
\begin{equation*}
\nu^{n}_{Y, r}:=\operatorname{Ker}\left(
\delta^{n}_{Y, r}:
\bigoplus_{y\in Y^{(0)}}i_{y *}W_{r}\Omega^{n}_{y, \log}
\to
\bigoplus_{y\in Y^{(1)}}i_{y *}W_{r}\Omega^{n-1}_{y, \log}
\right)    
\end{equation*}
(cf. \cite[p.715, Definition 2.1.1]{SaL}) which 
agree with $W_{r}\Omega^{n}_{Y, \log}$ in the case where $Y$ is a smooth scheme over $\spec(k)$ 
(cf. \cite[p.575, Definition 2.6]{Sh}, \cite[p.600, Theorem 4.1]{Sh}). 
\subsection{
$p$-adic vanishing cycle and
$p$-adic \'{e}tale Tate twist $\mathfrak{T}_{r}(n)$
}\label{RVT}
Let $A$ be a discrete valuation ring of mixed characteristic $(0, p)$
with the quotient field $K$. Let
$\pi$ be a prime element of $A$
and $k$ the residue field of $A$. 

Let $\mathfrak{X}$ be a regular semistable family over $\spec(A)$, that is,
a regular scheme of pure dimension which is flat of finite type over $\spec(A)$,
$\mathfrak{X}_{K}=\mathfrak{X}\otimes \spec(K)$ is
smooth over $\spec(K)$ and the special fiber $Y$ of $\mathfrak{X}$
is a reduced divisor with normal crossings on $\mathfrak{X}$ (cf. \cite[p.526, 1.11]{SaP}).

Let $j$ and $\iota$ be as follows:
\begin{equation*}
\mathfrak{X}_{K}
\xrightarrow{j} 
\mathfrak{X}
\xleftarrow{\iota}
Y.
\end{equation*}
First, we review the structure of the $p$-adic vanishing cycle
\begin{equation*}
M^{n}_{r}:=\iota^{*}R^{n}j_{*}\mu_{p^{r}}^{\otimes n}    
\end{equation*}
for integers $n\geq 0$ and $r>0$.
We define the \'{e}tale sheaf
$\mathcal{K}^{M}_{n, \mathfrak{X}_{K}/Y}$ on $Y$ as
\begin{equation*}
\mathcal{K}^{M}_{n, \mathfrak{X}_{K}/Y}
:=
(
\iota^{*}j_{*}\mathcal{O}^{\times}_{\mathfrak{X}_{K}}
)^{\otimes n}/J
\end{equation*}
where $J$ denotes the subsheaf which is generated by local sections of the form
\begin{equation*}
x_{1}\otimes\cdots\otimes x_{n}
~~(x_{i}\in \iota^{*}j_{*}\mathcal{O}^{\times}_{\mathfrak{X}_{K}})
\end{equation*}
with $x_{i}+x_{j}=0$ or $1$ for some $1\leq i<j\leq n$. By \cite[(1.2)]{B-K}, there is a natural map
\begin{equation}\label{TKFM}
\mathcal{K}^{M}_{n, \mathfrak{X}_{K}/ Y}
\to 
M^{n}_{r}
\end{equation}
and we define the filtrations $U^{*}$ and $V^{*}$ on the $p$-adic vanishing cycle 
$M^{n}_{r}$ by using the map (\ref{TKFM}) as follows:
\begin{defi}\label{DeUV}\upshape
\begin{description}
\item[(1)]
Let $U^{0}_{\mathfrak{X}_{K}}$ be the full sheaf 
$\iota^{*}j_{*}\mathcal{O}_{\mathfrak{X}_{K}}^{\times}$.
For $q\geq 1$, let
$U^{q}_{\mathfrak{X}_{K}}$ be the \'{e}tale subsheaf of 
$\iota^{*}j_{*}\mathcal{O}^{\times}_{\mathfrak{X}_{K}}$
which is generated by local sections of the form
$1+\pi^{q}\cdot a$ with $a\in \iota^{*}\mathcal{O}_{\mathfrak{X}}$.
We define the subsheaf
$U^{q}\mathcal{K}^{M}_{n, \mathfrak{X}_{K}/ Y}$
as the part which is generated by
$U^{q}_{\mathfrak{X}_{K}}\otimes\{\iota^{*}j_{*}\mathcal{O}^{\times}_{\mathfrak{X}_{K}}\}^{\otimes (n-1)}$.
\item[(2)] We define the subsheaf $U^{q}M^{n}_{r}$ $(q\geq 0)$
of $M^{n}_{r}$ as the image of
$U^{q}\mathcal{K}^{M}_{n, \mathfrak{X}_{K}/ Y}$
under the map (\ref{TKFM}).
We define the subsheaf $V^{q}M^{n}_{r}$ $(q\geq 0)$ of
$M^{n}_{r}$ as the part which is generated by $U^{q+1}M_{r}^{n}$
and the image of 
$U^{p}\mathcal{K}^{M}_{n-1, \mathfrak{X}_{K}/ Y}\otimes\langle a\rangle$
under the map (\ref{TKFM}).
\end{description}
\end{defi}
Since $\mathfrak{X}$ is a regular semistable family over $\spec(A)$, 
$\mathfrak{X}$ is everywhere \'{e}tale locally isomorphic to
\begin{equation*}
\spec
\left(
A[T_{1}, \cdots, T_{d}]/
(T_{1}\cdots T_{a}-\pi)
\right)
\end{equation*}
for some $1\leq a\leq d$. Then the sheaf of the modified differential modules 
$\omega^{q}_{\mathfrak{X}}$ are defined as
\begin{equation*}
\omega^{q}_{\mathfrak{X}}  
:=\displaystyle\wedge^{q}_{\mathcal{O}_{\mathfrak{X}}}
\omega^{1}_{\mathfrak{X}}
\end{equation*}
where $\omega^{1}_{\mathfrak{X}}$ is the $\mathcal{O}_{\mathfrak{X}}$-module 
which is generated
by $dT_{i}/T_{i}$ $(1\leq i\leq a)$ and $dT_{i}$ $(a+1\leq i\leq d)$
with the relation
\begin{equation*}
\displaystyle\sum^{a}_{i=1} 
dT_{i}/T_{i}
=0.
\end{equation*}
Moreover, we define $\omega^{q}_{Y}$ as
\begin{equation*}
\omega_{Y}^{q}:=
\omega_{\mathfrak{X}}^{q}\otimes_{\mathcal{O}_{\mathfrak{X}}}\mathcal{O}_{Y}
\end{equation*}
(cf. \cite[p.546, (1.5)]{Hy}. See also \cite[p.531]{SaP}). 

\vspace{1.0mm}

Then,  
\begin{align*}
\operatorname{gr}^{q}_{U/V}M^{n}_{r}
:=U^{q}M^{n}_{r}/V^{q}M^{n}_{r}
&&
\textrm{and}
&&
\operatorname{gr}^{q}_{V/U}M^{n}_{r}
:=V^{q}M^{n}_{r}/U^{q+1}M^{n}_{r}
\end{align*}
are described as follows:
\begin{thm}\upshape\label{CUV}(cf. Bloch-Kato \cite[pp.112--113, Corollary (1.4.1)]{B-K}
/ Hyodo \cite[p.548, (1.7) Corollary]{Hy},
Sato \cite[p.532, Theorem 3.3.7]{SaP},
\cite[p.184--185, Theorem 3.3]{SaR})
Let the notations be the same as above. Then
\begin{description}
\item[(1)] The map (\ref{TKFM}) is surjective, that is, the subsheaf $U^{0}M_{r}^{n}$
is the full sheaf $M^{n}_{r}$ for any integers $n\geq 0$ and $r>0$.
\item[(2)] Let $e$ be the absolute ramification index of $K$,
and let $r=1$. Then for $q$ with 
$1\leq q<e^{\prime}:=pe/(p-1)$, there are isomorphisms
\begin{align*}
\operatorname{gr}^{q}_{U/V}M^{n}_{1}
&\simeq
\begin{cases}
\omega^{n-1}_{Y}/\mathcal{B}^{n-1} & (p\nmid q)
\\
\omega^{n-1}_{Y}/\mathcal{Z}^{n-1}_{Y} & (p\mid q)
\end{cases} \\
~~
\\
\operatorname{gr}^{q}_{V/U}M^{n}_{1}
&\simeq 
\omega^{n-2}_{Y}/\mathcal{Z}^{n-2}_{Y}
\end{align*}
given by the following, respectively:
\footnotesize
\begin{align*}
\{
1+\pi^{q}a, x_{1}, 
\cdots, x_{n-1}
\}
~\mathrm{mod}~ 
V^{q}M^{n}_{1}  
&\mapsto
\begin{cases}
\bar{a}\cdot 
d\log(\bar{x_{1}})
\wedge\cdots\wedge d\log(\bar{x_{n-1}})
~\mathrm{mod}~\mathcal{B}^{n-1}_{Y}
& (p\nmid q)
\\
\bar{a}\cdot
d\log(\bar{x_{1}})
\wedge\cdots\wedge d\log(\bar{x_{n-1}})
~\mathrm{mod}~\mathcal{Z}^{n-1}_{Y}
& (p\mid q)
\end{cases}
\\
~~
\\
\{
1+\pi^{q}a, x_{1}, \cdots, x_{n-2}, \pi
\}
~\operatorname{mod}~U^{q+1}M_{1}^{n}
&\mapsto
\bar{a}\cdot d\log(\bar{x_{1}})
\wedge\cdots\wedge d\log(\bar{x_{n-2}})
~\operatorname{mod}~\mathcal{Z}^{n-2}_{Y}
\end{align*}
\normalsize
where $\mathcal{B}^{m}_{Y}$ (resp. $\mathcal{Z}_{Y}^{m}$) denotes 
the image of $d: \omega^{m-1}_{Y}\to \omega^{m}_{Y}$
(resp. the kernel of 
$d: \omega^{m}_{Y}\to \omega^{m+1}_{Y}$),
$a$ denotes a local section of $\mathcal{O}_{X}$ and
$\bar{a}$ denotes its residue class in $\mathcal{O}_{Y}$.
\item[(3)] We have
\begin{equation*}
U^{q}M^{n}_{1}
=
V^{q}M^{n}_{1}
=0
\end{equation*}
for any $q\geq e^{\prime}$.
\end{description}
\end{thm}

Next, we review the structure of $M^{n}_{r}/U^{1}M^{n}_{r}$.
We define the \'{e}tale subsheaf $FM_{r}^{n}$ as the part which is generated by
$U^{1}M^{n}_{r}$ and the image of
$(\iota^{*}\mathcal{O}^{\times}_{\mathfrak{X}})^{\otimes n}$
under the map (\ref{TKFM}). Then we have the followings:

\begin{thm}\upshape\label{CFU}(cf. \cite[p.533, Theorem 3.4.2]{SaP}, \cite[p.186, Theorem 3.4]{SaR})
There are short exact sequences of sheaves on $Y_{\et}$
\begin{equation*}
\xymatrix@R=15pt{
0
\ar[r]
&
FM^{n}_{r}
\ar[r]
&
M^{n}_{r}
\ar[r]^{\sigma}
&
\nu^{n-1}_{Y, r}
\ar[r]
&
0,
\\
0
\ar[r]
&
U^{1}M_{r}^{n}
\ar[r]
&
FM^{n}_{r}
\ar[r]^{\tau}
&
\lambda_{r}^{n}
\ar[r]
&
0,
}
\end{equation*}
where 
$\sigma$ is induced by the composite map (\cite[p.530, (3.2.2)]{SaP}) of \'{e}tale sheaves and
\cite[p.530, Lemma 3.2.4]{SaP},
and $\tau$ sends a symbol
\begin{equation*}
\{
x_{1}, x_{2}, \cdots, x_{n}
\} 
~(x_{i}\in \iota^{*}\mathcal{O}^{\times}_{\mathfrak{X}})
\end{equation*}
to
\begin{equation*}
d\log(
\bar{x_{1}}
\otimes
\bar{x_{2}}
\otimes
\cdots
\otimes
\bar{x_{n}}
).    
\end{equation*}
\end{thm}

Now we review the definition of the $p$-adic \'{e}tale Tate twist.

\begin{defi}\upshape\label{DTT}(\textbf{$p$-adic \'{e}tale Tate twist}, cf. \cite[p.537, Definition 4.2.4]{SaP} 
and \cite[p.187, Remark 3.7]{SaR})
Let the notations be the same as above. For $n=0$, 
\begin{equation*}
\mathfrak{T}_{r}(0)
:=
\mathbb{Z}/p^{r}\mathbb{Z}.
\end{equation*}
For $n\geq 1$, $\mathfrak{T}_{r}(n)$ is defined as a complex which is fitted into
the distinguished triangle
\begin{align*}
\iota_{*}\nu^{n-1}_{Y, r}[-n-1]
\to
\mathfrak{T}_{r}(n)
\to
\tau_{\leq n}Rj_{*}\mu_{p^{r}}^{\otimes n}
\xrightarrow{\sigma^{n}_{\mathfrak{X}, r}}
\iota_{*}\nu^{n-1}_{Y, r}[-n]
\end{align*}
where the morphism $\sigma^{n}_{\mathfrak{X}, r}$ is induced by the homomorphism $\sigma$
in Theorem \ref{CFU}.
\end{defi}
\begin{rem}\upshape\label{RTT}
\begin{description}
\item[(1)] Since we have an isomorphism 
\begin{equation*}
j^{*}R^{s}j_{*}\mu_{p^{r}}^{\otimes n}
\simeq 0
\end{equation*}
for $s>0$, we have an isomorphism
\begin{equation*}
R^{n}j_{*}\mu_{p^{r}}^{\otimes n}
\xrightarrow{\simeq}
\iota_{*}M^{n}_{r}
\end{equation*}
and so we have an isomorphism
\begin{equation*}
\mathcal{H}^{t}(
\mathfrak{T}_{r}(n)
)
\simeq 
0
\end{equation*}
for integers $n\geq 0$ and $t>n$ by Definition \ref{DTT} and Theorem \ref{CFU}
(cf. \cite[p.520, Theorem 1.1.1]{SaP}).
%
\item[(2)]
By Definition \ref{DTT} and Theorem \ref{CFU},
we have an isomorphism
\begin{equation*}
FM^{n}_{r}
\simeq 
\mathcal{H}^{n}(
\iota^{*}\mathfrak{T}_{r}(n)
)
\end{equation*}
and so we
have a finite filtration of $\mathcal{H}^{n}(\iota^{*}\mathfrak{T}_{1}(n))$ whose
subquotients are described by 
$\lambda_{1}^{n}$ and the quotients of
$\omega^{n-1}_{Y}$ and $\omega^{n-2}_{Y}$
by Theorem \ref{CUV} and Theorem \ref{CFU}.
\item[(3)]
In the case where $\mathfrak{X}$ is a smooth scheme over
$\spec(A)$,
we have a quasi-isomorphism
\begin{equation*}
\mathbb{Z}/p^{r}(n)_{\et}^{\mathfrak{X}}
\simeq
\mathfrak{T}_{r}(n)_{\mathfrak{X}}
\end{equation*}
by \cite[p.524, Conjecture 1.4.1]{SaP}.
Here $\mathbb{Z}(n)_{\et}$ is Bloch's cycle complex for \'{e}tale
topology and $\mathbb{Z}/p^{r}(n)_{\et}=\mathbb{Z}(n)_{\et}\otimes\mathbb{Z}/p^{r}$.
\end{description}
\end{rem}

Moreover, we have the following purity theorem for $\mathfrak{T}_{r}(n)$.

\begin{thm}\label{PETP}\upshape(cf. \cite[p.540, Theorem 4.4.7]{SaP} and \cite[p.187, Remark 3.7]{SaR})
Let the notations be the same as above.
Let $Z\hookrightarrow Y$ be a closed immersion of pure codimension with
$\operatorname{codim}_{\mathfrak{X}}(Z)=c$ and
$i_{Z}$ the composite map
$Z\hookrightarrow Y\hookrightarrow \mathfrak{X}$. 
Suppose that $Z$ is a normal crossing variety over $\spec(k)$.

Then the morphism
\begin{equation*}
\tau_{\leq n+c}(\operatorname{Gys}^{n}_{i_{Z}}):
\nu^{n-c}_{Z, r}[-n-c]
\to
\tau_{\leq n+c}R(i_{Z})^{!}
\mathfrak{T}_{r}(n)_{\mathfrak{X}}
\end{equation*}
is an isomorphism 
where $\operatorname{Gys}^{n}_{i_{Z}}$ is defined in 
\cite[p.540, Definition 4.4.5]{SaP}.
\end{thm}

The following directly follows from the theory of the motivic cohomology (cf.\cite{Ge})
and the above.

\begin{prop}\upshape\label{PTT}
Let $\mathfrak{X}$ be a smooth scheme over the spectrum of a discrete valuation ring of mixed characteristic $(0, p)$
and $R$ (the henselization of) the local ring $\mathcal{O}_{\mathfrak{X}, x}$ of $\mathfrak{X}$ at $x$.   
Then the sequence
\begin{align*}
0&\to 
\HO^{s}_{\et}(R, \mathbb{Z}/p^{r}(n))
\to
\bigoplus_{x\in \spec(R)^{(0)}}
\HO^{s}_{x}(R_{\et}, \mathbb{Z}/p^{r}(n))
\nonumber 
\\
&\to \bigoplus_{x\in \spec(R)^{(1)}}
\HO^{s+1}_{x}(R_{\et}, \mathbb{Z}/p^{r}(n)) 
\to\cdots
\end{align*}
is exact for integers $s\leq n$ and $r>0$.
\end{prop}
\begin{proof}\upshape
By \cite[p.774, Theorem 1.2.5]{Ge}, we have an  exact sequence
\begin{align*}
0
&\to 
\HO^{s}_{\Zar}(
R,
\mathbb{Z}/p^{r}(n)
)
\to
\displaystyle\bigoplus_{x\in \spec(R)^{(0)}}
\HO^{s}_{\Zar}(
\kappa(x),
\mathbb{Z}/p^{r}(n)
) \\
&\to
\displaystyle\bigoplus_{x\in \spec(R)^{(1)}}
\HO^{s-1}_{\Zar}(
\kappa(x),
\mathbb{Z}/p^{r}(n-1)
)
\to \cdots
\end{align*}
where the sequence is induced by the spectral sequence
\begin{equation*}
E^{u, v}_{1}
=
\displaystyle\bigoplus_{x\in \spec(R)^{(u)}}
\HO^{2n-u+v}_{\Zar}(
\kappa(x),
\mathbb{Z}/p^{r}(n-u)
)
\Rightarrow
E^{u+v}
=\HO^{2n+u+v}_{\Zar}(
R_{\et},
\mathbb{Z}/p^{r}(n)
)
\end{equation*}
(cf. \cite[p.782]{Ge}).
So we have an exact sequence
\begin{align*}
0
&
\to 
\HO^{s}_{\et}(
R,
\mathfrak{T}_{r}(n)
)
\to
\displaystyle\bigoplus_{x\in \spec(R)^{(0)}}
\HO^{s}_{\et}(
\kappa(x),
\mathfrak{T}_{r}(n)
)  
\\
&
\to
\displaystyle\bigoplus_{x\in \spec(R)^{(1)}}
\HO^{s-1}_{\et}(
\kappa(x),
\mathfrak{T}_{r}(n-1)
)
\to\cdots
\end{align*}
for $s\leq n$ and $r>0$
by the Beilinson-Lichtenbaum conjecture (cf. \cite[p.774, Theorem 1.2.2]{Ge}, \cite{V}) 
and Remark \ref{RTT} (3). 
Hence the statement follows from the purity theorem for $\mathfrak{T}_{r}(n)$ 
(cf. Theorem \ref{PETP} and \cite[p.241, VI, Theorem 5.1]{M}) and Remark \ref{RTT} (3).
\end{proof}
\section{\'{E}tale hypercohomology of $\mathfrak{T}_{r}(n)$ with support}\label{Ttexact}

In this section, we prove Theorem \ref{Intpre} (cf. Theorem \ref{premain}). 
In order to prove Theorem \ref{Intpre}, we investigate 
\'{e}tale hypercohomology groups of $\mathfrak{T}_{r}(n)$ with support.

\begin{prop}\upshape\label{VaniG}
Let $A$ be an equidimensional catenary local ring with the maximal ideal $\mathfrak{m}$
and $\mathcal{F}$ a sheaf on $\spec(A)_{\et}$.
Suppose that
\begin{equation*}
\HO^{s}_{\et}(
A_{x}, \mathcal{F}
)
=0
\end{equation*}
for any $s\geq N$ and $x\in \spec(A)_{\et}$. Then we have
\begin{equation*}
\HO^{t}_{\mathfrak{m}}(
A_{\et},
\mathcal{F}
)
=0
\end{equation*}
for any $t\geq N+\operatorname{dim}(A)$.
\end{prop}
\begin{proof}\upshape
The proof of the statement is similar to the proof of \cite[Proposition 2.6]{Sak4}.
We prove the statement by induction on 
$\operatorname{dim}(A)$.

Suppose that $\operatorname{dim}(A)=0$. Then the statement directly follows from 
the assumption of the statement.

Suppose that $\operatorname{dim}(A)=d+1$ and the statement holds in the case where
$\operatorname{dim}(A)\leq d$. 
By Theorem \ref{sp}, we have a spectral sequence
\begin{equation*}
E_{1}^{s, t}=
\bigoplus_{x\in \spec(A)^{(s)}}
\HO^{s+t}_{x}(A_{\et}, \mathcal{F})
\Rightarrow
E^{s+t}=
\HO^{s+t}_{\et}(A, \mathcal{F}).
\end{equation*}
Then we have
\begin{equation*}
E^{s+t}=0
\end{equation*}
for $s+t\geq N+d+1$ by the assumption of the statement. So we have
\begin{equation*}
E_{\infty}^{d+1, t}=0
\end{equation*}
for $t\geq N$. 
Since the statement holds in the case where $\operatorname{dim}(A)\leq d$
by the assumption of induction,
we have
\begin{equation*}
E_{r}^{s, t}
=0
\end{equation*}
for $s\leq d$, $t\geq N$ and $r\geq 1$. Hence we have
\begin{equation*}
E_{1}^{d+1, t}=
E_{2}^{d+1, t}=
\cdots 
=
E_{\infty}^{d+1, t}
=0
\end{equation*}
for $t\geq N$. This completes the proof.
\end{proof}
\begin{cor}\upshape(cf.
\cite[Proposition 2.6]{Sak4}
)\label{VaniSP}
Let $A$ be an equidimensional catenary 
local ring of positive characteristic $p>0$ and 
$\mathfrak{m}$ the maximal ideal of $A$. 
Let $\mathcal{F}$ be a $p$-torsion sheaf on $\spec(A)_{\et}$. 
Then we have
\begin{equation*}
\HO^{t}_{\mathfrak{m}}(
A_{\et},
\mathcal{F}
)
=0
\end{equation*}
for $t\geq \operatorname{dim}(A)+2$.
\end{cor}
\begin{proof}\upshape
The statement follows from \cite[Expos\'{e} X, Th\'{e}or\`{e}me 5]{SGA4} and Proposition \ref{VaniG}.
\end{proof}
\begin{cor}\upshape\label{VaniSC}
Let $A$ be an equidimensional catenary local ring with the maximal ideal $\mathfrak{m}$.
Let $\mathcal{F}_{\Zar}$ be a quasi-coherent sheaf 
of $\mathcal{O}_{\spec(A)}$-module
on $\spec(A)_{\Zar}$ and let $\mathcal{F}$ be the corresponding sheaf
$W(\mathcal{F}_{\Zar})$
on $\spec(A)_{\et}$
(cf. \cite[p.48, II, Examples 1.2 (d)]{M}, \cite[p.51, II, Corollary 1.6]{M}). 
Then we have
\begin{equation*}
\HO^{t}_{\mathfrak{m}}(
A_{\et},
\mathcal{F}
)
=0
\end{equation*}
for $t\geq \operatorname{dim}(A)+1$.
\end{cor}
\begin{proof}\upshape
By
\cite[p.103, III, Lemma 2.15]{M} and
\cite[p.114, III, Remark 3.8]{M}, we have
\begin{equation*}
\HO^{N}_{\et}(
A_{x}, \mathcal{F}
)=
\HO^{N}_{\Zar}(
A_{x}, \mathcal{F}_{\Zar}
)
=0
\end{equation*}
for any point $x\in \spec(A)$ and $N\geq 1$. Hence the statement follows from Proposition \ref{VaniG}.
\end{proof}
\begin{lem}\upshape\label{VanU}
Let $\mathfrak{X}$ be a regular semistable family
over the spectrum of a discrete valuation ring of
mixed characteristic $(0, p)$
and
$\iota: Y\to \mathfrak{X}$ 
the inclusion of
the closed fiber of $\mathfrak{X}$.
Let $n\geq 0$ be an integer.
Then we have    
\begin{equation*}
\HO_{x}^{t}(
Y_{\et},
U^{1}M^{n}_{1})
=0
\end{equation*}
for $x\in Y\cap \mathfrak{X}^{(s)}$ and
$t\geq s$.
Here $U^{1}M^{n}_{1}$ is defined in Definition \ref{DeUV}.
\end{lem}
\begin{proof}\upshape
Let $x\in Y\cap \mathfrak{X}^{(s)}$.
By Corollary \ref{VaniSC}, we have
\begin{equation}\label{SWV}
\HO_{x}^{t}(Y_{\et}, 
\omega_{Y}^{n-1})
=
\HO_{x}^{t}(Y_{\et}, 
\omega_{Y}^{n-2})
=
0
\end{equation}
for $t\geq s$. 
Here $\omega_{Y}^{*}$ is defined in \S \ref{RVT}.
So we have
\begin{equation*}
\HO_{x}^{t}(
Y_{\et}, \omega_{Y}^{n-1}/\mathcal{B}^{n-1}_{Y}
)    
=
\HO_{x}^{t}(
Y_{\et}, \omega_{Y}^{n-1}/\mathcal{Z}^{n-1}_{Y}
)
=
\HO_{x}^{t}(
Y_{\et}, \omega_{Y}^{n-2}/\mathcal{Z}^{n-2}_{Y}
)
=0
\end{equation*}
for $t\geq s$ by (\ref{SWV}) and Corollary \ref{VaniSP}.
Hence the statement follows from Theorem \ref{CUV}.
\end{proof}
\begin{thm}\upshape\label{iW}
Let $\mathfrak{X}$ be a regular semistable family
over the spectrum of a discrete valuation ring of
mixed characteristic $(0, p)$ and
$\iota: Y\to \mathfrak{X}$ 
the inclusion of
the closed fiber of $\mathfrak{X}$.
Let 
$n\geq 0$,
$r>0$,
$s\geq 1$ 
be integers
and $x\in Y\cap \mathfrak{X}^{(s)}$.
Then the homomorphism
\begin{equation}\label{isoil}
\HO_{x}^{n+t}(
Y_{\et}, 
\iota^{*}\mathfrak{T}_{r}(n))   
\to
\HO_{x}^{t}(
Y_{\et},
\lambda_{r}^{n}
)
\end{equation}
is an isomorphism for any $r>0$ and $t=s$.
Here the homomorphism (\ref{isoil}) is induced by
the composite
\begin{align}\label{FITFL}
\iota^{*}\mathfrak{T}_{r}(n)  
\to
\iota^{*}\tau_{\geq n}\mathfrak{T}_{r}(n)
\xleftarrow{\simeq}
\iota^{*}\mathcal{H}^{n}(\mathfrak{T}_{r}(n))[-n]
\to
\lambda_{r}^{n}[-n]
\end{align}
where the last map is induced by $\tau$ in Theorem \ref{CFU}.
Moreover, we have an isomorphism
\begin{equation}\label{YxS0}
\HO_{x}^{n+t}(
Y_{\et},
\iota^{*}\mathfrak{T}_{r}(n)
)
\simeq 
0
\end{equation}
for any $r>0$ and $t>s$.

\end{thm}
\begin{proof}\upshape
Let $n$ be any non-negative integer. 
By the spectral sequence
\begin{equation*}
E^{u, v}_{2}
=
\HO_{x}^{u}\left(
Y_{\et},
\mathcal{H}^{v}(
\iota^{*}\tau_{\leq n-1}
\mathfrak{T}_{r}(n)
)
\right)    
\Rightarrow
\HO^{u+v}_{x}(
Y_{\et},
\iota^{*}\tau_{\leq n-1}\mathfrak{T}_{r}(n)
)
\end{equation*}
and Corollary \ref{VaniSP}, 
we have an isomorphism
\begin{equation}\label{SiTv}
\HO_{x}^{n+t}(
Y_{\et},
\iota^{*}\tau_{\leq n-1}
\mathfrak{T}_{r}(n)
)
\simeq
0
\end{equation}
for $t\geq s$ and
the isomorphism (\ref{YxS0}). 
So the homomorphism
\begin{equation*}
\HO_{x}^{n+t}(
Y_{\et},
\iota^{*}
\mathfrak{T}_{r}(n)
)    
\to
\HO_{x}^{t}(
Y_{\et},
\iota^{*}\mathcal{H}^{n}
(\mathfrak{T}_{r}(n)
)
)
\end{equation*}
is an isomorphism for $t\geq s$ and is surjective for $t=s-1$ by (\ref{SiTv}).
Moreover, the homomorphism
\begin{equation*}
\HO_{x}^{t}\left(
Y_{\et},
\iota^{*}\mathcal{H}^{n}(
\mathfrak{T}_{1}(n)
)
\right)
\to
\HO_{x}^{t}(
Y_{\et},
\lambda_{1}^{n}
)
\end{equation*}
is an isomorphism for $t=s$
and is surjective for $t=s-1$ by Lemma \ref{VanU}.

So the homomorphism (\ref{isoil}) is an isomorphism for $t=s$ and 
is surjective for $t=s-1$
in the case where $r=1$.
Moreover, we have the commutative diagram
%
\footnotesize
\begin{equation*}
\xymatrix@C=9pt{
\HO_{x}^{n+s-1}(
Y_{\et},
\iota^{*}\mathfrak{T}_{1}(n)
)
\ar[r]\ar[d]
&
\HO_{x}^{n+s}(
Y_{\et},
\iota^{*}\mathfrak{T}_{r}(n)
)
\ar[r]\ar[d]
&
\HO_{x}^{n+s}(
Y_{\et},
\iota^{*}\mathfrak{T}_{r+1}(n)
)
\ar[r]\ar[d]
&
\HO_{x}^{n+s}(
Y_{\et},
\iota^{*}\mathfrak{T}_{1}(n)
)
\ar[r]\ar[d]
&
0
\\
\HO_{x}^{s-1}(
Y_{\et},
\lambda_{1}^{n}
)
\ar[r]
&
\HO_{x}^{s}(
Y_{\et},
\lambda_{r}^{n}
)
\ar[r]
&
\HO_{x}^{s}
(
Y_{\et},
\lambda_{r+1}^{n}
)
\ar[r]
&
\HO_{x}^{s}
(
Y_{\et},
\lambda_{1}^{n}
)
}
\end{equation*}
\normalsize
where the leftmost map is surjective and
the rightmost map is an isomorphism.
Here the upper row is exact by \cite[pp.538--539, Proposition 4.3.1 (3)]{SaP} and (\ref{YxS0}).
Moreover, the bottom row is exact by 
\cite[p.728, Corollary 3.2.3]{SaL}.
Hence, by induction on $r$, the statement
follows from the five lemma.
\end{proof}
\begin{prop}\upshape\label{Otit}
Let the notations be the same as in Theorem \ref{iW}. 
Let 
$j: U\hookrightarrow \mathfrak{X}$ be the open complement
$\mathfrak{X}\setminus Y$.
Then the distinguished triangle
\begin{equation*}
\iota^{*}\mathfrak{T}_{r}(n)
\to
\iota^{*}\tau_{\leq n}Rj_{*}\mu_{p^{r}}^{\otimes n}
\to
\nu_{r}^{n-1}[-n]
\end{equation*}
(cf. \cite[p.537, Definition 4.2.4]{SaP})
induces an exact sequence
\begin{align*}
0\to    
\HO_{x}^{s}(
Y_{\et},
\lambda_{r}^{n}
)
\to
\HO_{x}^{n+s}(
Y_{\et},
\iota^{*}\tau_{\leq n}
Rj_{*}\mu^{\otimes n}_{p^{r}}
)
\to
\HO_{x}^{s}(
Y_{\et},
\nu_{r}^{n-1}
)
\to 0
\end{align*}
for $x\in Y\cap X^{(s)}$.
\end{prop}
\begin{proof}\upshape
By Theorem \ref{iW}, it suffices to show that
the homomorphism
\begin{equation}\label{injYiTt}
\HO_{x}^{n+s}(Y_{\et}, \iota^{*}\mathfrak{T}_{r}(n))
\to 
\HO_{x}^{n+s}(Y_{\et}, \iota^{*}\tau_{\leq n}
Rj_{*}\mu_{p^{r}}^{\otimes n}
)
\end{equation}
is injective. So it suffices to show that 
the homomorphism
\begin{equation}\label{SurYiTt}
\HO_{x}^{n+s-1}(
Y_{\et},
\iota^{*}\tau_{\leq n}Rj_{*}\mu_{p^{r}}^{\otimes n}
)    
\to
\HO_{x}^{s-1}(
Y_{\et},
\nu_{r}^{n-1}
)
\end{equation}
is surjective. 

Suppose that $s=1$. 
By \cite[p.93, III, Corollary 1.28]{M},
it suffices to show the statement in the case where 
$\mathfrak{X}$ is the spectrum of a henselian discrete valuation ring $R$.
Then
we have isomorphisms
%
\begin{align*}
&
\HO^{n+1}_{x}(
Y_{\et},
\iota^{*}\mathfrak{T}_{r}(n)
)
\simeq 
\HO^{n+1}_{\et}(
R,
\mathfrak{T}_{r}(n)
)
\end{align*}
and
\begin{align*}
\HO^{n+1}_{x}(
Y_{\et},
\iota^{*}\tau_{\leq n}Rj_{*}\mu_{p^{r}}^{\otimes n}
)
\simeq
\HO^{n+1}_{\et}(
R, \tau_{\leq n}Rj_{*}\mu_{p^{r}}^{\otimes n}
)
\end{align*}
by \cite[p.777, The proof of Proposition 2.2.b)]{Ge}. 
So the homomorphism (\ref{injYiTt}) corresponds to
the homomorphism
\begin{equation*}
\HO^{n+1}_{\et}(R, \mathfrak{T}_{r}(n))
\to
\HO^{n+1}_{\et}(
R, 
\tau_{\leq n}Rj_{*}\mu_{p^{r}}^{\otimes n}
)
\end{equation*}
and the homomorphism (\ref{injYiTt}) is an injective for $s=1$ by 
\cite[p.52, Corollary 4.7]{Sak} (or by \cite[Lemma 4.1]{Sak4}). 
Hence the homomorphism (\ref{SurYiTt}) is surjective for $s=1$.

Suppose that $s\geq 1$. For any point $x\in Y\cap X^{(s+1)}$, there exists a 
$y\in Y\cap X^{(s)}$ such that $\bar{\{y\}}\cap\spec(\mathcal{O}_{X, x})$
is the spectrum of a discrete valuation ring. Then the homomorphism
\begin{equation*}
\HO_{y}^{s-1}(
Y_{\et}, \nu_{r}^{n-1}
)
\to
\HO_{x}^{s}(
Y_{\et}, \nu_{r}^{n-1}
)
\end{equation*}
corresponds to
\begin{equation*}
\HO^{0}_{\et}(
\kappa(y),
\nu_{r}^{n-s}
)    
\to
\HO^{0}_{\et}(
\kappa(x),
\nu_{r}^{n-s-1}
)
\end{equation*}
by \cite[p.718, Theorem 2.4.2]{SaL} and so is surjective. Moreover, we have a commutative diagram
\begin{equation*}
\xymatrix{
\HO_{y}^{n+s-1}
(Y_{\et}, 
\iota^{*}\tau_{\leq n}Rj_{*}\mu_{p^{r}}^{\otimes n})
\ar[r]\ar[d]
&
\HO_{y}^{s-1}(
Y_{\et}, \nu_{r}^{n-1}
)\ar[d]
\\
\HO_{x}^{n+s}
(Y_{\et}, 
\iota^{*}\tau_{\leq n}Rj_{*}\mu_{p^{r}}^{\otimes n})
\ar[r]
&
\HO_{x}^{s}(
Y_{\et}, \nu_{r}^{n-1}
)
}    
\end{equation*}
where the right map is surjective. Hence the statement follows by the induction on $s$.
\end{proof}
\begin{lem}\label{SAP}\upshape
Let $Z$ be a closed scheme of a scheme $X$ and
$\iota: Y\to X$ a closed immersion of $X$. Then we have an isomorphism
\begin{equation*}
\HO^{s}_{Z}(
X_{\et}, 
\iota_{*}\mathcal{F}^{\bullet}
)    
\xrightarrow{\simeq}
\HO^{s}_{Z\cap Y}
(
Y_{\et},
\mathcal{F}^{\bullet}
)
\end{equation*}
for any $s$ and any bounded below complex $\mathcal{F}^{\bullet}$
of torsion sheaves on $Y_{\et}$.    
\end{lem}
\begin{proof}
The statement follows from the proper base change theorem. 
Let
\footnotesize
\begin{equation*}
\xymatrix{
Y \ar[d]_{\iota} 
&
U^{\prime}\ar[l]_{j^{\prime}}\ar[d]^{\iota^{\prime}}  
\\
X &\ar[l]^{j} U
}    
\end{equation*}
\normalsize
be Cartesian. Here $U=X\setminus Z$. Since $\iota$ is proper, we have an isomorphism
\begin{equation*}
j^{*}\iota_{*}\mathcal{F}^{\bullet}
\xrightarrow{\simeq}
(\iota^{\prime})_{*}(j^{\prime})^{*}
\mathcal{F}^{\bullet}
\end{equation*}
by the proper base change theorem (\cite[pp.223--224, VI, Corollary 2.3]{M})
(or by \cite[p.69, II, Theorem 3.2 (a)]{M} 
and \cite[p.71, II, Corollary 3.5 (a)]{M}).
So we have isomorphisms
\begin{equation*}
\HO^{s}_{\et}(
X, 
Rj_{*}j^{*}(\iota_{*}\mathcal{F}^{\bullet})
)  
\xrightarrow{\simeq}
\HO^{s}_{\et}(
Y, 
R(j^{\prime})_{*}(j^{\prime})^{*}\mathcal{F}^{\bullet}
)
\end{equation*}
and
\begin{equation*}
\HO^{s}_{Z}(
X_{\et}, \iota_{*}\mathcal{F}^{\bullet}
)
\xrightarrow{\simeq}
\HO^{s}_{Y\cap Z}(
Y_{\et}, \mathcal{F}^{\bullet}
)
\end{equation*}
for any $s$.
\end{proof}
\begin{thm}\upshape\label{exijS}
Let $\mathfrak{X}$ be a regular semistable family over the spectrum of a discrete valuation ring $B$
of mixed characteristic $(0, p)$,
$\iota: Y\to \mathfrak{X}$ the inclusion of the closed fiber of $\mathfrak{X}$
and $j: U\hookrightarrow \mathfrak{X}$ the open complement
$U:=\mathfrak{X}\setminus Y$.
Let 
$n\geq 0$,
$r>0$,
$s\geq 1$ 
be integers
and $x\in Y\cap \mathfrak{X}^{(s)}$.
Then the distinguished triangle
\begin{align*}
j_{!}\mu^{\otimes n}_{p^{r}}
\to
\mathfrak{T}_{r}(n)
\to
\iota_{*}\iota^{*}\mathfrak{T}_{r}(n)
\end{align*}
(cf. \cite[pp.75--76, II, Remark 3.13]{M} and 
\cite[p.537, Definition 4.2.4]{SaP}) induces an exact sequence
\begin{align*}
0
\to
\HO_{x}^{s}(
Y_{\et},
\lambda_{r}^{n}
)
\to
\HO_{x}^{n+s+1}(
\mathfrak{X}_{\et},
j_{!}\mu_{p^{r}}^{\otimes n}
)
\to
\HO_{x}^{n+s+1}(
\mathfrak{X}_{\et},
\mathfrak{T}_{r}(n)
)
\to 0
\end{align*}
and an isomorphism
\begin{equation*}
\HO_{x}^{n+t}(
\mathfrak{X}_{\et}, j_{!}\mu_{p^{r}}^{\otimes n}
)
\xrightarrow{\simeq}
\HO_{x}^{n+t}(
\mathfrak{X}_{\et},
\mathfrak{T}_{r}(n)
)
\end{equation*}
for $t\geq s+2$.
\end{thm}
\begin{proof}\upshape
By Theorem \ref{iW} and Lemma \ref{SAP},
it suffices to show that the homomorphism 
\begin{equation}\label{Sji}
\HO^{n+s}_{x}(
\mathfrak{X}_{\et},
j_{!}\mu_{p^{r}}^{\otimes n}
)
\to
\HO^{n+s}_{x}(
\mathfrak{X}_{\et},
\mathfrak{T}_{r}(n)
)
\end{equation}
is surjective.
We have a commutative diagram
\begin{equation*}
\xymatrix{
\iota_{*}Ri^{!}\mathfrak{T}_{r}(n)
\ar@{=}[r]\ar[d]
&
\iota_{*}R\iota^{!}\mathfrak{T}_{r}(n)  \ar[d]
\\
\mathfrak{T}_{r}(n)
\ar[r]
&
\iota_{*}\iota^{*}\mathfrak{T}_{r}(n)
}    
\end{equation*}
where the horizontal arrows are induced by
the unit of the adjunction
$\operatorname{id}\to \iota_{*}\iota^{*}$.
So we have a commutative diagram
\begin{equation*}
\xymatrix{
\tau_{\leq n+1}\iota_{*}R\iota^{!}\mathfrak{T}_{r}(n)
\ar[r]\ar[d] 
&
\iota_{*}\iota^{*}\mathfrak{T}_{r}(n) \ar@{=}[d]
\\
\mathfrak{T}_{r}(n) \ar[r]
&
\iota_{*}\iota^{*}\mathfrak{T}_{r}(n)
}    
\end{equation*}
and this commutative diagram induces
a morphism of distinguished triangles
\begin{equation*}
\xymatrix{
\tau_{\leq n+1}\iota_{*}Ri^{!}\mathfrak{T}_{r}(n)
\ar[r]\ar[d]
&
\iota_{*}\iota^{*}\mathfrak{T}_{r}(n)
\ar[r]\ar@{=}[d]
&
\iota_{*}\tau_{\leq n}\iota^{*}Rj_{*}\mu_{p^{r}}^{\otimes n}  
\ar[d]
\\
\mathfrak{T}_{r}(n)
\ar[r]
&
\iota_{*}\iota^{*}\mathfrak{T}_{r}(n)
\ar[r]
&
j_{!}\mu_{p^{r}}^{\otimes n}[+1].
}    
\end{equation*}
Hence we have a commutative diagram
\begin{equation*}
\xymatrix{
\HO^{n+s-1}_{x}(
Y_{\et},
\iota^{*}\tau_{\leq n}Rj_{*}\mu_{p^{r}}^{\otimes n}
)
\ar[r]\ar[d]
&
\HO^{n+s}_{x}(
\mathfrak{X}_{\et},
\tau_{\leq n+1}
\iota_{*}Ri^{!}\mathfrak{T}_{r}(n)
) \ar[d]
\\
\HO^{n+s}_{x}(
\mathfrak{X}_{\et},
j_{!}\mu_{p^{r}}^{\otimes n}
)
\ar[r]  
&
\HO^{n+s}_{x}(
\mathfrak{X}_{\et},
\mathfrak{T}_{r}(n)
)
}    
\end{equation*}
where the right arrow is an isomorphism 
by Definition \ref{DTT} and 
Theorem \ref{PETP}.
Since the homomorphism (\ref{SurYiTt}) is surjective by Proposition \ref{Otit}, 
the homomorphism (\ref{Sji}) is surjective by the above commutative diagram. 
This completes the proof.

\end{proof}
\begin{cor}\upshape\label{exmut}
Let the notations be the same as in Theorem \ref{exijS}.
Suppose that $B$ contains $p$-th roots of unity.   
Then we have an exact sequence
\begin{equation*}
0\to  
\HO^{s}_{x}(
Y_{\et}, \lambda^{n}_{1}
)
\to
\HO^{n+s+1}_{x}(
\mathfrak{X}_{\et}, \mathfrak{T}_{1}(n-1)
)
\to
\HO^{n+s+1}_{x}(
\mathfrak{X}_{\et}, \mathfrak{T}_{1}(n)
)
\to 
0
\end{equation*}
and an isomorphism
\begin{equation*}
\HO^{n+s+1}_{x}(
\mathfrak{X}_{\et},
\mathfrak{T}_{1}(n-1)
)    
\simeq 
\HO^{n+s+1}_{x}(
\mathfrak{X}_{\et},
\mathfrak{T}_{1}(n-k)
)
\end{equation*}
for $1\leq k\leq n$.
\end{cor}
\begin{proof}\upshape
Since $B$ contains $p$-th roots of unity, 
we have an isomorphism
\begin{equation*}
\mu_{p}^{\otimes (n-k)} 
\simeq
\mu_{p}^{\otimes n}
\end{equation*}
for $1\leq k\leq n$. Hence the statement follows from Theorem \ref{exijS}.
\end{proof}
\begin{rem}\upshape\label{Usup}
Let $\mathfrak{X}$ be
a regular semistable family over the spectrum of a discrete valuation ring $B$
of mixed characteristic $(0, p)$,
$\iota: Y\to \mathfrak{X}$ the inclusion of the closed fiber of $\mathfrak{X}$
and $j: U\hookrightarrow \mathfrak{X}$ the open complement
$\mathfrak{X}\setminus Y$.
Let $\mathcal{F}$ be a sheaf on $Y_{\et}$.
Then we have
\begin{equation*}
\HO^{s}_{x}(
\mathfrak{X}_{\et},
\iota_{*}\mathcal{F}
)    
=0
\end{equation*}
for $x\in U$ and $s\geq 0$.
So we have an isomorphism
\begin{equation*}
\HO^{s}_{x}(
\mathfrak{X}_{\et}, j_{!}\mu_{p^{r}}^{\otimes n}
)
\xrightarrow{\sim}
\HO^{s}_{x}(
\mathfrak{X}_{\et}, \mathfrak{T}_{r}(n)
)
\end{equation*}
for $x\in U$ and $s\geq 0$.
If $B$ contains $p$-th roots of unity, then we have isomorphisms
\begin{equation*}
\HO_{x}^{s}(
\mathfrak{X}_{\et},
\mathfrak{T}_{1}(n-1)
)
\xleftarrow{\sim}   
\HO_{x}^{s}(
\mathfrak{X}_{\et}, j_{!}\mu_{p}^{\otimes (n-1)}
)
\simeq
\HO_{x}^{s}(
\mathfrak{X}_{\et},
j_{!}\mu_{p}^{\otimes n}
)
\xrightarrow{\sim}
\HO^{s}_{x}(
\mathfrak{X}_{\et},
\mathfrak{T}_{1}(n)
)
\end{equation*}
for $x\in U$ and $s\geq 0$.
\end{rem}
By using Corollary \ref{exmut}, we prove Theorem \ref{Intpre}.
Before proving Theorem \ref{Intpre}, we prove the followings:
\begin{lem}\upshape\label{comdel}
Let $X$ be a scheme, 
$\iota: Y\to X$
a closed immersion of $X$
and 
$j: U:=X\setminus Y\to X$
the open immersion of $X$.
Let $\mathcal{F}^{\bullet}$ be a bounded below complex of sheaves of
abelian groups on $X_{\et}$.
Then the diagram
\begin{equation*}
\xymatrix{
\HO^{n}_{\et}(
X,
\mathcal{F}^{\bullet}
)
\ar[r]
\ar[d]
&
\HO^{n}_{\et}(
X,
Rj_{*}j^{*}\mathcal{F}^{\bullet}
)
\ar[d]^{\delta}
\\
\HO^{n}_{\et}(
X,
\iota_{*}\iota^{*}\mathcal{F}^{\bullet}
\ar[r]_-{\delta^{\prime}}
)
&
\HO^{n+1}_{\et}(
X,
\iota_{*}R\iota^{!}(
j_{!}j^{*}\mathcal{F}^{\bullet}
)
)
}    
\end{equation*}
is anti-commutative. Here 
$\delta$ and $\delta^{\prime}$ are the connecting homomorphisms 
in the long exact cohomology sequences associated to
\begin{equation*}
\cdots\to 
\iota_{*}R\iota^{!}(j_{!}j^{*}\mathcal{F}^{\bullet})
\to
j_{!}j^{*}\mathcal{F}^{\bullet}
\to
Rj_{*}j^{*}\mathcal{F}^{\bullet}
\to\cdots
\end{equation*}
and
\begin{equation*}
\cdots\to 
\iota_{*}R\iota^{!}(j_{!}j^{*}\mathcal{F}^{\bullet})
\to
\iota_{*}R\iota^{!}\mathcal{F}^{\bullet}
\to
\iota_{*}\iota^{*}\mathcal{F}^{\bullet}
\to\cdots
\end{equation*}
respectively.
\end{lem}
\begin{proof}\upshape
By \cite[Lemma 13.18.9]{SP}, there is a commutative diagram
%
\begin{equation*}
\xymatrix{
0
\ar[r]
&
j_{!}j^{*}\mathcal{F}^{\bullet}
\ar[r]\ar[d]
&
\mathcal{F}^{\bullet}
\ar[r]\ar[d]
&
\iota_{*}\iota^{*}\mathcal{F}^{\bullet}
\ar[r]\ar[d]
&
0 \\
0
\ar[r]
&
I_{1}^{\bullet}
\ar[r]
&
I_{2}^{\bullet}
\ar[r]
&
I_{3}^{\bullet}
\ar[r]
&
0
}    
\end{equation*}
%
where the vertical arrows are injective resolutions and the rows are short exact sequences
of complexes.
Then the statement follows from the definition of the connecting homomorphism.

The detail of the proof is as follows.
Since 
\begin{math}
R^{s}\iota_{*}\iota^{!}I^{n}_{t}=0
\end{math}
for $s>0$ and $t=1, 2, 3$,
we have exact sequences
\begin{equation*}
0
\to
\iota_{*}\iota^{!}I^{n}_{t}
\to
I^{n}_{t}
\to
j_{*}j^{*}I^{n}_{t}
\to
0
\end{equation*}
for $t=1, 2, 3$ and
\begin{equation*}
0\to   
\iota_{*}\iota^{!}I^{n}_{1}
\to   
\iota_{*}\iota^{!}I^{n}_{2}
\to   
\iota_{*}\iota^{!}I^{n}_{3}
\to 0.
\end{equation*}
Since the functors $\iota_{*}$, $\iota^{!}$ preserve injective objectives,
the above exact sequences are also split. Since we have an isomorphism $j^{*}\iota_{*}\simeq 0$, 
the morphisms
\begin{equation*}
\iota_{*}\iota^{!}I^{\bullet}_{3}
\to
I^{\bullet}_{3}
\end{equation*}
and
\begin{equation*}
j_{*}j^{*}I^{\bullet}_{1}
\to
j_{*}j^{*}I^{\bullet}_{2}
\end{equation*}
are quasi-isomorphisms.
Moreover, $j_{*}j^{*}I_{1}^{\bullet}$ and $j_{*}j^{*}I_{2}^{\bullet}$
are also injective objects because 
the functors $j_{*}$, $j^{*}$ preserve injective objects.
Consider the commutative diagrams
\footnotesize
\begin{equation*}
\xymatrix{
&
\Gamma(X,
I_{1}^{n}
)
\ar[r]^{b^{n}_{1}}\ar[d]_{f^{n}}
&
\Gamma(X, j_{*}j^{*}I_{1}^{n})
\ar[d]_{h^{n}}
\ar[r]
&
0
\\
\Gamma(X, \iota_{*}\iota^{!}I_{2}^{n})
\ar[r]^{a^{n}_{2}}\ar[d]_{e^{n}}
&
\Gamma(X, I_{2}^{n})
\ar[r]^{b^{n}_{2}}\ar[d]_{g^{n}}
\ar[r]
&
\Gamma(X, j_{*}j^{*}I_{2}^{n})
\ar[r]
&
0
\\
\Gamma(X, \iota_{*}\iota^{!}I_{3}^{n})
\ar[r]^{a^{n}_{3}}
&
\Gamma(X, I_{3}^{n})
}    
\end{equation*}
\normalsize
where the sequences are exact. Then $a^{\bullet}_{3}$ and $h^{\bullet}$ are quasi-isomorphisms 
by \cite[p.41, Lemma 4.5]{H}.

Let
\begin{equation*}
d^{k}[\mathcal{A}^{\bullet}]: 
\Gamma(X, \mathcal{A}^{k})
\to
\Gamma(X, \mathcal{A}^{k+1})
\end{equation*}
be a boundary operator of
\begin{math}
\Gamma(X, \mathcal{A}^{\bullet})    
\end{math}
for an integer $k$ and a bounded
below complex $\mathcal{A}^{\bullet}$
of sheaves of abelian groups on $X_{\et}$.
Let $\tilde{x}$ be an element of
\begin{equation*}
\HO^{n}_{\et}(X, \mathcal{F}^{\bullet})
=
\operatorname{ker}\left(
d^{n}[I_{2}^{\bullet}]
\right)
/
\operatorname{im}\left(
d^{n-1}[I_{2}^{\bullet}]
\right)
\end{equation*}
and $x$ be an element of $\operatorname{ker}(d^{n}[I_{2}^{\bullet}])$
such that
\begin{math}
\tilde{x}=x+\operatorname{im}(d^{n-1}[I^{\bullet}_{2}]).    
\end{math}
Since $h^{\bullet}$ is a quasi-isomorphism and
$b^{n}_{1}$ is surjective,
there exists a 
\begin{math}
y\in
\Gamma(
X,
I^{n}_{1}
)
\end{math}
such that
\begin{equation*}
h^{n}\circ b_{1}^{n}(y)
-
b_{2}^{n}(x)
\in 
\operatorname{im}\left(
d^{n-1}[j_{*}j^{*}I^{\bullet}_{2}]
\right).
\end{equation*}
Moreover,
\begin{math}
(d^{n}[I_{1}^{\bullet}])(y)
\in \operatorname{ker}(d^{n+1}[\iota_{*}\iota^{!}(I^{\bullet}_{1})])
\end{math}
and then the image under the composite
\begin{equation*}
\HO^{n}_{\et}(
X, \mathcal{F}^{\bullet}
)
\to
\HO^{n}_{\et}(
X, Rj_{*}j^{*}\mathcal{F}^{\bullet}
)
\xrightarrow{\delta}
\HO^{n+1}_{\et}(
X, \iota_{*}R\iota^{!}(j_{!}j^{*}\mathcal{F}^{\bullet})
)
\end{equation*}
of a point $\tilde{x}$ is
\begin{equation*}
(d^{n}[I_{1}^{\bullet}])(y)
+
\operatorname{im}(
d^{n}[\iota_{*}\iota^{!}(I^{\bullet}_{1})]
).
\end{equation*}

On the other hand, there exists a 
$z\in \Gamma(X, \iota_{*}\iota^{!}I^{n}_{2})$
such that
\begin{equation}\label{yz}
a_{2}^{n}(z)
-   
\left(x-f^{n}(y)\right)
\in 
\operatorname{im}\left(
d^{n-1}[I^{\bullet}_{2}]
\right).
\end{equation}
Then
\begin{math}
(d^{n}[\iota_{*}\iota^{!}(I^{\bullet}_{2})])(z)
\in
\operatorname{ker}\left(
d^{n+1}[\iota_{*}\iota^{!}(I^{\bullet}_{1})]
\right)
\end{math}
and
the image under the composite
\begin{equation*}
\HO^{n}_{\et}(
X, \mathcal{F}^{\bullet}
)    
\to
\HO^{n}_{\et}(
X, \iota_{*}\iota^{*}\mathcal{F}^{\bullet}
)
\xrightarrow{\delta^{\prime}}
\HO^{n+1}_{\et}(
X,
\iota_{*}\iota^{!}(j_{!}j^{*}\mathcal{F}^{\bullet})
)
\end{equation*}
of a point $\bar{x}$ is
\begin{equation*}
(d^{n}[\iota_{*}\iota^{!}(I^{\bullet}_{2})])(z)
+
\operatorname{im}(
d^{n}[\iota_{*}\iota^{!}(I^{\bullet}_{1})]
).
\end{equation*}
Since $x\in \operatorname{ker}(d^{n}[I^{\bullet}_{2}])$,
we have
\begin{equation*}
a^{n+1}_{2}\circ(d^{n}[\iota_{*}\iota^{!}(I^{\bullet}_{2})])(z)
=
-f^{n+1}\circ(d^{n}[I^{\bullet}_{1}])(y)
\end{equation*}
by the relation (\ref{yz}). So the statement follows.
\end{proof}
Let $B$ be a discrete valuation ring of mixed characteristic $(0, p)$,
$\pi$ a prime element of $B$ and $\mathfrak{X}$ a smooth scheme over $\spec(B)$. 
Let $R$ be the local ring $\mathcal{O}_{\mathfrak{X}, x}$ of $\mathfrak{X}$ at a point $x$
or the henselization of $\mathcal{O}_{\mathfrak{X}, x}$.

Let $U=\spec(R)\setminus\spec(R/(\pi))$ and $j: U\to \spec(R)$ 
the inclusion of the generic fiber of $\spec(R)$.
Put
\begin{align*}
C^{s}=
\displaystyle\bigoplus_{x\in \spec(R/(\pi))^{(s-1)}}
\HO^{s}_{x}(
(R/(\pi))_{\et},
\lambda^{n}_{r}
),
&
&D^{s}=
\displaystyle\bigoplus_{x\in \spec(R)^{(s)}}
\HO^{n+s+1}_{x}(
R_{\et},
j_{!}\mu_{P^{r}}^{\otimes n}
)
\end{align*}
and
\begin{equation*}
E^{s}
=
\displaystyle\bigoplus_{x\in \spec(R)^{(s)}}
\HO^{n+s+1}_{x}(
R_{\et},
\mathfrak{T}_{r}(n)
)
\end{equation*}
for integers $s\geq 0$.
By Lemma \ref{spsqc} and Theorem \ref{exijS}, 
there is an exact sequence of cochain complexes
\begin{equation*}
0\to
C^{\bullet}
\to
D^{\bullet}
\to
E^{\bullet}
\to
0
\end{equation*}
and this exact sequence of cochain complexes induces the homomorphism
\begin{align*}
&\operatorname{ker}\left(
E^{0}\to E^{1}
\right)
\to
\operatorname{ker}\left(
C^{1}\to C^{2}
\right).
\end{align*}
So this homomorphism induces the homomorphism
\begin{equation}\label{frsp}
\HO^{n+1}_{\et}(
R, \mathfrak{T}_{r}(n)
)    
\to
\HO^{1}_{\et}(
R/(\pi), \lambda_{r}^{n}
)
\end{equation}
by \cite[p.42, Proposition 3.4]{Sak},
\cite[p.51, Theorem 4.6]{Sak} and \cite[p.600, Theorem 4.1]{Sh}.
\begin{lem}\label{Cornat}\upshape
Let the notations be the same as above.
Then the homomorphism (\ref{frsp}) is equal up to sign the homomorphism 
\begin{equation}\label{natred}
\HO^{n+1}_{\et}(
R, \mathfrak{T}_{r}(n)
)    
\to
\HO^{1}_{\et}(
R/(\pi), \lambda_{r}^{n}
)    
\end{equation}
which is induced by the unit of adjunction
\begin{math}
\mathfrak{T}_{r}(n)
\to
\iota_{*}\iota^{*}\mathfrak{T}_{r}(n)
\end{math}
and the morphism (\ref{FITFL}).
\end{lem}
\begin{proof}
Consider the diagram 
\begin{equation}\label{comrd}
\xymatrix{
\HO^{n+1}_{\et}(
R,
\mathfrak{T}_{r}(n)
)
\ar[r]\ar[d]
&
\HO^{1}_{\et}(
R/(\pi), 
\lambda_{r}^{n}
) \ar[d] 
\\
\HO^{n+1}_{\et}(
R_{(\pi)},
\mathfrak{T}_{r}(n)
)
\ar[r]
&
\HO^{1}_{\et}(
\kappa((\pi)),
\lambda_{r}^{n}
)
}    
\end{equation}
where the vertical arrows are the natural map and the horizontal arrows  are the homomorphisms (\ref{frsp}).

Then the diagram (\ref{comrd}) is commutative.
If the horizontal arrows in the diagram (\ref{comrd}) are the homomorphisms (\ref{natred}), the diagram (\ref{comrd}) is also commutative. Moreover, the right map in the diagram (\ref{comrd}) is injective by
\cite[p.600, Theorem 4.1]{Sh}. So it suffices to show the statement in the case where $R$ is a discrete valuation ring of
mixed characteristic case $(0, p)$.
Hence the statement follows from Lemma \ref{comdel}.
\end{proof}
\begin{thm}\upshape\label{premain}
Let $B$ be a discrete valuation ring of mixed characteristic
$(0, p)$ and $\pi$ a prime element of $B$.
Let $\mathfrak{X}$ be a smooth scheme over $\spec(B)$
and 
$R=\mathcal{O}_{\mathfrak{X}, x}^{h}$ the henselization
of the local ring of $\mathfrak{X}$ at $x\in \mathfrak{X}$.
Suppose that $B$ contains $p$-th roots of unity.
Then the Cousin complex
%
\begin{align*}
0
\to
& \HO^{n+1}_{\et}(
R, \mathbb{Z}/p(n-r)
)
\to
\bigoplus_{x\in \spec(R)^{(0)}}
\HO_{x}^{n+1}(R_{\et}, \mathbb{Z}/p(n-r))
\to \\
& \bigoplus_{x\in \spec(R)^{(1)}}
\HO_{x}^{n+2}(R_{\et}, \mathbb{Z}/p(n-r))
\to
\cdots
\end{align*}

%
is exact
for $r=0, 1$.
\end{thm}
\begin{proof}\upshape
Let 
$\iota: Y=\spec(R/(\pi))\to \spec(R)$
be the inclusion of the closed fiber of
$\spec(R)$ and 
$j: U\hookrightarrow \spec(R)$
the open complement $\spec(R)\setminus Y$.
Since $X$ is a smooth scheme over $\spec(B)$, we have an isomorphism
\begin{equation*}
\mathbb{Z}/p^{r}(n)^{\spec(R)}_{\et} 
\simeq
\mathfrak{T}_{r}(n)
\end{equation*}
in $D^{+}(\spec(R)_{\et}, \mathbb{Z}/p^{r}\mathbb{Z})$ 
by \cite[p.524, Conjecture 1.4.1]{SaP}. 
Here 
$D^{+}(\spec(R)_{\et}, \mathbb{Z}/m\mathbb{Z})$ 
is the derived category of 
bounded complexes of \'{e}tale 
$\mathbb{Z}/m\mathbb{Z}$-sheaves on $X$. 
Moreover, we have isomorphisms
\begin{equation*}
\mathbb{Z}/p^{r}(n)_{\et}^{U}
\simeq
\mu_{p^{r}}^{\otimes n}
\end{equation*}
in $D^{+}(U_{\et}, \mathbb{Z}/p^{r}\mathbb{Z})$ and
\begin{equation*}
\mathbb{Z}/p^{r}(n)^{Z}_{\et}
\simeq
\nu^{n}_{r}[-n]
\end{equation*}
in $D^{+}(Y_{\et}, \mathbb{Z}/p^{r}\mathbb{Z})$ 
by \cite[Theorem 1.5]{Ge-L2} and \cite[p.787, \S 5, (12)]{Ge}.
Here $\mu_{p^{r}}$ is the sheaf of $p^{r}$-th roots of unity.

It suffices to show the statement in the case where 
$Y\neq\emptyset$. 
Let $R^{\prime}$ be an equidimensional Noetherian local ring.
Then we have the spectral sequence
\begin{equation}\label{gsp}
E^{s, t}_{1}(R^{\prime}, n)=
\bigoplus_{x\in\spec(R^{\prime})^{(s)}}
\HO^{s+t}_{x}(
R^{\prime}_{\et}, 
\mathbb{Z}/p(n)
)
\Rightarrow
E^{s+t}(
R^{\prime}, n
)
=
\HO^{s+t}_{\et}(
R^{\prime}, 
\mathbb{Z}/p(n)
)
\end{equation}
by Theorem \ref{sp}. We prove the statement by using the spectral sequence (\ref{gsp}).  
By \cite[p.52, Corollary 4.7]{Sak}, we have
\begin{equation*}
E^{0, n+1}_{2}(R, n)
= 
E^{n+1}(R, n)
\end{equation*}
for any non-negative integer $n$. 
Let $R^{h}_{(\pi)}$ be the henselization of the local ring $R_{(\pi)}$.
Since we have an isomorphism
\begin{equation*}
\HO^{n+v}_{\et}(
R^{h}_{(\pi)}, 
j_{!}\mu_{p}^{\otimes n}
)  
\simeq 0
\end{equation*}
by \cite[p.777, The proof of Proposition 2.2.b)]{Ge},
we have an isomorphism
\begin{equation*}
\HO_{\et}^{n+v}(
k(R^{h}_{(\pi)}),
\mu_{p}^{\otimes n}
)    
\simeq
\HO_{(\pi)}^{n+v+1}(
(R^{h}_{(\pi)})_{\et},
j_{!}\mu_{p}^{\otimes n}
).
\end{equation*}
So we have
\begin{equation*}
E^{0, n+v}_{2}(R, n)
=0
\end{equation*}
for $v\geq 2$
by \cite[p.61, Proposition 6.1]{Sak} and Theorem \ref{exijS}. 
Since we have an isomorphism
\begin{equation*}
\HO^{n+v}_{\et}(
R,
\mathbb{Z}/p(n)
)    
\simeq
\HO^{n+v}_{\et}(
R/(\pi),
\iota^{*}(\mathbb{Z}/p(n))
) 
\end{equation*}
by \cite[p.777, The proof of Proposition 2.2.b)]{Ge}, we have an isomorphism
\begin{equation}\label{isoRz}
\HO^{n+v}_{\et}(
R,
\mathbb{Z}/p(n)
) 
\simeq 0
\end{equation}
for $v\geq 2$ by \cite[Expos\'{e} X, Th\'{e}or\`{e}me 5.1]{SGA4} and
the spectral sequence
\begin{equation*}
E^{s, t}_{2}
=\HO^{s}_{\et}(
R/(\pi),
\mathcal{H}^{t}(
\iota^{*}\mathbb{Z}/p(n)
)
)
\Rightarrow
\HO^{s+t}_{\et}(
R/(\pi),
\iota^{*}\mathbb{Z}/p(n)
).
\end{equation*}
Consequently, we have
\begin{equation*}
E^{0, n+v}_{2}(R, n)=E^{n+v}(R, n)
=0
\end{equation*}
for $v\geq 2$. So it suffices to show that
\begin{equation}\label{1E2v}
E^{s, n+v}_{2}(R, n)=0    
\end{equation}
for $s\geq 1$ and $v=1, 2$.
We prove the equation (\ref{1E2v}) by induction on $s$.

First we show the equation (\ref{1E2v}) in the case where $s=1$.
By the spectral sequence (\ref{gsp}) and the isomorphism (\ref{isoRz}), 
we have
\begin{equation*}
E^{1, n+1}_{\infty}(R, n)
=
E^{n+2}(R, n)
=0.
\end{equation*}
By Proposition \ref{PTT}, we have
\begin{equation*}
E^{1+r, n+2-r}_{r}(R, n)=0    
\end{equation*}
for $r\geq 2$. Hence we have
\begin{equation}\label{1En+1}
E^{1, n+1}_{2}(R, n)
=
E^{1, n+1}_{3}(R, n)
=
\cdots 
=
E^{1, n+1}_{\infty}(R, n)
=0.
\end{equation}
By Lemma \ref{spsqc}, Corollary \ref{exmut} and Remark \ref{Usup}, we have an exact sequence
\begin{equation*}
E^{0, n+1}_{2}(R, n)
\to
E^{0, n+1}_{2}(R/(\pi), n)
\to
E^{1, n+1}_{2}(R, n-1)
\to
E^{1, n+1}_{2}(R, n).
\end{equation*}
Moreover, we have
\begin{equation*}
E^{0, n+1}_{2}(R/(\pi), n)
=
E^{n+1}(R/(\pi), n)
\end{equation*}
by \cite[p.600, Theorem 4.1]{Sh}.
So the map
\begin{equation*}
E^{0, n+1}_{2}(R, n)
\to
E^{0, n+1}_{2}(R/(\pi), n)
\end{equation*}
is surjective by Lemma \ref{Cornat} and \cite[p.59, Theorem 5.6]{Sak}. 
Hence we have
\begin{equation*}
E^{1, n+1}_{2}(R, n-1)
=0
\end{equation*}
by the equation (\ref{1En+1}). Therefore we have the equation (\ref{1E2v}) in the case where $s=1$.

Next we assume that the equation (\ref{1E2v}) holds in the case where
$1\leq s\leq s^{\prime}$. By the spectral sequence (\ref{gsp}) and the isomorphism (\ref{isoRz}), 
we have
\begin{equation*}
E^{s^{\prime}+1, n+1}_{\infty}(R, n)
=
E^{s^{\prime}+n+2}(R, n)
=
0.
\end{equation*}
By the assumption of induction, we have
\begin{equation*}
E_{r}^{s^{\prime}-r+1, n+r}(R, n)
=
0
\end{equation*}
for $r\geq 2$.
Moreover, we have
\begin{equation*}
E_{r}^{s^{\prime}+r+1, n-r+2}(R, n)=0    
\end{equation*}
for $r\geq 2$ by Proposition \ref{PTT}. So we have
\begin{equation}\label{IAE}
E_{2}^{s^{\prime}+1, n+1}(R, n)
=
E_{3}^{s^{\prime}+1, n+1}(R, n)
=
\cdots
=
E^{s^{\prime}+1, n+1}_{\infty}(R, n)
=
0.
\end{equation}
Moreover, we have an exact sequence
\begin{equation*}
E_{2}^{s^{\prime}, n+1}(R/(\pi), n)
\to
E_{2}^{s^{\prime}+1, n+1}(R, n-1)
\to
E_{2}^{s^{\prime}+1, n+1}(R, n)
\end{equation*}
by Lemma \ref{spsqc}, Corollary \ref{exmut} and Remark \ref{Usup}.
Hence we have
\begin{equation*}
E_{2}^{s^{\prime}+1, n+1}
(R, n-1)=0
\end{equation*}
by \cite[p.600, Theorem 4.1]{Sh} and the equation (\ref{IAE}). 
Therefore the equation (\ref{1E2v}) holds in the case where $s=s^{\prime}+1$.
This completes the proof.
\end{proof}

\section{The surjectivity of the map between the mod $p$ \'{e}tale motivic cohomology groups}\label{SKH}
In this section, we prepare the proof of Theorem \ref{IntGCL}.

Let $A$ be a regular local ring of characteristic $p>0$.
Let 
\begin{math}
\Omega^{n}_{A}
\end{math}
be the exterior algebra 
\begin{math}
\Omega^{n}_{\spec (A)/\mathbb{Z}}
\end{math}
over
$\mathcal{O}_{\spec (A)}$
of the sheaf
\begin{math}
\Omega^{1}_{\spec (A)/\mathbb{Z}}
\end{math}
of the absolute differentials on
$\spec (A)$ and define the subsheaf 
$\mathfrak{B}^{n}$
of $\Omega^{n}_{A}$ by
\begin{equation*}
\mathfrak{B}^{n}
=
\operatorname{Im}\left(
d:\Omega^{n-1}_{A}
\to
\Omega^{n}_{A}
\right).
\end{equation*}
Then we have
\begin{equation*}
\HO^{0}_{\et}
(A, 
\Omega^{n}_{A}
/\mathfrak{B}^{n}
)
=
\Omega_{A}^{n}/\operatorname{B}^{n}_{A}
\end{equation*}
by \cite[Lemma 5.2]{Sak}. 
Here $\operatorname{B}^{n}_{A}$ denotes
$\mathfrak{B}^{n}(\spec (A))$.
So we have a surjective homomorphism
\begin{equation*}
\delta_{n}:
\Omega^{n}_{A}
/\operatorname{B}^{n}_{A}
\to
\HO_{\et}^{1}
(
A, 
\Omega_{A, \log}
)
\end{equation*}
by the exact sequence
\begin{equation*}
0\to\Omega^{n}_{A, \log}
\to\Omega^{n}_{A}
\to\Omega^{n}_{A}/\mathfrak{B}^{n}
\to 0
\end{equation*}
(cf. \cite[p.576, Proposition 2.8]{Sh}). Here 
$\Omega^{n}_{A, \log}$ denotes $W_{1}\Omega_{\spec(A), \log}$.
Then we have the following:
\begin{lem}\upshape\label{Cpf}
Let $A$ be a regular local ring of characteristic $p>0$. 
Then the diagram
\begin{equation*}
\xymatrix{
A\otimes (A^{*})^{\otimes n}
\ar[r]^-{\gamma_{n}}
\ar[d]_{
\delta_{0}\otimes\psi_{n}
} 
&
\Omega_{A}^{n}/\operatorname{B}^{n}_{A}
\ar[d]^-{\delta_{n}}  \\
\HO^{1}_{\et}(A, \mathbb{Z}/p)
\otimes
\HO^{0}_{\et}(A, \Omega^{n}_{A, \log})
\ar[r]_-\cup &
\HO^{1}_{\et}(A, \Omega^{n}_{A, \log})
}    
\end{equation*}
is commutative.
Here 
we write the homomorphisms
\begin{equation*}
\gamma_{n}:  
A\otimes (A^{*})^{\otimes n}
\to
\Omega_{A}^{n}/\operatorname{B}^{n}_{A};~~ 
\gamma_{n}(
x\otimes y_{1}\otimes\cdots
\otimes y_{n}
)
=
x\frac{dy_{1}}{y_{1}}
\wedge\cdots\wedge
\frac{dy_{n}}{y_{n}}
\bmod \operatorname{B}^{n}_{A}
\end{equation*}
and
\begin{equation*}
\psi_{n}: (A^{*})^{\otimes n}\to 
\HO^{0}_{\et}(
A,
\Omega^{n}_{A, \log}
); ~~
\psi_{n}(
y_{1}
\otimes\cdots\otimes
y_{n}
)
=
\frac{dy_{1}}{y_{1}}
\wedge\cdots\wedge
\frac{dy_{n}}{y_{n}}.
\end{equation*}
\end{lem}
\begin{proof}\upshape
Since the homomorphism
\begin{equation*}
\HO^{1}_{\et}(A, \Omega^{n}_{A, \log})
\to
\HO^{1}_{\et}(k(A), \Omega^{n}_{A, \log})
\end{equation*}
is injective, 
the proof of the statement is reduced to
the case where $A$ is a field.

Suppose that $A$ is a field. 
By \cite[p.53, II, Theorem 1.9]{M},
\'{e}tale cohomology of a field
coincides with Galois cohomology.
Hence the statement follows from the definition of
$\delta_{n}$ (cf. \cite[p.27, I, (1.3.2) Theorem]{N}) and
the definition of the cup product 
(cf. \cite[p.37, I, \S 4]{N}).
\end{proof}
Let $\mathfrak{X}$ be a regular semistable family over the spectrum of
a discrete valuation ring of mixed characteristic $(0, p)$
and $\iota: Y\to \mathfrak{X}$ the inclusion of
the closed fiber. Let $r>0$ be an integer.
Let $\lambda^{n}_{Y, r}$ be the sheaf on $Y_{\et}$
which is defined in
\cite[p.726, Definition 3.1.1]{SaL}, i.e.,
\begin{equation*}
\lambda^{n}_{Y, r}
=
\operatorname{Im}\left(
(\mathbb{G}_{m, Z})^{\otimes n}
\to
\bigoplus_{y\in Y^{(0)}}
(i_{y})_{*}
W_{r}\Omega^{n}_{\spec(\kappa(y)), \log}
\right)
\end{equation*}
and
$\mathfrak{T}_{r}(n)$ be the $p$-adic Tate twist (cf. \cite[p.537, Definition 4.2.4]{SaP}).
By \cite[p.537, (4.2.5) in Definition 4.2.4]{SaP} and \cite[p.533, Theorem 3.4.2]{SaP}, 
we define the homomorphism
\begin{equation}\label{PrTN0}
\iota^{*}: \iota^{*}\mathfrak{T}_{r}(n)
\to
\lambda^{n}_{r}[-n]
\end{equation}
as the composite map of \'{e}tale sheaves
\begin{equation*}
\iota^{*}\mathfrak{T}_{r}(n)
\to 
\iota^{*}\tau_{\geq n}\mathfrak{T}_{r}(n)
\simeq
\mathcal{H}^{n}(
\iota^{*}\mathfrak{T}_{r}(n))[-n]
\simeq
FM^{n}_{r}[-n]
\to
\lambda^{n}_{r}[-n]
\end{equation*}
where the last map
sends
a symbol $\{x_{1}, \cdots, x_{n}\}$ 
to
\begin{math}
\frac{d\bar{x_{1}}}{\bar{x_{1}}}
\wedge\cdots\wedge
\frac{d\bar{x_{n}}}{\bar{x_{n}}}.
\end{math}
Here $j:U\to \mathfrak{X}$ is the inclusion of the generic fiber
and
$FM^{n}_{r}$ is the \'{e}tale subsheaf of $\iota^{*}R^{n}j_{*}\mu_{p^{r}}^{\otimes n}$
which is defined in \cite[\S 3.4]{SaP}. Then we have the following:
\begin{lem}\upshape\label{PullC}
Let the notations be the same as above.
Then the diagram
\begin{align*}
\xymatrix{
\iota^{*}\mathfrak{T}_{r}(n)\otimes \iota^{*}\mathfrak{T}_{r}(n^{\prime}) 
\ar[r]^-{\cup}\ar[d]
& \iota^{*}\mathfrak{T}_{r}(n+n^{\prime}) 
\ar[d]  \\
\lambda^{n}_{r}[-n]\otimes\lambda^{n^{\prime}}_{r}[-n^{\prime}]  
\ar[r]_-{\cup}
& \lambda^{n+n^{\prime}}_{r}[-(n+n^{\prime})]
}    
\end{align*}
is commutative. Here the upper map is
the cup product which is defined in 
\cite[p.538, Proposition 4.2.6]{SaP}.
\end{lem}
\begin{proof}\upshape
By the definition of 
$\iota^{*}: \iota^{*}\mathfrak{T}_{r}(n)\to\lambda^{n}_{r}[-n]$, 
the statement directly follows.
\end{proof}
\begin{lem}\upshape\label{mun}
Let $R$ be a local ring of mixed characteristic
$(0, p)$. Let $a\in R$ and
$P(X)$
be a monic polynomial
\begin{math}
X^{p}-X-a\in R[X].
\end{math}
Then we have
\begin{equation*}
(P, P^{\prime})=R[X]   
\end{equation*}
where $P^{\prime}(X)$
is the formal derivative of
$P(X)$.
\end{lem}
\begin{proof}\upshape
We have
\begin{equation*}
pP-XP^{\prime}
=(1-p)X-pa
~~
\textrm{and}
~~
X-p(1-p)^{-1}a\in (P, P^{\prime}).
\end{equation*}
Put 
\begin{math}
b=p(1-p)^{-1}a\in R.
\end{math}
Then 
\begin{equation*}
pb^{p-1}-1
\in(X-b, P^{\prime})
~~
\textrm{and}
~~
pb^{p-1}-1
\in(P, P^{\prime}).
\end{equation*}
Let $\mathfrak{m}$ be the maximal ideal of $R$.
Since $R$ is local and 
$p\in\mathfrak{m}$,
\begin{equation*}
pb^{p-1}-1\not\in
\mathfrak{m}
~~
\textrm{and}
~~
pb^{p-1}-1\in R^{*}.
\end{equation*}
This completes the proof.
\end{proof}
\begin{lem}\upshape\label{surmp1}
Let $R$ be a regular local ring
of mixed characteristic $(0, p)$.
Let $\mathfrak{p}$ be a prime ideal of $R$
with $p\in \mathfrak{p}$.
Then the homomorphism
\begin{equation*}
\HO^{1}_{\et}(R, \mathbb{Z}/p)
\to
\HO^{1}_{\et}(R/\mathfrak{p}, \mathbb{Z}/p)
\end{equation*}
is surjective.
\end{lem}
\begin{proof}\upshape
By \cite[p.103, III, Lemma 2.15]{M} and
\cite[p.114, III, Remark 3.8]{M},
we have
\begin{equation*}
\HO^{i}_{\et}(R/\mathfrak{p},
\mathbb{G}_{a})
=0
\end{equation*}
for $i\geq 1$.
So
an element of 
\begin{math}
\HO^{1}_{\et}(
R/\mathfrak{p},
\mathbb{Z}/p
)    
\end{math}
corresponds to
\begin{math}
(R/\mathfrak{p})[X]/(X^{p}-X-c)
\end{math}
for $c\in R/\mathfrak{p}$
by the Artin-Schreier theorey (cf. \cite[p.67]{M}).
Hence the statement follows from Lemma \ref{mun}.
\end{proof}
Let $A$ be a discrete valuation ring of mixed characteristic $(0, p)$,
$\mathfrak{X}$ a smooth scheme over $\spec(A)$ and 
$\iota: Y\to \mathfrak{X}$ the inclusion of the closed fiber.
Then we define the morphism
\begin{equation}\label{PrTN}
\iota^{*}: \mathfrak{T}_{1}(n)
\to  \iota_{*}\lambda_{1}^{n}[-n]  
\end{equation}
as the right adjunct of
%
%
the homomorphism (\ref{PrTN0}) for
$r=1$.
Then we have the following:
\begin{prop}\upshape\label{Bsur}
Let $R$ be a local ring of
a smooth algebra over a discrete valuation
ring $A$
of mixed characteristic $(0, p)$. 
Let $\pi$ be a prime element of $A$ and
$\mathfrak{p}=(\pi)$.
Then the homomorphism
\begin{equation}\label{irp}
\HO^{n+1}_{\et}(R, \mathfrak{T}_{1}(n))
\to
\HO^{1}_{\et}(R/\mathfrak{p}, \lambda_{1}^{n})
\end{equation}
is surjective 
where the homomorphism (\ref{irp}) is induced by the morphism (\ref{PrTN}).
\end{prop}
\begin{proof}\upshape
We have a commutative diagram
\begin{equation*}
\xymatrix{
\HO^{1}_{\et}(R, \mathfrak{T}_{1}(0))
\otimes
\HO^{n}_{\et}(R, \mathfrak{T}_{1}(n))
\ar[rr]^-{\cup}\ar[d]
&&
\HO^{n+1}_{\et}(R, \mathfrak{T}_{1}(n))
\ar[d] \\
\HO^{1}_{\et}(R/\mathfrak{p}, \lambda_{1}^{0})
\otimes
\HO^{0}_{\et}(R/\mathfrak{p}, \lambda_{1}^{n})
\ar[rr]_-{\cup}
&&
\HO^{1}_{\et}(R/\mathfrak{p}, \lambda_{1}^{n})
}
\end{equation*}
by Lemma \ref{PullC}. Moreover, the lower map is surjective by 
\cite[p.122, Lemma (4.2)]{B-K} and Lemma \ref{Cpf}. 
So it suffices to show that
the left map is surjective. 
Since 
the homomorphism
\begin{equation*}
\overbrace{
\HO^{0}_{\et}(R/\mathfrak{p}, \lambda_{1}^{1})
\otimes 
\cdots
\otimes
\HO^{0}_{\et}(R/\mathfrak{p}, \lambda_{1}^{1})
}^{n\textup{-times}}
\to
\HO^{0}_{\et}(R/\mathfrak{p}, \lambda_{1}^{n})
\end{equation*}
and
the homomorphism
\begin{equation*}
R^{*}\to (R/\mathfrak{p})^{*}    
\end{equation*}
are surjective,
the  homomorphism (\ref{irp})
%
%
is also surjective.
Therefore the statement follows from 
Lemma \ref{surmp1}.
\end{proof}

Let $A$ be a discrete valuation ring of mixed characteristic 
$(0, p)$ and $\pi$ a prime element of $A$.
Let $\mathfrak{X}$ be a smooth scheme over $\spec(A)$, $j: U\to \mathfrak{X}$ 
the inclusion of the generic fiber of $\mathfrak{X}$ 
and 
$\iota: Y\to \mathfrak{X}$ the inclusion of the closed fiber of $\mathfrak{X}$. 
Then we define the following homomorphisms:

\begin{defi}\upshape
Let the notations be the same as above.
Then we define the homomorphism 
\begin{equation}\label{pr}
\HO^{n+1}_{\et}(
\mathfrak{X}, 
\mathbb{Z}/p(n))
\to
\HO^{n+2}_{\et}(
\mathfrak{X}, 
\tau_{\leq n+1}Rj_{*}\mathbb{Z}/p(n+1))
\end{equation}
by 
%
\begin{align*}
\xymatrix@C=36pt@R=2.8pt
{
\HO^{n+1}_{\et}\left(
\mathfrak{X}, 
\mathbb{Z}/p(n)
\right)  
\ar[r] 
& \HO^{n+2}_{\et}\left(
\mathfrak{X}, 
\tau_{\leq n+1}Rj_{*}
\mathbb{Z}/p(n+1)
\right)
\\
a \ar@{(-}[u] \ar@{|->}[r] & \bar{\pi}\cup \bar{a} \ar@{(-}[u]
}
\end{align*}
%
where
\begin{equation*}
\bar{\pi}=\operatorname{Im}\left(
\HO^{1}_{\et}(U, \mathbb{G}_{m}[-1])
\to
\HO^{1}_{\et}(
\mathfrak{X}, 
\tau_{\leq 1}Rj_{*}\mathbb{Z}/p(1))
\right)(\pi)
\end{equation*}
and
\begin{equation*}
\bar{a}=\operatorname{Im}\left(
\HO^{n+1}_{\et}(
\mathfrak{X}, 
\mathbb{Z}/p(n))
\to
\HO^{n+1}_{\et}(
\mathfrak{X}, 
\tau_{\leq n}Rj_{*}\mathbb{Z}/p(n))
\right)(a).    
\end{equation*}
Moreover, we define the homomorphism
\begin{equation}\label{v}
\HO^{n+2}_{\et}(
\mathfrak{X}, 
\tau_{\leq n+1}Rj_{*}\mathbb{Z}/p(n+1))
\to
\HO^{n+1}_{\et}(Y, \mathbb{Z}/p(n))
\end{equation}
which is induced by the map
\begin{equation*}
\tau_{\leq n+1}Rj_{*}\mu_{p}^{\otimes n+1}
\to
\iota_{*}(\nu_{1}^{n}[-n-1])
\end{equation*}
(cf. Definition \ref{DTT}, \cite[p.537, (4.2.1)]{SaP}). 
\end{defi}
Then we have the following lemma:

\begin{lem}\upshape\label{compH}
Let $A$ be a discrete valuation ring
of mixed characteristic $(0, p)$
and $k$ the residue field of $A$. 
Let $\iota: \spec(k)\to \spec(A)$ be the closed immersion
and $j: \spec(k(A))\to\spec(A)$ the open immersion.

Then
the composite of (\ref{pr})
and (\ref{v})
equals
the homomorphism
\begin{equation}\label{NAkp}
\HO^{n+1}_{\et}(A, \mathbb{Z}/p(n))
\to 
\HO^{n+1}_{\et}(k, \mathbb{Z}/p(n))    
\end{equation}
which is induced by the morphism (\ref{PrTN}).
\end{lem}
\begin{proof}\upshape
Let $\mathfrak{m}$ be the maximal ideal of $\spec(A)$ and
$\bar{\mathfrak{m}}$ a geometric point of $\spec(A)$ such that
$\kappa(\bar{\mathfrak{m}})$ is the separable closure of
$k=\kappa(\mathfrak{m})$. Let
\begin{math}
(\iota^{*}R^{n}j_{*}\mu_{p}^{\otimes n})_{\bar{\mathfrak{m}}}
\end{math}
be the stalks of
\begin{math}
\iota^{*}R^{n}j_{*}\mu_{p}^{\otimes n}
\end{math}
at $\bar{\mathfrak{m}}$. By \cite[p.88, III, Theorem 1.15]{M}, 
we have an isomorphism
\begin{equation*}
(\iota^{*}R^{n}j_{*}\mu_{p}^{\otimes n})_{\bar{\mathfrak{m}}}
\simeq  
\HO^{n}_{\et}(K_{\bar{\mathfrak{m}}}, \mu_{p}^{\otimes n})
\end{equation*}
where $K_{\bar{\mathfrak{m}}}$ is the maximal unramified extension
of $k(A)$. 
We define the homomorphism
\begin{equation}\label{1ipr}
\HO^{1}_{\et}(k, 
\iota^{*}\mathcal{H}^{n}(\mathfrak{T}_{1}(n))
)
\to
\HO^{1}_{\et}(k, 
\iota^{*}R^{n+1}j_{*}\mu_{p}^{\otimes (n+1)}
)
\end{equation}
which is induced by the homomorphism
\begin{align*}
\xymatrix@C=36pt@R=2.8pt
{
\mathcal{H}^{n}(\mathfrak{T}_{1}(n))_{\bar{\mathfrak{m}}}
\ar[r] 
& 
(R^{n+1}j_{*}\mu_{p}^{\otimes (n+1)})_{\bar{\mathfrak{m}}}
\\
b \ar@{(-}[u] \ar@{|->}[r] & \pi_{2}\cup b_{2} \ar@{(-}[u]
}
\end{align*}
%
where
\begin{equation*}
\pi_{2}=\operatorname{Im}\left(
\HO^{1}_{\et}(K_{\bar{\mathfrak{m}}}, \mathbb{G}_{m}[-1])
\to
\HO^{1}_{\et}(K_{\bar{\mathfrak{m}}}, \mu_{p})
\right)(\pi)
\end{equation*}
and
\begin{equation*}
b_{2}
=\operatorname{Im}\left(
\mathcal{H}^{n}(\mathfrak{T}_{1}(n))_{\bar{\mathfrak{m}}}
\to
\HO^{n}_{\et}(K_{\bar{\mathfrak{m}}}, \mu_{p}^{\otimes n})
\right)(b).    
\end{equation*}
By the property of the cup product, the diagram
\begin{equation*}
\xymatrix{
\HO^{n+1}_{\et}(A, \tau_{\leq n}Rj_{*}\mu_{p}^{\otimes n})\otimes
\HO^{1}_{\et}(A, \tau_{\leq 1}Rj_{*}\mu_{p})
\ar[r]^-{\cup}\ar[d] &
\HO^{n+2}_{\et}(A, \tau_{\leq n+1}Rj_{*}\mu_{p}^{\otimes (n+1)}) \ar[d]
\\
\HO^{1}_{\et}(k, \iota^{*}R^{n}j_{*}\mu_{p}^{\otimes n})\otimes
\HO^{0}_{\et}(k, \iota^{*}R^{1}j_{*}\mu_{p}) \ar[r]_-{\cup}
&
\HO^{1}_{\et}(k, \iota^{*}R^{n+1}j_{*}\mu_{p}^{\otimes (n+1)})
}
\end{equation*}
is commutative. So we have a commutative diagram
\begin{equation*}
\xymatrix{
\HO^{n+1}_{\et}(
A, \mathbb{Z}/p(n)
)\ar[r]^-{(\ref{pr})}\ar[d]
& 
\HO^{n+2}_{\et}(
A, \tau_{\leq n+1}Rj_{*}\mathbb{Z}/p(n+1)) 
\ar[d]
\\
\HO^{1}_{\et}(
k, 
\iota^{*}\mathcal{H}^{n}(\mathbb{Z}/p(n)))
\ar[r]_-{(\ref{1ipr})}
&
\HO^{1}_{\et}(k, 
\iota^{*}R^{n+1}j_{*}\mathbb{Z}/p(n+1))
}    
\end{equation*}
by the definition of the cup product (cf. \cite[p.36, I, \S 4]{N}). 
Hence the statement follows.
\end{proof}

In the rest of this section, we prove that the sequence (\ref{IntGCL}) 
is exact in the special cases
without using the results in the previous section.

\begin{prop}\upshape\label{KVpur}
Let $A$ be a discrete valuation ring of mixed characteristic
$(0, p)$, 
$k$ the residue field of $A$
and $[k: k^{p}]\leq p^{s}$.
Let $\mathfrak{X}$ be a smooth scheme over
$\spec(A)$ and $\iota: Y\to \mathfrak{X}$
the inclusion of the closed fiber. 

Suppose that 
$n=\operatorname{dim}(\mathfrak{X})+s$.
Then we have an isomorphism
\begin{equation*}
R\iota^{!}\mathfrak{T}_{r}(n)  
\simeq 
\nu_{r}^{n-1}[-(n-2)]
\end{equation*}
for $r\geq 1$.
\end{prop}
\begin{proof}\upshape
By the definition of $\mathfrak{T}_{r}(n)$ 
(cf. \cite[pp.522--523, Lemma 1.3.1]{SaP}),
it suffices to show that
\begin{equation}\label{VaniK}
\tau_{\geq n+1}Rj_{*}\mu_{p^{r}}^{\otimes n}=0.  
\end{equation}
So it suffices to show that
\begin{equation}\label{KUV}
\HO^{i}_{\et}(
\mathcal{O}^{sh}_{\mathfrak{X}, x}
\times_{\mathcal{O}_{\mathfrak{X}, x}} 
U,
\mu_{p^{r}}^{\otimes n}
)=0    
\end{equation}
for $i\geq n+1$ by \cite[p.88, III, Theorem 1.15]{M}. 
Here
\begin{math}
\mathcal{O}^{sh}_{\mathfrak{X}, x}
\end{math}
is the strictly henselization of a local ring 
$\mathcal{O}_{\mathfrak{X}, x}$. 

Let $x^{\prime}$
be an inverse image of $x$ under the morphism
\begin{math}
\mathfrak{X}\times_{A}A^{sh}
\to
\mathfrak{X}.
\end{math}
Then we have
\begin{equation*}
\mathcal{O}_{\mathfrak{X}, x}^{sh}
=
\mathcal{O}_{\mathfrak{X}\times_{A}A^{sh}, x^{\prime}}^{sh}.
\end{equation*}
Let $k_{s}$ be the separable closure of $k$.
Since $k\to k_{s}$ is \'{e}tale, 
we have
\begin{equation*}
[k_{s}: k^{p}_{s}]
=
[k: k^{p}]
\leq 
p^{s}
\end{equation*}
by the first fundamental exact sequence 
(\cite[p.193, Theorem 25.1]{Ma})
and \cite[p.202, Theorem 26.5]{Ma}. 
So it suffices to show the equation (\ref{KUV}) in the case where
$k$ is a separable closed field.

Suppose that $k=k_{s}$.
By \cite[pp.118--119, Theorem 15.5]{Ma} and
\cite[p.119, Theorem 15.6]{Ma},
we have
\begin{equation*}
\operatorname{trdeg}_{k}\kappa(x)
=
\operatorname{dim}(Y)
-\operatorname{dim}(\mathcal{O}_{Y, x})
\end{equation*}
where
$\operatorname{trdeg}_{k}\kappa(x)$
is the transcendence degree of
$\kappa(x)$ over $k$.
So we have
\begin{equation}\label{Yrank}
\operatorname{rank}_{\kappa(x)}
\Omega_{
\kappa(x)/\mathbb{Z}
}
\leq
s+
\operatorname{dim}(Y)
-\operatorname{dim}(
\mathcal{O}_{Y, x})
\end{equation}
by \cite[p.57, Corollaire 2.3.2]{G-O}
and \cite[p.202, Theorem 26.5]{Ma}.
Since $\kappa(x)\to\kappa(x)_{s}$ is \'{e}tale,
we have
\begin{equation*}
\operatorname{rank}_{\kappa(x)_{s}}
\Omega_{
\kappa(x)_{s}/\mathbb{Z}
}
=
\operatorname{rank}_{\kappa(x)}
\Omega_{
\kappa(x)/\mathbb{Z}
}
\end{equation*}
by the first fundamental exact sequence.
Hence we have
\begin{equation*}
\operatorname{cd}_{p}(
\mathcal{O}^{sh}_{\mathfrak{X}, x}\times U
)
\leq 
s+
\operatorname{dim}(
\mathfrak{X}
)=n
\end{equation*}
by \cite[p.71, Th\'{e}or\`{e}me 6.1]{G-O}. This completes the proof.
\end{proof}
\begin{lem}\upshape\label{ADHB}
Let $\mathfrak{X}$ be an essentially smooth scheme over the spectrum
of a Dedekind domain.
Let $\epsilon: \mathfrak{X}_{\et}\to \mathfrak{X}_{\Zar}$ be the canonical map of sites
and $\epsilon_{*}$ the forgetful functor. Then we have
\begin{equation}\label{HBZm}
\Gamma(\mathfrak{X},
R^{n+1}\epsilon_{*}\mathbb{Z}(n-1)_{\et}
)_{m}
=
\Gamma(
\mathfrak{X},
R^{n}\epsilon_{*}\mathbb{Z}/m(n-1)_{\et}
)
\end{equation}
for any integer $m>0$.
If 
\begin{equation*}
\HO^{i}_{\Zar}(
\mathfrak{X}, 
\mathbb{Z}/m(n-1)
)=0
\end{equation*}
for $i\geq n$,
then we have
\begin{equation}\label{HBem}
\Gamma(
\mathfrak{X},
R^{n}\epsilon_{*}\mathbb{Z}/m(n-1)_{\et}
)
=
\HO_{\et}^{n}
(
\mathfrak{X},
\mathbb{Z}/m(n-1)
).
\end{equation}
\end{lem}
\begin{proof}\upshape
By \cite[p.774, Theorem 1.2.2]{Ge} and
\cite[p.786, Corollary 4.4]{Ge},
we have the isomorphisms
\begin{equation*}
R^{n}\epsilon_{*}
\mathbb{Z}(n-1)_{\et}
\simeq
\mathcal{H}^{n}
(
\mathbb{Z}(n-1)_{\Zar}
)
=0.
\end{equation*}
Hence the sequence
\begin{equation*}
0
\to  
R^{n}\epsilon_{*}\mathbb{Z}/m(n-1)_{\et}
\to
Rj^{n+1}\epsilon_{*}\mathbb{Z}(n-1)_{\et}
\xrightarrow{\times m}
R^{n+1}\epsilon_{*}\mathbb{Z}(n-1)_{\et}
\end{equation*}
is exact and we have the equation (\ref{HBZm}).
If
\begin{equation*}
\HO^{i}_{\Zar}(
\mathfrak{X}, 
\mathbb{Z}/m(n-1)
)=0
\end{equation*}
for $i\geq n$,
then we have the equation (\ref{HBem})
by \cite[p.774, Theorem 1.2.2]{Ge}.
This completes the proof.
\end{proof}
\begin{prop}\upshape\label{KGC}
Let $A$ be a discrete valuation ring of mixed characteristic $(0, p)$,
$k$ the residue field of $A$ and $[k: k^{p}]=p^{s}$.
Let $\mathfrak{X}$ be a smooth scheme over $\spec(A)$,
$R$  a local ring of $\mathfrak{X}$
at a point $x$ and $n=\operatorname{dim}(\mathfrak{X})+s$.
Suppose that 
$A$ contains $p$-th roots of unity. 
Then the sequence 
\begin{align*}
0
&\to 
\HO^{n+1}_{\et}(
R, \mathbb{Z}/p^{r}(n)
)
\to
\HO^{n+1}_{\et}(
k(R), 
\mathbb{Z}/p^{r}(n)
) \\
&\to
\bigoplus_{x\in\spec(R)^{(1)}}
\HO^{n}_{\et}(
\kappa(x),
\mathbb{Z}/p^{r}(n-1)
)  
\to\cdots
\end{align*}
is exact for any positive integer $r$
in the following cases:
\begin{itemize}
\item[(i)] $k$ is finite field and $p\neq 2$.
\item[(ii)] $A$ is henselian.
\end{itemize}
\end{prop}
\begin{proof}\upshape
Let $U$ be the generic fiber of $\spec(R)$. Then we have
\begin{equation*}
\HO^{i}_{\Zar}(U, \mathbb{Z}(n))
=
\HO^{i}_{\Zar}(
U, \mathbb{Z}/p(n)
)
=0    
\end{equation*}
for $i\geq n+1$ by \cite[p.779, Theorem 3.2]{Ge} 
and \cite[p.786, Corollary 4.4]{Ge}.
So the homomorphism
\begin{equation}\label{UKinj}
\HO^{n+1}_{\et}(
U, \mathbb{Z}/p(n)
)
\to
\HO^{n+1}_{\et}(
k(R),
\mathbb{Z}/p(n))
\end{equation}
is injective by \cite[Theorem 4.6]{Sak} and 
Lemma \ref{ADHB}.
Since the homomorphism
\begin{equation*}
\HO^{n}_{\et}\left(
R/(\pi),
\mathbb{Z}/p(n-1)
\right)
\to
\HO^{n}_{\et}\left(
\kappa(
(\pi)
),
\mathbb{Z}/p(n-1)
\right)
\end{equation*}
is injective, the composite of maps
\begin{equation*}
\HO^{n}_{\et}(
R,
\mathbb{Z}/p(n-1)
)
\to
\HO^{n+1}_{\et}(
U,
\mathbb{Z}/p(n)
)
\end{equation*}
which is induced by the morphism (\ref{pr}) and
\begin{equation}\label{BoK}
\HO^{n+1}_{\et}(
U,
\mathbb{Z}/p(n)
)
\to
\HO^{n}_{\et}(
R/(\pi),
\mathbb{Z}/p(n-1)
)
\end{equation}
equals 
the natural map
\begin{equation*}
\HO^{n}_{\et}(
R, \mathbb{Z}/p(n-1)
)
\to
\HO^{n}_{\et}(
R/(\pi), 
\mathbb{Z}/p(n-1)
)
\end{equation*}
by (\ref{VaniK}) and Lemma \ref{compH}. 
So the homomorphism 
(\ref{BoK}) is surjective by Proposition \ref{Bsur}.
Hence the homomorphism
\begin{equation}\label{RUinj}
\HO^{i}_{\et}(
R,
\mathbb{Z}/p(n)
)\to
\HO^{i}_{\et}(
U,
\mathbb{Z}/p(n)
)
\end{equation}
is an isomorphism for $i\geq n+2$ by 
Proposition \ref{KVpur} and
\cite[Expos\'{e} X, Th\'{e}or\`{e}me 5.1]{SGA4}.
Since $A$ contains $p$-th roots of unity, 
the homomorphism
\begin{equation}\label{KinRK}
\HO^{i} _{\et}(
R,
\mathbb{Z}/p(n)
)
\to
\HO^{i}_{\et}(
k(R),
\mathbb{Z}/p(n)
)
\end{equation}
is injective for $i\geq n+2$
by (\ref{RUinj}) and (\ref{UKinj}).
In the case (i), $\operatorname{cd}_{p}(k(A))=2$.
In the case (ii),
\begin{math}
\operatorname{cd}_{p}(k(A))
\leq s +2
\end{math}
by \cite[p.48, Th\'{e}or\`{e}me 1.2]{G-O}.
So we have
\begin{equation}\label{cdkR}
\operatorname{cd}_{p}(k(R))
\leq
\operatorname{dim}(\mathfrak{X})
+s+1=n+1
\end{equation}
by \cite[p.367, (6.5.14) Theorem]{N}. Hence we have
\begin{equation*}
\HO^{i}_{\et}(
R,
\mathbb{Z}/p(n)
)
=0
\end{equation*}
for $i\geq n+2$ by the equation (\ref{cdkR}) and the injectivity of (\ref{KinRK}). 
Since the sequence
\begin{equation*}
\HO^{i}_{\et}(
R,
\mathbb{Z}/p(n)
) 
\to
\HO^{i}_{\et}(
R,
\mathbb{Z}/p^{r+1}(n)
) 
\to
\HO^{i}_{\et}(
R,
\mathbb{Z}/p^{r}(n)
) 
\end{equation*}
is exact, 
we have the equation
\begin{equation}\label{KVr}
\HO^{i}_{\et}(
R,
\mathbb{Z}/p^{r}(n)
)    
=0
\end{equation}
for $i\geq n+2$ and $r\geq 1$.
Let $Y$ be the closed fiber of $\mathfrak{X}$ and 
$k_{s}$ the separable closure of $k$. Since
the composite of $k(Y)$ and $k_{s}$ is \'{e}tale over $k(Y)$,
we have
\begin{equation*}
\operatorname{rank}_{k(Y)}\Omega_{k(Y)/\mathbb{Z}}
\leq 
s+\operatorname{dim}(Y)
=n-1
\end{equation*}
by (\ref{Yrank}). So we have
\begin{equation}\label{KVSup}
\HO_{x}^{i}(
R_{\et},
\mathbb{Z}/p^{r}(n)
)
=
\HO_{\et}^{i-2c}(
\kappa(x),
\mathbb{Z}/p^{r}(n-c)
)
=0
\end{equation}
for $x\in \spec(R)^{(c)}\cap Y$ and 
$i\geq n+c+2$
by Proposition \ref{KVpur} 
and \cite[p.583, Theorem 3.2]{Sh}.
In the case where $x\in \spec(R)^{(c)}\setminus Y$
and $i\geq n+c+2$, we have 
\begin{equation*}
\operatorname{trdeg}_{k(A)}\kappa(x)
\leq
\operatorname{dim}(\mathfrak{X})
-c-1
\end{equation*}
by
\cite[pp.118--119, Theorem 15.5]{Ma} and
\cite[p.119, Theorem 15.6]{Ma}.
So we have also the equation (\ref{KVSup}) 
for $x\in \spec(R)^{(c)}\setminus Y$
and $i\geq n+c+2$
by \cite[p.241, VI, Theorem 5.1]{M}
and \cite[p.367, (6.5.14) Theorem]{N}.
Consider the spectral sequence
\begin{equation*}
E_{1}^{s, t} 
=
\displaystyle
\bigoplus_{x\in \spec(R)^{(s)}}
\HO^{s+t}_{x}\left(
R_{\et},
\mathbb{Z}/p^{r}(n)
\right)
\Rightarrow
E^{s+t}
=\HO^{s+t}_{\et}
\left(
R,
\mathbb{Z}/p^{r}(n)
\right)
\end{equation*}
(cf. Theorem \ref{sp}). Then we have the equation
\begin{equation}\label{KVs2}
E_{2}^{s, t}=0
\end{equation}
for 
$t\geq n+2$
by the equation (\ref{KVSup}).
Moreover, we have the equation (\ref{KVs2})
for $t\leq n$ and $s>0$
by \cite[p.774, Theorem 1.2.2 and Theorem 1.2.5]{Ge}.
Hence we have
\begin{equation*}
E_{2}^{s, n+1}=0    
\end{equation*}
for $s>0$ by the equations (\ref{KVr}) and (\ref{KVs2}).
This completes the proof.
\end{proof}
\section{The Gersten-type conjecture for the mod $p$ \'{e}tale motivic cohomology}\label{Sahe}
In this section, we first show relations between the generalized Brauer group
and the mod $p$ \'{e}tale motivic cohomology in high degrees.
\begin{prop}\upshape\label{compi}
Let $A$ be a discrete valuation ring and 
$\pi$ the maximal ideal of $A$.
Let $R$ be 
(the henselization of)
a local ring of
a smooth algebra over $A$
with $\operatorname{char}(R)=(0, p)$.
Then we have
\begin{equation}\label{mpi}
\HO^{n+1}_{\et}(
R/(\pi), \iota^{*}\mathbb{Z}/p^{r}(n)
)
=
\HO^{n+1}_{\et}(
R/(\pi), \mathbb{Z}/p^{r}(n)
)
\end{equation}
and
\begin{equation}\label{Vmpi}
\HO^{s}_{\et}(R/(\pi), \iota^{*}\mathbb{Z}/p^{r}(n))
=0
\end{equation}
for $s\geq n+2$ and any integer $r>0$
where $\iota: \spec(R/(\pi))\to\spec(R)$ is
the inclusion of the closed fiber.
\end{prop}
\begin{proof}\upshape
Assume that $r=1$.
We have already proved (\ref{mpi}) 
in the proof of \cite[Theorem 5.6]{Sak}
(cf. \cite[(5.6)]{Sak}).
Since $\operatorname{cd}_{p}(R/(\pi))\leq 1$
by \cite[Expos\'{e} X, Th\'{e}or\`{e}me 5.1]{SGA4}, 
we do not need to assume 
that $R$ is henselian
in the proof of \cite[(5.6)]{Sak}.
Moreover, we have
\begin{equation}\label{pl}
\HO^{n+1}_{\et}(
R/(\pi),
\tau_{\leq n}\iota^{*}Rj_{*}\mu_{p}
)
=
\HO^{n+1}_{\et}(
R/(\pi),
\mathbb{Z}/p(n)
)
\oplus
\HO^{n}_{\et}(
R/(\pi),
\mathbb{Z}/p(n-1)
)
\end{equation}
by \cite[(5.9)]{Sak}.
Here $j: \spec(R\otimes_{A}k(A))\to\spec(R)$ is the inclusion of
the generic fiber.
Since $\operatorname{cd}_{p}(R/(\pi))\leq 1$
by \cite[Expos\'{e} X, Th\'{e}or\`{e}me 5.1]{SGA4}, 
we have
\begin{equation}\label{Vps}
\HO^{s}_{\et}(R/(\pi),
\mathbb{Z}/p(n)
)  
=0
\end{equation}
and
\begin{equation}\label{Vji}
\HO^{s}_{\et}
\left(
R/(\pi),
\tau_{\leq n}\iota^{*}Rj_{*}\mu_{p}
\right)
=0
\end{equation}
for $s\geq n+2$ by the spectral sequence
\begin{equation*}
E^{s, t}_{2}
=
\HO^{s}_{\et}
\left(
R/(\pi),
\mathcal{H}^{t}(\tau_{\leq n}\iota^{*}Rj_{*}\mu_{p})
\right)
\Rightarrow
\HO^{s+t}_{\et}
\left(
R/(\pi),
\tau_{\leq n}\iota^{*}Rj_{*}\mu_{p}
\right).
\end{equation*}
If $s=n+2$, the equation (\ref{Vmpi}) follows from (\ref{pl}) and (\ref{Vji}).
If $s> n+2$, the equation (\ref{Vmpi}) follows from (\ref{Vps})
and (\ref{Vji}). Hence the statement holds for $r=1$.

By the induction on $r$, the statement follows from the snake lemma.
This completes the proof.
\end{proof}
\begin{prop}\upshape\label{gwh}
Let the notations $R$, $A$, $\pi$ be the same as above. 
Suppose that $R$ contains $p$-th roots of unity.
Then we have an exact sequence
\begin{equation*}
\HO^{n+1}_{\et}(R, \mathbb{Z}/p(n-1))
\to
\HO^{n+1}_{\et}(R, \mathbb{Z}/p(n))
\to
\HO^{n+1}_{\et}(R/(\pi), \mathbb{Z}/p(n))
\to
0
\end{equation*}
and isomorphisms
\begin{equation*}
\HO^{n+1}_{\et}(
R, 
\mathbb{Z}/p(n-1)
)
\simeq
\HO^{n+1}_{\et}(
R, 
\mathbb{Z}/p(n-k)
)
\simeq
\HO^{n+1}_{\et}(
R, 
j_{!}\mu_{p}^{\otimes (n-k)}
)
\end{equation*}
for $1\leq k\leq n$. Here 
$j: \spec(R\otimes_{A}k(A))\to\spec(R)$
is the inclusion of the generic fiber.
\end{prop}
\begin{proof}\upshape
Let $\iota: \spec(R/(\pi))\to \spec(R)$ be the inclusion 
of the closed fiber.
Since $j_{!}$ is exact (cf. \cite[p.76, II, Proposition 3.14 (b)]{M}), 
we have an isomorphism
\begin{equation*}
j_{!}(\mathbb{Z}/p(n)_{\et}
)
\simeq
j_{!}\mu_{p}^{\otimes n}
\end{equation*}
by \cite[p.774, Theorem 1.2.4]{Ge}.
So
we have a distinguished
triangle
\begin{equation*}
\cdots\to    
j_{!}\mu_{p}^{\otimes n}
\to \mathbb{Z}/p(n)_{\et}
\to \iota_{*}\iota^{*}(
\mathbb{Z}/p(n)_{\et}
)\to\cdots
\end{equation*}
by \cite[pp.75--76, II, Remark 3.13]{M}.
Hence we have an exact sequence
\begin{align*}
&\HO^{n+1}_{\et}(R, j_{!}\mu_{p}^{\otimes n})\to  
\HO^{n+1}_{\et}(R, \mathbb{Z}/p(n))\to
\HO^{n+1}_{\et}(R/(\pi), \mathbb{Z}/p(n))  \\
\to
&\HO^{n+2}_{\et}(R, j_{!}\mu_{p}^{\otimes n})\to  
\HO^{n+2}_{\et}(R, \mathbb{Z}/p(n))
\to 0
\end{align*}
and an isomorphism 
\begin{equation*}
\HO^{s}_{\et}(
R, j_{!}\mu_{p}^{\otimes n}
) 
=
\HO^{s}_{\et}(
R, \mathbb{Z}/p(n)
)
\end{equation*}
for $s\geq n+3$ by Proposition \ref{compi}.
Since $R$ contains $p$-th roots of unity, 
the statement follows from
Proposition \ref{Bsur} and Proposition \ref{compi}.
\end{proof}
\begin{lem}\upshape\label{Zeh}
Let $X$ be a smooth scheme over the spectrum of
a discrete valuation ring $A$ of mixed characteristic $(0, p)$.
We consider the Cousin complex $\BO^{q, n}_{m}(X)^{\bullet}$ (cf. \S\ref{Prel}).
Then we have the followings:
\begin{itemize}
\item[(i)] Let $m$ be any positive integer. 
Then we have
\begin{equation*}
\HO_{\Zar}^{j}(X, 
R^{s}\epsilon_{*}\mathbb{Z}/m(n)
)
=
\HO^{j}(\BO^{s-n, n}_{m}(
X
)^{\bullet}
)
\end{equation*}
for $j\geq 0$ and $s\leq n$ 
where $\epsilon: X_{\et}\to X_{\Zar}$ is the canonical map of sites
and $\epsilon_{*}$ is the forgetful functor.
\item[(ii)] If $A$ contains $p$-th roots of unity, 
we have
\begin{equation*}
\HO_{\Nis}^{j}(X, 
R^{s}\alpha_{*}\mathbb{Z}/p(n)
)
=
\HO^{j}(\BO^{s-n, n}_{p}(
X
)^{\bullet}
)
\end{equation*}
for $j\geq 0$ and $s>n$ where $\alpha: X_{\et}\to X_{\Nis}$ is the canonical map of sites
and $\alpha_{*}$ is the forgetful functor.
\end{itemize}
\end{lem}
\begin{proof}\upshape
First we show (i).
By \cite[p.774, Theorem 1.2.1]{Ge}, we have
\begin{equation*}
\HO^{s+j}_{x}(
X_{\et}, \mathbb{Z}/m(n)
)
=
\HO^{s-j}_{\et}(
\kappa(x), 
\mathbb{Z}/m(n-j)
)
\end{equation*}
for $x\in X^{(j)}$.
So we have a flabby resolution
%
\begin{align*}
0&\to R^{s}\epsilon_{*}\mathbb{Z}/m(n)
\to 
\bigoplus_{x\in X^{(0)}}
(i_{x})_{*}\HO^{s}_{\et}(
\kappa(x), \mathbb{Z}/m(n)
)  \\
&
\to
\bigoplus_{x\in X^{(1)}}
(i_{x})_{*}\HO^{s-1}_{\et}(
\kappa(x), \mathbb{Z}/m(n-1)
)\to\cdots  
\end{align*}
%
by \cite[p.774, Theorem 1.2.2 and Theorem 1.2.5]{Ge}.
Hence (i) follows.

Next we show (ii). Since the Nisnevich cohomological
dimension of a field is $0$ by 
\cite[pp.279--280, 1.32. Theorem]{Ni},
a presheaf
\begin{equation*}
(i_{x})_{*}R^{t}(i^{\prime}_{x})^{!}
\left(R\alpha_{*}\mathbb{Z}/p(n)\right)
\end{equation*}
is a flabby sheaf 
for any integer $t\geq 0$
by \cite[p.89, III, Lemma 1.19]{Ge}.
Here we write 
$i^{\prime}_{x}$
for the canonical map
\begin{math}
\spec(\kappa(x))
\to
\spec(\mathcal{O}_{X, x}).
\end{math}
So we have a flabby resolution
%
\begin{align*}
0&\to  
R^{s}\alpha_{*}\mathbb{Z}/p(n)
\to
\bigoplus_{x\in X^{(0)}}
(i_{x})_{*}R^{s}(i^{\prime}_{x})^{!}
\left(R\alpha_{*}\mathbb{Z}/p(n)\right)
\\
&\to
\bigoplus_{x\in X^{(1)}}
(i_{x})_{*}R^{s+1}(i^{\prime}_{x})^{!}
\left(R\alpha_{*}\mathbb{Z}/p(n)\right) 
\to\cdots
\end{align*}
%
by Theorem \ref{premain}.
Moreover, we have an isomorphism
\begin{equation*}
\HO_{\Nis}^{t}
\left(
K, R(i^{\prime}_{x})^{!}(R\alpha_{*}\mathbb{Z}/p(n))
\right)
\xrightarrow{\simeq}
\HO^{0}_{\Nis}
\left(
K, R^{t}(i^{\prime}_{x})^{!}(
R\alpha_{*}\mathbb{Z}/p(n))
\right)
\end{equation*}
for any integer $t\geq 0$ 
and a finite \'{e}tale extension $K$ of $\kappa(x)$
because the Nisnevich cohomological dimension of a field is $0$.
Hence (ii) follows from \cite[p.776, Proposition 2.2 a)]{Ge}.
This completes the proof.
\end{proof}
\begin{prop}\upshape\label{Asin}
Let $R$ be a local ring of a smooth algebra over a discrete
valuation ring $A$ of mixed characteristic $(0, p)$. 
Let $n\geq 0$ be an integer.
Suppose that 
$R$ contains $p$-th roots of unity. 

If the homomorphism
\begin{equation*}
\HO^{n+r}_{\et}(
R, 
\mathbb{Z}/p(n)
)    
\to
\HO^{n+r}_{\et}(
k(R), 
\mathbb{Z}/p(n)
) 
\end{equation*}
is injective for $r\geq 1$, then the sequence
%
\begin{align*}
0
&
\to  
\HO^{n+r}_{\et}(
R, 
\mathbb{Z}/p(n)
)    
\to
\bigoplus_{x\in \spec(R)^{(0)}}
\HO^{n+r}_{x}(
R_{\et}, 
\mathbb{Z}/p(n)
)\\
&
\to
\bigoplus_{x\in \spec(R)^{(1)}}
\HO^{n+r+1}_{x}(
R_{\et}, 
\mathbb{Z}/p(n)
)
\to \cdots
\end{align*}
is exact for $r\geq 1$.
\end{prop}
\begin{proof}\upshape
By \cite[Theorem 4.6]{Sak}, the sequence
%
\begin{align*}
0
&
\to
\HO^{n+r+1}_{\et}(R,
\mathbb{Z}/p(n+r)
)
\to
\HO^{n+r+1}_{\et}(k(R),
\mathbb{Z}/p(n+r)
) \\
&
\to
\bigoplus_{x\in \spec(R)^{(1)}}
\HO^{n+r+2}_{x}(
R_{\et}, 
\mathbb{Z}/p(n+r)
)
\end{align*}
%
is exact.
Moreover, the sequence 
%
\begin{align*}
0
&
\to 
\HO^{n+r+1}_{\et}(
\kappa((\pi)),
\mathbb{Z}/p(n+r)
)
\to
\HO^{n+r+2}_{(\pi)}(
R_{\et}, \mathbb{Z}/p(n)
)
\\
&
\to
\HO^{n+r+2}_{(\pi)}(
R_{\et}, \mathbb{Z}/p(n+r)
)
\end{align*}
%
is exact 
for $1\leq r\leq n$
by Theorem \ref{premain}.
Since 
$R$ contains $p$-th roots of unity,
we have isomorphisms
\begin{equation*}
\mu_{p}^{\otimes n}
\simeq
\mu_{p}^{\otimes (n+r)}
\end{equation*}
and
\begin{equation*}
\HO^{n+r+2}_{x}(R_{\et}, \mathbb{Z}/p(n))
\simeq 
\HO^{n+r+2}_{x}(R_{\et}, \mathbb{Z}/p(n+r))
\end{equation*}
for $x\in\spec(R)^{(1)}\setminus\spec(R/(\pi))$.
Here $\pi$ is a prime element of $A$.
Hence
the sequence
\begin{align}\label{rpex}
0
&\to    
\HO^{n+r+1}_{\et}(
R, 
\mathbb{Z}/p(n)
)
\to    
\HO^{n+r+1}_{\et}(
k(R), 
\mathbb{Z}/p(n)
) \nonumber
\\
&\to
\bigoplus_{x\in \spec(R)^{(1)}}
\HO^{n+r+2}_{x}(
R_{\et}, 
\mathbb{Z}/p(n)
)
\end{align}
is exact for $r\geq 1$
by Proposition \ref{gwh}
because 
\begin{math}
\kappa((\pi))=k(R/(\pi))
\end{math}
and the homomorphism
\begin{equation*}
\HO^{n+r+1}_{\et}(R/(\pi),
\mathbb{Z}/p(n+r))
\to
\HO^{n+r+1}_{\et}(
k(R/(\pi)),
\mathbb{Z}/p(n+r))
\end{equation*}
is injective by \cite[p.600, Theorem 4.1]{Sh}.

So it suffices to show that
\begin{equation}\label{VHa}
\HO^{j}(
\BO^{1+r, n}_{p}(
R
)^{\bullet}
)
=0    
\end{equation}
for $r=0, 1$ and $j>0$. 
Let $\beta: \spec(R)_{\Nis}\to \spec(R)_{\Zar}$ be the canonical map of sites and
$\beta^{*}$ is the sheafification functor. 
Since 
\begin{math}
\beta^{*}\mathbb{Z}/p(n)_{\Zar}
=
\mathbb{Z}/p(n)_{\Nis},
\end{math}
we have a distinguished triangle
\begin{equation*}
\cdots\to 
\mathbb{Z}/p(n)_{\Nis}
\to
R\alpha_{*}\mathbb{Z}/p(n)_{\et}
\to
\tau_{\geq n+1}
R\alpha_{*}\mathbb{Z}/p(n)_{\et}
\to
\cdots
\end{equation*}
by \cite[p.774, Theorem 1.2.2]{Ge}.
So we have an isomorphism
\begin{equation*}
\HO_{\et}^{s}(R, \mathbb{Z}/p(n))
\simeq
\HO^{s}_{\Nis}
(
R,
\tau_{\geq n+1}R\alpha_{*}
\mathbb{Z}/p(n)
)
\end{equation*}
for $s\geq n+1$ by \cite[p.786, Corollary 4.4]{Ge}. 
Moreover, we have isomorphisms
\begin{align*}
\HO^{s}_{\Nis}
\left(
R, \tau_{\geq s}(\tau_{\geq n+1}R\alpha_{*}\mathbb{Z}/p(n)
)
\right)
\simeq
\HO^{0}_{\Nis}
(R, R^{s}\alpha_{*}\mathbb{Z}/p(n)
)
\simeq
\HO^{s}_{\et}(R, \mathbb{Z}/p(n))
\end{align*}
for $s-1\geq n+1$ by (\ref{rpex}) 
and Lemma \ref{Zeh} (ii). 
Hence we have
\begin{equation}\label{sNis}
\HO^{s}_{\Nis}\left(
R,
\tau_{\leq s-1}(\tau_{\geq n+1}R\alpha_{*}
\mathbb{Z}/p(n)
)
\right)    
=0
\end{equation}
for $s-1\geq n+1$. 
Then we show the equation (\ref{VHa}) by the induction on $j$.

Assume that $j=1$. By the equation 
(\ref{sNis}) for $s=n+2$
and Lemma \ref{Zeh} (ii), 
we have
\begin{equation*}
\HO^{1}(
\BO_{p}^{1, n}(
R
)^{\bullet}
)
=
\HO^{1}_{\Nis}(
R,
R\alpha_{*}^{n+1}\mathbb{Z}/p(n)
)
=0.
\end{equation*}
On the other hand, the sequence
\begin{align}\label{hex}
\HO^{0}(\BO^{1, n}_{p}(
R
)^{\bullet}
)
\to 
\HO^{0}(\BO^{1, n}_{p}(
R/(\pi)
)^{\bullet}
)
\to
\HO^{1}(\BO^{2, n-1}_{p}(
R
)^{\bullet}
)
\to 
\HO^{1}(\BO^{1, n}_{p}(
R
)^{\bullet}
)
\to 0
\end{align}
is exact by Lemma \ref{Cornat}, Theorem \ref{premain} and \cite[p.600, Theorem 4.1]{Sh}.
Then the first map in (\ref{hex}) equals up to sign the natural map
\begin{equation*}
\HO^{n+1}_{\et}(
R,
\mathbb{Z}/p(n)
)    
\to
\HO^{n+1}_{\et}(
R/(\pi),
\mathbb{Z}/p(n)
)
\end{equation*}
by Lemma \ref{Cornat}.
So the first map in (\ref{hex}) is surjective 
by Proposition \ref{Bsur}. 
Hence we have an isomorphism
\begin{equation*}
\HO^{1}(\BO^{2, n-1}_{p}(
R
)^{\bullet}
)
\to
\HO^{1}(\BO^{1, n}_{p}(
R
)^{\bullet}
)
\end{equation*}
and the equation (\ref{VHa}) for $j=1$ and $r=0, 1$.

Assume that we have the equation (\ref{VHa}) 
for $j\leq i$. 
Then we have
\begin{equation*}
\HO^{s}_{\Nis}
\left(
R,
\tau_{\leq s-1}(
\tau_{\geq n+2}R\alpha_{*}
\mathbb{Z}/p(n)
)
\right)
=0
\end{equation*}
for $s\leq n+2+i$ by the assumption.
So we have the equation (\ref{VHa})
for $j=i+1$ and $r=0$ by the equation (\ref{sNis})
for $s=n+i+2$.
Moreover, we have an 
isomorphism
\begin{equation*}
\HO^{i+1}(\BO^{2, n-1}_{p}(
R
)^{\bullet}
)  
\to
\HO^{i+1}(\BO^{1, n}_{p}(
R
)^{\bullet}
) 
\end{equation*}
by Theorem \ref{premain} and \cite[p.600, Theorem 4.1]{Sh}.
Hence we have the equation (\ref{VHa}) for 
$j=i+1$ and $r=1$.
This completes the proof.
\end{proof}
Finally, we prove the Gersten-type conjecture for the mod $p$ \'{e}tale motivic cohomology 
(Theorem \ref{IntGCL}).
\begin{lem}\upshape\label{Surrp}
Let $A$ be a discrete valuation ring of mixed characteristic $(0, p)$ and $\pi$ a prime element of $A$.
Let $R$ be 
a local ring of a smooth scheme over $\spec(A)$,
$\iota: \spec(R/(\pi))\to \spec(R)$ 
the inclusion of the closed fiber and
$j: \spec(R[\pi^{-1}])\to \spec(R)$
the inclusion of the generic fiber.
Suppose that
$R$ contains $p$-th roots of unity.

Then the homomorphism
\begin{equation*}
\HO^{n+1}_{\et}(
R, 
\tau_{\leq n}Rj_{*}\mu_{p}^{\otimes n}
)    
\to
\HO^{n+1}_{\et}(
R/(\pi), 
\tau_{\leq n}\iota^{*}Rj_{*}\mu_{p}^{\otimes n}
)
\end{equation*}
is surjective.
\end{lem}
\begin{proof}\upshape
By Proposition \ref{compi},
the composite
\begin{equation*}
\HO^{n+1}_{\et}(
R/(\pi),
\iota^{*}\mathfrak{T}_{1}(n)
)
\to
\HO^{1}_{\et}(R/(\pi), 
\iota^{*}\mathcal{H}^{n}(\mathfrak{T}_{1}(n))
)
\to
\HO^{1}_{\et}(R/(\pi), \Omega^{n}_{R/(\pi), \log})
\end{equation*}
is an isomorphism. Here $\Omega^{n}_{R/(\pi), \log}$ denotes
$W_{1}\Omega^{n}_{\spec(R/(\pi)), \log}$.
By Proposition \ref{Bsur} and 
Lemma \ref{compH},
the homomorphism
\begin{equation*}
\HO^{n+1}_{\et}(
R, \tau_{\leq n}Rj_{*}\mu_{p}^{\otimes n}
)
\to
\HO^{1}_{\et}(R/(\pi), 
\Omega^{n-1}_{
R/(\pi), \log}
)
\end{equation*}
is surjective. 
Hence we have a commutative diagram
\footnotesize
\begin{equation*}
\xymatrix
{
\HO^{n+1}_{\et}(
R, \mathfrak{T}_{1}(n))
\ar[r]\ar[d]
&
\HO^{n+1}_{\et}(
R, \tau_{\leq n}Rj_{*}\mu_{p}^{\otimes n}
)
\ar[r]\ar[d]
&
\HO^{1}_{\et}(R/(\pi), 
\Omega^{n-1}_{R/(\pi), \log})
\ar[r]\ar@{=}[d]
&
0 
\\
\HO^{1}_{\et}(
R/(\pi), 
\Omega^{n}_{R/(\pi), \log}
)
\ar[r]
&
\HO^{n+1}_{\et}(
R/(\pi), \tau_{\leq n}\iota^{*}Rj_{*}\mu_{p}^{\otimes n}
)
\ar[r]
&
\HO^{1}_{\et}(R/(\pi), 
\Omega^{n-1}_{R/(\pi), \log})
}    
\end{equation*}
\normalsize
where both rows are exact and the left map
is surjective by Proposition \ref{Bsur}.
Therefore the statement follows from the snake lemma.
\end{proof}
\begin{thm}\upshape\label{mainGC}
Let $R$ be 
a local ring of a smooth scheme over the spectrum of a discrete valuation ring
of mixed characteristic $(0, p)$. Let $n\geq 0$ be an integer.
Suppose that 
$R$ contains $p$-th roots of unity.
Then the sequence
\begin{align*}
0\to  
&\HO^{n+r}_{\et}(
R, \mathfrak{T}_{1}(n)
)
\to
\displaystyle\bigoplus_{x\in \spec(R)^{(0)}}
\HO^{n+r}_{x}(
R_{\et},
\mathfrak{T}_{1}(n)
)
\\
\to
&\displaystyle\bigoplus_{x\in \spec(R)^{(1)}}
\HO^{n+r+1}_{x}(
R_{\et},
\mathfrak{T}_{1}(n)
)
\to
\cdots
\end{align*}
is exact for $r\geq 1$.
\end{thm}
\begin{proof}\upshape
Let the notations be the same as in Lemma \ref{Surrp}.
Since $R$ contains $p$-th roots of unity,
we have an isomorphism
\begin{equation*}
\HO_{\et}^{n+1}(R, j_{!}\mu_{p}^{\otimes n})   
\simeq
\HO_{\et}^{n+1}(R, 
\mathfrak{T}_{1}(n-k)
)
\end{equation*}
for $1\leq k\leq n$ by Proposition \ref{gwh}.
If the homomorphism 
\begin{equation}\label{surri}
\HO^{n}_{\et}(
R, 
\mathfrak{T}_{1}(n)
)
\to
\HO^{n}_{\et}
(
R/(\pi), 
\iota^{*}\mathfrak{T}_{1}(n)
)
\end{equation}
is surjective,
the homomorphism 
\begin{equation*}
\HO^{n+1}_{\et}(
R,
j_{!}\mu_{p}^{\otimes n}
)    
\to
\HO^{n+1}_{\et}(
R,
\mathfrak{T}_{1}(n)
) 
\end{equation*}
is  injective and
the homomorphism
\begin{equation*}
\HO^{n+k}_{\et}(R, \mathfrak{T}_{1}(n))
\to
\HO^{n+k}_{\et}(k(R), \mathfrak{T}_{1}(n))
\end{equation*}
is injective 
for $k\geq 2$ by \cite[p.51, Theorem 4.6]{Sak} and Proposition \ref{gwh}.
Then the statement follows from
Proposition \ref{Asin}.
So we prove that the homomorphism (\ref{surri}) is surjective.

We consider the commutative diagram
%
\small
\begin{equation}\label{comt}
\xymatrix{
\HO^{n}_{\et}(
R, \tau_{\leq n-1}\mathfrak{T}_{1}(n)
)
\ar[r]\ar[d]
&
\HO^{n}_{\et}(
R, \mathfrak{T}_{1}(n)
)
\ar[d]
\\
\HO^{n}_{\et}(
R/(\pi), \tau_{\leq n-1}\iota^{*}\mathfrak{T}_{1}(n)
)
\ar[r]
&
\HO^{n}_{\et}(
R/(\pi), 
\iota^{*}\mathfrak{T}_{1}(n)
)
\ar[r]
&
\HO^{n}_{\et}(
R/(\pi), 
\tau_{\geq n}\iota^{*}\mathfrak{T}_{1}(n)
)
}    
\end{equation}
\normalsize
where the lower sequence is exact.
Since $R$ contains $p$-th roots of unity,
we have quasi-isomorphisms
\begin{equation*}
\tau_{\leq n-1}\mathfrak{T}_{1}(n)
\simeq
\tau_{\leq n-1}Rj_{*}\mu_{p}^{\otimes n}
\simeq
\tau_{\leq n-1}Rj_{*}\mu_{p}^{\otimes n-1}
\end{equation*}
and so the left map in the diagram (\ref{comt})
is surjective by Lemma \ref{Surrp}. 
Moreover we have an isomorphism
\begin{equation*}
\HO^{0}_{\et}(
R/(\pi),
\mathcal{H}^{n}(
\iota^{*}\mathfrak{T}_{1}(n)
)
\to
\HO^{n}_{\et}(
R/(\pi), 
\tau_{\geq n}\iota^{*}\mathfrak{T}_{1}(n)
)    
).
\end{equation*}
Hence it suffices to show that
the composite
\begin{equation}\label{comp}
\HO^{n}_{\et}(
R, \mathfrak{T}_{1}(n)
)    
\to
\HO^{n}_{\et}(
R/(\pi), 
\iota^{*}\mathfrak{T}_{1}(n)
) 
\to 
\HO^{0}_{\et}(
R/(\pi), 
\mathcal{H}^{n}(\iota^{*}\mathfrak{T}_{1}(n))
) 
\end{equation}
is surjective. 
By \cite[p.55, Lemma 5.2]{Sak} and 
Theorem \ref{CUV}, 
we have
\begin{equation*}
\HO^{s}_{\et}(
R/(\pi), 
\operatorname{gr}^{q}_{U/V}M^{n}_{1}
)    
=
\HO^{s}_{\et}(
R/(\pi), 
\operatorname{gr}^{q}_{V/U}M^{n}_{1}
)    
=
0
\end{equation*}
for $s>0$, $q>0$ and
\begin{equation*}
U^{q}M^{n}_{1}
=
V^{q}M^{n}_{1}
=
0
\end{equation*}
for sufficiently large $q$.
Here the notations 
$U^{q}M^{n}_{1}$,
$V^{q}M^{n}_{1}$,
$\operatorname{gr}^{q}_{U/V}M^{n}_{1}$
and
$\operatorname{gr}^{q}_{V/U}M^{n}_{1}$
are defined in \S \ref{RVT}.
So we have
\begin{equation*}
\HO^{s}_{\et}(
R/(\pi), 
U^{q}M^{n}_{1}
)    
=
\HO^{s}_{\et}(
R/(\pi), 
V^{q}M^{n}_{1}
)    
=
0
\end{equation*}
for $s>0$ and $q>0$.
Hence the sequences
\small
\begin{align*}
&0\to
\HO^{0}_{\et}(R/(\pi), 
U^{1}M^{n}_{1})
\to
\HO^{0}_{\et}(R/(\pi), 
\iota^{*}\mathcal{H}^{n}(\mathfrak{T}_{1}(n)))
\to
\HO^{0}_{\et}(R/(\pi), \Omega^{n}_{R/(\pi), \log})
\to 0
\\
&0\to
\HO^{0}_{\et}(R/(\pi), V^{q}M^{n}_{1})
\to
\HO^{0}_{\et}(R/(\pi), U^{q}M^{n}_{1})
\to
\HO^{0}_{\et}(R/(\pi), \operatorname{gr}^{q}_{U/V}M^{n}_{1})
\to 0
\\
&0\to
\HO^{0}_{\et}(R/(\pi), U^{q+1}M^{n}_{1})
\to
\HO^{0}_{\et}(R/(\pi), V^{q}M^{n}_{1})
\to
\HO^{0}_{\et}(R/(\pi), \operatorname{gr}^{q}_{V/U}M^{n}_{1})
\to 0
\end{align*}
\normalsize
are exact for $q>0$. Here $\Omega^{n}_{R/(\pi), \log}$ denotes
$W_{1}\Omega^{n}_{\spec(R/(\pi)), \log}$.
Moreover, there are surjective maps
\begin{align*}
\operatorname{K}^{M}_{n}(R/(\pi))/p
&\xrightarrow{\simeq}   
\HO^{0}_{\et}(
R/(\pi), 
\Omega^{n}_{R/(\pi), \log}
)  \\
y_{1}\otimes\cdots\otimes y_{n}
&\mapsto
\frac{dy_{1}}{y_{1}}
\wedge\cdots\wedge
\frac{dy_{n}}{y_{n}}
\end{align*}
and
\begin{align*}
R/(\pi)\otimes((R/\pi)^{*})^{\otimes n}
&\to
\Omega_{R/(\pi)}^{n}
\\
x\otimes y_{1}\otimes\cdots y_{n}
&\mapsto
x\frac{dy_{1}}{y_{1}}\wedge\cdots
\wedge
\frac{dy_{n}}{y_{n}}
\end{align*}
by \cite[Proposition 3.1]{G-L} and \cite[p.122, Lemma (4.2)]{B-K}.
Here $K^{M}_{n}(R/(\pi))$ is the $n$-th Milnor $K$-group
of $R/(\pi)$.
Therefore the composite (\ref{comp}) is surjective by 
Theorem \ref{CUV}, Theorem \ref{CFU}
and \cite[p.55, Lemma 5.2]{Sak}. 
This completes the proof.
\end{proof}
\section{The Gersten resolution for $\mathbb{P}^{m}_{A}$ and $\mathbb{A}^{m}_{A}$
}\label{SP1h}

Let $A$ be a regular local ring
with $\operatorname{dim}(A)\leq 1$.
Suppose that $\operatorname{char}(A)$
is $p>0$ or $(0, p)$.
Let $f: Z\to X$ be a 
morphism of essentially smooth schemes over $\spec(A)$. 
Let $r$ be a positive integer.
In the case where $A$ is mixed characteristic, there is a morphism
\begin{equation}\label{ipul}
f^{*}\left(
\mathbb{Z}/p^{r}(n)_{\et}^{X}
\right)
\to
\mathbb{Z}/p^{r}(n)_{\et}^{Z}
\end{equation}
by \cite[p.538, Proposition 4.2.8]{SaP}. 
So we can define a morphism
\begin{equation}\label{adipul}
f^{*}:  
\mathbb{Z}/p^{r}(n)^{X}_{\et}
\to
Rf_{*}\left(
\mathbb{Z}/p^{r}(n)_{\et}^{Z}
\right)
\end{equation}
by the adjunction map 
\begin{math}
\mathbb{Z}/p^{r}(n)^{X}_{\et}
\to
Rf_{*}f^{*}\left(
\mathbb{Z}/p^{r}(n)_{\et}^{X}
\right)
\end{math}
and the morphism (\ref{ipul}).
In the case where $\operatorname{char}(A)=p>0$, we have the homomorphism
$f^{*}$ by \cite[p.732, Theorem 3.5.1]{SaL}.

Let $\iota: Z\to X$ be a closed immersion
of codimension $1$ between essentially smooth schemes 
over $\spec(A)$.
Suppose that
\begin{math}
\operatorname{char}(Z)
=
\operatorname{char}(X).
\end{math}
Then we have the Gysin morphism
\begin{equation*}
\operatorname{Gys}_{\iota}^{n}:
\mathbb{Z}/p^{r}(n-1)^{Z}_{\et}[-2]
\to
R\iota^{!}\mathbb{Z}/p^{r}(n)_{\et}^{X}
\end{equation*}
by \cite[p.527, Definition 2.2.1]{SaP}, \cite[p.547, Theorem 6.1.3]{SaP}
and \cite[p.48, (3.5.18)]{Gr}.
Moreover, we have the map
\begin{equation*}
\iota_{*}:
\iota_{*}
\left(
\mathbb{Z}/p^{r}(n-1)^{Z}_{\et}
\right)[-2]
\to
\mathbb{Z}/p^{r}(n)^{X}_{\et}
\end{equation*}
which is the left adjunct of the Gysin morphism
$\operatorname{Gys}_{i}^{n}$.

Then we have the followings:

\begin{prop}\upshape\label{GyPr}
Let the notations be the same as above. Then we have
\begin{equation*}
\iota_{*}\left(
\iota^{*}(x)\cup z
\right)
=
x\cup \iota_{*}z
\end{equation*}
for 
$x\in \mathbb{Z}/p^{r}(n)_{\et}^{X}$
and
$z\in \iota_{*}(\mathbb{Z}/p^{r}(m)^{Z}_{\et})[-2]$. 
\end{prop}
\begin{proof}\upshape
Assume that $A$ is a discrete valuation 
ring of mixed characteristic
$(0, p)$.
Let 
\begin{equation*}
\xymatrix{
U^{\prime} \ar[r]^{j_{Z}}\ar[d]_{\iota_{U}}
& Z \ar[d]_{\iota}
& Y^{\prime} \ar[l]_{\iota_{Z}}\ar[d]^{\iota_{Y}}  \\
U \ar[r]_{j}
& X
& Y \ar[l]^{\iota^{\prime}}
}
\end{equation*}
be Cartesian where $j: U\to X$ 
is the inclusion of the generic fiber and
 $\iota^{\prime}: Y\to X$ is 
the inclusion of the closed fiber.
Then we have a commutative diagram
\begin{equation*}
\begin{CD}
\tau_{\leq n}Rj_{*}
\mu_{p^{r}}^{\otimes n}
\otimes^{\mathbb{L}}
\iota_{*}\tau_{\leq m}R(j_{Z})_{*}\mu_{p^{r}}^{\otimes m}[-2]
@>
{\operatorname{id}\otimes Rj_{*}(\iota_{U})_{*}}
>>
\tau_{\leq n}Rj_{*}
\mu_{p^{r}}^{\otimes n}
\otimes^{\mathbb{L}}
\tau_{\leq m}Rj_{*}
\mu_{p^{r}}^{\otimes m}  \\
@V \cup\circ (Rj_{*}(\iota_{U})^{*}\otimes\operatorname{id})
VV  @VV\cup V \\
\iota_{*}\tau_{\leq m+n}R(j_{Z})_{*}\mu_{p^{r}}^{\otimes m+n}[-2]
@>>{Rj_{*}(\iota_{U})_{*}}>
\tau_{\leq m+n}Rj_{*}
\mu_{p^{r}}^{\otimes m+n}
\end{CD}
\end{equation*}
by \cite[p.250, VI, Proposition 6.5 (a)]{M}.
Although 
$k(A)$ is a separable closed field
in \cite[p.250, VI, Proposition 6.5 (a)]{M}, 
\cite[p.250, VI, Proposition 6.5 (a)]{M} holds without this assumption.
So we have a commutative diagram
\begin{equation*}
\begin{CD}
\mathfrak{T}_{r}(n)_{X}
\otimes^{\mathbb{L}}
\iota_{*}\mathfrak{T}_{r}(m)_{Z}[-2]
@>
{\operatorname{id}\otimes \iota_{*}}
>>
\mathfrak{T}_{r}(m)_{X}
\otimes^{\mathbb{L}}
\mathfrak{T}_{r}(n)_{X}  \\
@V\cup\circ(\iota^{*}\otimes\operatorname{id})VV  @VV\cup V \\
\iota_{*}\mathfrak{T}_{r}(m+n)_{Z}[-2]
@>>{\iota_{*}}>
\mathfrak{T}_{r}(m+n)_{X}
\end{CD}
\end{equation*}
by definitions of Product (cf. \cite[p.538, Proposition 4.2.6]{SaP}), 
$\iota^{*}$ (cf. \cite[p.538, Proposition 4.2.8]{SaP}) and
$\iota_{*}$ (cf. \cite[p.540, Definition 4.4.5]{SaP}).
Hence the statement is true
in the case where
$A$ is mixed characteristic.

Assume that $\operatorname{char}(A)=p>0$.
Then we have a commutative diagram
\begin{equation*}
\begin{CD}
\nu_{r}^{n}[-n]
\otimes (\iota_{Y})_{*}\nu_{r}^{m}[-m-2] 
@>\operatorname{id}[-n]\otimes(\iota_{Y})_{*}[-m]>>
\nu_{r}^{n}[-n]\otimes \nu_{r}^{m}[-m]\\
@V{\cup\circ((\iota_{Y})^{*}\otimes\operatorname{id})}VV 
@VV{\cup}V \\
(\iota_{Y})_{*}\nu_{r}^{n+m}[-(n+m)-2]
@>>{(\iota_{Y})_{*}[-(n+m)]}>
\nu_{r}^{n+m}[-(n+m)]
\end{CD}
\end{equation*}
by 
\cite[p.542, The proof of Proposition 4.4.10]{SaP}. 
Hence the statement is true
in the case where
$\operatorname{char}(A)=p>0$. This completes the proof.

\end{proof}
\begin{cor}\upshape\label{ProjC}
Let the notations be the same as above. Then the diagram
\footnotesize
\begin{equation*}
\xymatrix{
\HO^{s}_{\et}(
X, 
\mathbb{Z}/p^{r}(n)
)\otimes
\HO^{t-2}_{\et}
\left(
Z, 
\mathbb{Z}/p^{r}(m-1)
\right)
\ar[r]^-{\operatorname{id}\otimes \iota_{*}}
\ar[d]_-{\cup\circ(\iota^{*}\otimes\operatorname{id})}
&
\HO^{s}_{\et}(
X, 
\mathbb{Z}/p^{r}(n)
)\otimes
\HO^{t}_{\et}
\left(
X, 
\mathbb{Z}/p^{r}(m)
\right)\ar[d]^-{\cup}   \\
\HO^{s+t-2}_{\et}
(Z,
\mathbb{Z}/p^{r}(n+m-1)
)\ar[r]_-{\iota_{*}}
&
\HO^{s+t}_{\et}(
X,
\mathbb{Z}/p^{r}(n+m)
)
}    
\end{equation*}
\normalsize
is commutative.
In particular, we have
\begin{equation*}
\iota_{*}(\iota^{*}(x))
=x\cup \iota_{*}(1)
\end{equation*}
for $x\in\HO^{s}_{\et}(X, \mathbb{Z}/p^{r}(n))$.
\end{cor}
\begin{proof}\upshape
By the property of the cup product,
the diagrams
\footnotesize
\begin{equation*}
\xymatrix@C=10pt{
\HO^{s}_{\et}
(X, 
\mathbb{Z}/p^{r}(n)
)\otimes 
\HO^{t}_{\et}
(X,
\iota_{*}\mathbb{Z}/p^{r}(m-1)[-2]
)\ar[r]^-{\cup}\ar[d]_{\operatorname{id}\otimes \iota_{*}}
&
\HO^{s+t}_{\et}
\left(
X, 
\mathbb{Z}/p^{r}(n)
\otimes
\iota_{*}\mathbb{Z}/p^{r}(m-1)[-2]
\right) 
\ar[d]
\\
\HO^{s}_{\et}
\left(
X,
\mathbb{Z}/p^{r}(n)
\right)\otimes
\HO^{t}_{\et}(
X,
\mathbb{Z}/p^{r}(m)
)\ar[r]_-{\cup}
&
\HO^{s+t}_{\et}(
X,
\mathbb{Z}/p^{r}(n+m)
)
}  
\end{equation*}
\normalsize
and
\footnotesize
\begin{equation*}
\xymatrix@C=10pt{
\HO^{s}_{\et}(
X, \mathbb{Z}/p^{r}(n)
)
\otimes
\HO^{t}_{\et}(
X, 
\iota_{*}\mathbb{Z}/p^{r}(m-1)[-2]
)\ar[r]^-{\cup}\ar[d]_-{\iota^{*}\otimes\operatorname{id}}
&
\HO^{s+t}_{\et}(
X, 
\mathbb{Z}/p^{r}(n)\otimes
\iota_{*}\mathbb{Z}/p^{r}(m-1)[-2]
)\ar[d]
\\
\HO^{s}_{\et}(
X, 
\iota_{*}\mathbb{Z}/p^{r}(n)
)
\otimes
\HO^{t}_{\et}(
X, 
\iota_{*}\mathbb{Z}/p^{r}(m)[-2]
)\ar[r]_-{\cup}
&
\HO^{s+t}_{\et}(
X,
\iota_{*}\mathbb{Z}/p^{r}(n+m)[-2]
)
}    
\end{equation*}
\normalsize
are commutative. Moreover, the diagram
\footnotesize
\begin{equation*}
\xymatrix{
\HO^{s+t}_{\et}
(
X, 
\mathbb{Z}/p^{r}(n)
\otimes
\iota_{*}\mathbb{Z}/p^{r}(m-1)[-2]
) 
\ar[d]_-{\operatorname{id}\otimes \iota_{*}}
\ar[r]^-{\iota^{*}\otimes\operatorname{id}} 
&
\HO^{s+t}_{\et}
(X,
\iota_{*}\mathbb{Z}/p^{r}(n)
\otimes
\iota_{*}\mathbb{Z}/p^{r}(m-1)[-2]
)
\ar[d]^{\iota_{*}\circ\cup}
\\
\HO^{s+t}_{\et}
(
X,
\mathbb{Z}/p^{r}(n)
\otimes
\mathbb{Z}/p^{r}(m)
)\ar[r]
&
\HO^{s+t}_{\et}
(X, 
\mathbb{Z}/p^{r}(n+m)
)
}  
\end{equation*}
\normalsize
is commutative by Proposition \ref{GyPr}. 
Hence the statement follows.
\end{proof}
\begin{lem}\upshape\label{GyC}
$A$ be a regular local ring
with $\operatorname{dim}(A)\leq 1$.
Suppose that $\operatorname{char}(A)$
is $p>0$ or $(0, p)$.
Let $\mathbb{P}_{A}^{m-1}$ be a hyperplane in $\mathbb{P}_{A}^{m}$
and
$i_{m}: \mathbb{P}_{A}^{m-1}\to \mathbb{P}_{A}^{m}$ 
the corresponding closed immersion. 
Let $\mathcal{O}_{\mathbb{P}_{A}^{m}}(1)$ be the tautological invertible
sheaf on $\mathbb{P}_{A}^{m}$
and
$\xi\in \HO^{2}_{\et}(\mathbb{P}^{m}_{A}, \mathbb{Z}/p^{r}(1))$
be the value of the first class
\begin{math}
\operatorname{c}^{1}(\mathcal{O}_{\mathbb{P}_{A}^{m}}(1))
\in 
\HO^{1}_{\et}(\mathbb{P}^{m}_{A}, \mathbb{G}_{m})
\end{math}
under the map associated with the 
distinguished triangle
\begin{equation*}
\cdots\to
\mathbb{Z}(1)_{\et}
\xrightarrow{\times p^{r}}
\mathbb{Z}(1)_{\et}
\to
\mathbb{Z}/p^{r}(1)_{\et}
\to
\cdots
\end{equation*}
%
and an isomorphism
\begin{equation*}
\mathbb{Z}(1)_{\et}
\simeq
\mathbb{G}_{m}[-1]    
\end{equation*}
(cf. \cite[Lemma 11.2]{L}). Then we have
\begin{equation}\label{xtx}
(i_{m})^{*}(\xi)=\xi
\end{equation}
and the image of the morphism
\begin{equation}\label{Gyi}
(i_{m})_{*}:
\mathbb{Z}/p^{r}
\to
\HO^{2}_{\et}(
\mathbb{P}_{A}^{m}, \mathbb{Z}/p^{r}(1)
)
\end{equation}
is generated by $\xi$.
\end{lem}
\begin{proof}\upshape
We have a commutative diagram
\begin{equation*}
\xymatrix{
(i_{m})^{*}(\mathbb{Z}(1)_{\et})
\ar[r]\ar[d]
&
(i_{m})^{*}(\mathbb{Z}/p^{r}(1)_{\et}) \ar[d]^{i^{*}}  
\\
\mathbb{Z}(1)_{\et}
\ar[r]
&
\mathbb{Z}/p^{r}(1)_{\et} 
}
\end{equation*}
where the left map is induced by the natural map
$(i_{m})^{*}(\mathbb{G}_{m})\to\mathbb{G}_{m}$.
Since $\operatorname{c}^{1}(\mathcal{O}_{\mathbb{P}_{A}^{m-1}}(1))$ 
is the image of
$\operatorname{c}^{1}(\mathcal{O}_{\mathbb{P}_{A}^{m}}(1))$ under the natural map
\begin{equation*}
\HO^{1}_{\et}(
\mathbb{P}^{m}_{A}, 
\mathbb{G}_{m}
)
\to
\HO^{1}_{\et}(
\mathbb{P}^{m-1}_{A}, 
\mathbb{G}_{m}
)
,
\end{equation*}
we have the equation (\ref{xtx}).

Next we prove that the image of (\ref{Gyi}) is generated
by $\xi$.
The diagram
\footnotesize
\begin{equation}\label{diaP}
\xymatrix{
\HO_{\mathbb{P}^{m-1}_{A}}^{2}
\left(
(\mathbb{P}^{m}_{A})_{\Zar}, 
\mathbb{Z}(1)
\right)\ar[r]
\ar[d]
& 
\HO_{\mathbb{P}^{m-1}_{A}}^{2}
\left(
(\mathbb{P}^{m}_{A})_{\Zar}, 
\mathbb{Z}/p^{r}(1)
\right)\ar[r]\ar[d]
&
\HO_{\mathbb{P}^{m-1}_{A}}^{2}
\left(
(\mathbb{P}^{m}_{A})_{\et}, 
\mathbb{Z}/p^{r}(1)
\right)\ar[d]   \\
\HO_{\Zar}^{2}
\left(
\mathbb{P}^{m}_{A}, 
\mathbb{Z}(1)
\right)\ar[r]
& 
\HO_{\Zar}^{2}
\left(
\mathbb{P}^{m}_{A}, 
\mathbb{Z}/p^{r}(1)
\right)\ar[r]
&
\HO_{\et}^{2}
\left(
\mathbb{P}^{m}_{A}, 
\mathbb{Z}/p^{r}(1)
\right)
}    
\end{equation}
\normalsize
is commutative.
By \cite[p.781, Corollary 3.5]{Ge} and 
\cite[p.786, Corollary 4.4]{Ge}, 
the maps
\begin{align}\label{nms}
\HO_{\mathbb{P}^{m-1}_{A}}^{s}
\left(
(\mathbb{P}^{m}_{A})_{\Zar}, 
\mathbb{Z}(1)
\right)
\to
\HO_{\Zar}^{s}
\left(
\mathbb{P}^{m}_{A}, 
\mathbb{Z}(1)
\right)
\end{align}
and
\begin{align*}
\HO_{\mathbb{P}^{m-1}_{A}}^{s}
\left(
(\mathbb{P}^{m}_{A})_{\Zar}, 
\mathbb{Z}/p^{r}(1)
\right)
\to
\HO_{\Zar}^{s}
\left(
\mathbb{P}^{m}_{A}, 
\mathbb{Z}/p^{r}(1)
\right)
\end{align*}
are isomorphisms for $s\geq 2$.
Moreover, we have
\begin{equation*}
\HO^{3}_{\Zar}(
\mathbb{P}_{A}^{m},
\mathbb{Z}(1))
=0
\end{equation*}
by (\ref{nms}) and \cite[p.780, (6)]{Ge}.
Hence the left and the middle maps in (\ref{diaP})
are isomorphisms and
the first map of the lower row in (\ref{diaP})
is surjective.
So the first map of the upper row in (\ref{diaP})
is surjective.

Let $i: Z\to X$ be a closed immersion of
codimension $1$
between essentially smooth schemes 
over $\spec(A)$.
Then we have
\begin{align*}
\mathcal{H}^{t}\left(
\mathbb{Z}/p^{r}(n)_{\Zar}^{Z}
\right)=0 
~~\textrm{and}~~
R^{t+1}i^{!}\mathbb{Z}/p^{r}(n)_{\Zar}^{X}
=0
\end{align*}
for $t\geq n+1$ by \cite[p.786, Corollary 4.4]{Ge} 
and \cite[p.780, (6)]{Ge}. 
So the Gysin isomorphism 
$\tau_{\leq n+1}\left(\operatorname{Gys}_{i}^{n}\right)$
(cf. \cite[p.547, Theorem 6.1.3]{SaP})   
induces an isomorphism
\begin{equation}\label{GyZp}
\tau_{\leq n+1}R\epsilon_{*}(\operatorname{Gys}_{i}^{n})
: \mathbb{Z}/p^{r}(n-1)_{\Zar}[-2]
\to
Ri^{!}\mathbb{Z}/p^{r}(n)_{\Zar}
\end{equation}
by \cite[p.776, Proposition 2.2.a)]{Ge}
and \cite[p.774, Theorem 1.2.2]{Ge}.
Hence the second map of the upper row is an isomorphism.
Since the morphism (\ref{Gyi}) is the composition of
the map which is induced by the Gysin map
and the right map in (\ref{diaP}),
the image of the morphism (\ref{Gyi})
equals the image of the composition 
of the maps of the lower row in (\ref{diaP}).
Moreover, the diagram
\begin{equation*}
\xymatrix{
\HO^{2}_{\Zar}
(
\mathbb{P}^{m}_{A},
\mathbb{Z}(1)
)\ar[r]
\ar[d]
&
\HO^{2}_{\Zar}
(
\mathbb{P}^{m}_{A},
\mathbb{Z}/p^{r}(1)
)
\ar[d]  \\
\HO^{2}_{\et}
(
\mathbb{P}^{m}_{A},
\mathbb{Z}(1)
)\ar[r]
&
\HO^{2}_{\et}
(
\mathbb{P}^{m}_{A},
\mathbb{Z}/p^{r}(1)
) 
}    
\end{equation*}
is commutative and 
the left map in the diagram is an isomorphism 
by \cite[p.774, Theorem 1.2.2]{Ge}. 
Since 
$\HO^{1}_{\et}(\mathbb{P}^{m}_{A}, \mathbb{G}_{m})$
is generated by
$\operatorname{c}^{1}(\mathcal{O}_{\mathbb{P}_{A}^{m}}(1))$,
the image of the morphism (\ref{Gyi})
is generated by $\xi$.
This completes the proof.
\end{proof}
\begin{thm}\upshape\label{P1h}
Let $A$ be a regular local ring
with $\operatorname{dim}(A)\leq 1$.
Suppose that $\operatorname{char}(A)$
is $p>0$ or $(0, p)$.
Then we have an isomorphism
\begin{equation*}
\HO_{\Zar}^{s}
\left(
\mathbb{P}^{m}_{A},
\tau_{\geq n+1}
R\epsilon_{*}\mathbb{Z}/p^{r}(n)
\right)
=
\bigoplus_{j=0}^{s^{\prime}}
\HO_{\et}^{s-2j}
(
A, 
\mathbb{Z}/p^{r}(n-j)
)
\end{equation*}
for $s\geq n+1$ and any $r>0$ where
$\epsilon: X_{\et}\to X_{\Zar}$ is the canonical map of sites
and $\epsilon_{*}$ is the forgetful functor.
Here $s^{\prime}=s-n-1$ if $s\leq n+m+1$ and
$s^{\prime}=m$ if $s>n+m+1$.
\end{thm}
\begin{proof}\upshape
We have a distinguished triangle
\begin{equation*}
\cdots\to
\mathbb{Z}/p^{r}(n)_{\Zar}
\to
R\epsilon_{*}\mathbb{Z}/p^{r}(n)_{\et}
\to
\tau_{\geq n+1}R\epsilon_{*}\mathbb{Z}/p^{r}(n)_{\et}
\to \cdots
\end{equation*}
by \cite[p.774, Theorem 1.2.2]{Ge}. So it suffices to show that the sequence
\begin{equation}\label{Ze}
0\to   
\HO_{\Zar}^{s}(
\mathbb{P}_{A}^{m},
\mathbb{Z}/p^{r}(n)
)
\to
\HO_{\et}^{s}(
\mathbb{P}_{A}^{m},
\mathbb{Z}/p^{r}(n)
)
\to
\bigoplus_{j=0}^{s^{\prime}}
\HO_{\et}^{s-2j}
(
A, 
\mathbb{Z}/p^{r}(n-j)
)
\to
0 
\end{equation}
is exact. 
Let $f_{m}$ be the projection $\mathbb{P}^{m}_{A}\to\spec(A)$.
In order to show that the sequence (\ref{Ze})
is exact, we show that we have an isomorphism
\begin{equation}\label{bP}
\HO_{\et}^{s-2}(\mathbb{P}_{A}^{m-1}, 
\mathbb{Z}/p^{r}(n-1))
\oplus
\HO_{\et}^{s}(A, \mathbb{Z}/p^{r}(n))
\xrightarrow{((i_{m})_{*}, (f_{m})^{*})}
\HO_{\et}^{s}(\mathbb{P}_{A}^{m}, 
\mathbb{Z}/p^{r}(n))
\end{equation}
where $\mathbb{P}_{A}^{m-1}$ is a hyperplane in $\mathbb{P}_{A}^{m}$
and $i_{m}: \mathbb{P}_{A}^{m-1}\to\mathbb{P}_{A}^{m}$ 
is the corresponding closed immersion.
Let 
$\xi\in \HO^{1}_{\et}(\mathbb{P}^{m^{\prime}}_{A}, \mathbb{G}_{m})$ 
be the same as in Lemma \ref{GyC}.
For $m^{\prime}>0$, 
we have an isomorphism
\begin{equation}\label{PR}
\bigoplus_{j=0}^{m^{\prime}}\xi^{j}\cup - :
\bigoplus_{j=0}^{m^{\prime}}\mathbb{Z}/p^{r}(n-j)_{\et}
[-2j]
\xrightarrow[\sim]{}
R(f_{m^{\prime}})_{*}(
\mathbb{Z}/p^{r}(n)_{\et}
)
\end{equation}
by \cite[p.190, Theorem 4.1]{SaR} and \cite[p.191, Lemma 4.2]{SaR}.
Since 
\begin{equation*}
(f_{m-1})^{*}
=(i_{m})^{*}\circ (f_{m})^{*}    
\end{equation*}
by \cite[p.538, Proposition 4.2.8]{SaP}, 
we have
\begin{align*}
R(f_{m})_{*}(i_{m})_{*}\left(
\xi^{j}\cup (f_{m-1})^{*}(x)
\right)
&=
\xi^{j}\cup
R(f_{m})_{*}\left((i_{m})_{*}(i_{m})^{*}\right)
((f_{m})^{*}(x))\\
&=
\xi^{j}\cup
R(f_{m})_{*}\left(
\operatorname{id}\cup (i_{m})_{*}(1)
\right)
((f_{m})^{*}(x))
\end{align*}
for $x\in\mathbb{Z}/p^{r}(n-j)_{\et}[-2j]$
by Proposition \ref{GyPr}, Corollary \ref{ProjC} and Lemma \ref{GyC}.
By (\ref{PR}) for 
$m^{\prime}=m-1$, we have an isomorphism
\begin{align}\label{IA(m-1)}
&\left(
\bigoplus_{j=0}^{m-1}
\mathbb{Z}/p^{r}(n-1-j)_{\et}[-2j]
\right)[-2]
\oplus \mathbb{Z}/p^{r}(n)_{\et}  \nonumber \\
\xrightarrow{
\left(
(\bigoplus_{j=0}^{m-1}\xi^{j})[-2], 
\operatorname{id}
\right)} 
&R(f_{m-1})_{*}\mathbb{Z}/p^{r}(n-1)_{\et}[-2]
\oplus
\mathbb{Z}/p^{r}(n)_{\et}
\end{align}
and the composite of (\ref{IA(m-1)}) and
\begin{equation}\label{bPS}
R(f_{m-1})_{*}\mathbb{Z}/p^{r}(n-1)_{\et}[-2]
\oplus
\mathbb{Z}/p^{r}(n)_{\et}
\xrightarrow[]{(R(f_{m})_{*}(i_{m})_{*}, (f_{m})^{*})}
R(f_{m})_{*}\mathbb{Z}/p^{r}(n)_{\et}
\end{equation}
%
is an isomorphism
by the isomorphism (\ref{PR}) and Lemma \ref{GyC}. 
So the homomorphism (\ref{bPS}) is an isomorphism.
Hence the homomorphism (\ref{bP}) is an isomorphism.

On the other hand, 
we have an isomorphism
\begin{align}\label{bPZ}
\HO_{\Zar}^{s-2}(\mathbb{P}_{A}^{m-1}, 
\mathbb{Z}/p^{r}(n-1)
)
\oplus
\HO_{\Zar}^{s}(A, \mathbb{Z}/p^{r}(n))  \nonumber \\
\xrightarrow{
(
\tau_{\leq n+1}R\epsilon_{*}(i_{m})_{*}, 
\tau_{\leq n}R\epsilon_{*}(f_{m})^{*}
)
}
\HO_{\Zar}^{s}(\mathbb{P}_{A}^{m}, 
\mathbb{Z}/p^{r}(n)
)
\end{align}
by (\ref{GyZp}) and
\cite[p.781, Corollary 3.5]{Ge}.

Since we have the isomorphisms (\ref{bP}) and (\ref{bPZ}),
the homomorphism (\ref{Ze}) is an isomorphism 
by the induction on $m$.
\end{proof}
\begin{cor}\upshape\label{HBP1}
Let $A$ be a regular local ring with $\operatorname{dim}(A)\leq 1$.
Suppose that $\operatorname{char}(A)$ is $p>0$ or $(0, p)$.
Then we have an isomorphism
\begin{equation*}
\HB^{n+1}(
\mathbb{P}_{A}^{m}
)_{p^{r}} 
=
\HO_{\et}^{n+1}(A,
\mathbb{Z}/p^{r}(n)
)
\end{equation*}
for any integer $r>0$
where $\HB^{n+1}(X)$ 
is the $(n+1)$-th Brauer group of a smooth scheme $X$
(cf. \cite[Definition 3.1]{Sak}) and
$\HB^{n+1}(X)_{p^{r}}$
is the $p^{r}$-torsion subgroup of $\HB^{n+1}(X)$. 

Suppose that $\operatorname{char}(A)=p>0$.
Then
\begin{align*}
\HO_{\Zar}^{s}(
\mathbb{P}^{m}_{A},
R^{n+1}\epsilon_{*}
\mathbb{Z}/p^{r}(n)
)
=
\begin{cases}
\HO_{\et}^{n-s+1}(
A, \mathbb{Z}/p^{r}(n-s)
) 
& (0\leq s\leq m)  \\
0
& (s>m)
\end{cases}
\end{align*}
for any integer $r>0$ where 
$\epsilon: (\mathbb{P}^{m}_{A})_{\et}\to(\mathbb{P}^{m}_{A})_{\Zar}$ 
is the canonical map of sites and $\epsilon_{*}$
is the forgetful functor.
\end{cor}
\begin{proof}\upshape
By 
\cite[Expos\'{e} X, Th\'{e}or\`{e}me 5.1]{SGA4},
the statement directly follows from 
Theorem \ref{P1h}.
\end{proof}
\begin{cor}\upshape
Let $A$ be a regular local ring with $\operatorname{dim}(A)\leq 1$.
Suppose that $\operatorname{char}(A)$ is $p>0$ or $(0, p)$.
Let $r>0$ be an integer.
Then the natural map
\begin{equation}\label{HBHas}
\HO^{n+1}_{\et}(
k(
\mathbb{P}_{A}^{m}
),
\mathbb{Z}/p^{r}(n)
)    
\to
\prod
_{x\in (\mathbb{P}_{A}^{m})^{(1)}}
\HO^{n+1}_{\et}(
k(
\mathcal{O}_{
\mathbb{P}_{A}^{m}, x}^{h}
),
\mathbb{Z}/p^{r}(n)
)
\end{equation}
is injective where
\begin{math}
\mathcal{O}_{
\mathbb{P}_{A}^{m}, x}^{h}
\end{math}
is the henselization
of $\mathbb{P}_{A}^{m}$ at a point $x\in (\mathbb{P}_{A}^{m})^{(1)}$.

Suppose that $A$ is henselian discrete valuation ring. 
Then the homomorphism
\begin{equation}\label{HBAr}
\HB^{n+1}(
\mathbb{P}_{A}^{m}
)_{p^{r}}
\simeq
\HB^{n+1}(
\mathbb{P}_{k}^{m}
)_{p^{r}}    
\end{equation}
is an isomorphism.
\end{cor}
\begin{proof}\upshape
By \cite[Theorem 4.6]{Sak} and Corollary \ref{HBP1}, 
the sequence
\begin{align*}
0
&
\to 
\HO^{n+1}_{\et}(A, \mathbb{Z}/p^{r}(n))
\to
\HO^{n+1}_{\et}(
k(
\mathbb{P}_{A}^{m}
), 
\mathbb{Z}/p^{r}(n))
\\
&\to
\bigoplus_{x\in(\mathbb{P}_{A}^{m})^{(1)}}
\HO^{n+2}_{x}(
(\mathbb{P}_{A}^{m})_{\et},
\mathbb{Z}/p^{r}(n)
)
\end{align*}
is exact.
Moreover, we have an isomorphism
\begin{equation*}
\HO^{n+1}_{\et}(
\mathcal{O}^{h}_{\mathbb{P}^{m}_{A}, x}, \mathbb{Z}/p^{r}(n)
)    
\simeq 
\HO^{n+1}_{\et}(
\kappa(x), \mathbb{Z}/p^{r}(n)
)    
\end{equation*}
for $x\in (\mathbb{P}^{m}_{A})^{(1)}$ by 
\cite[Theorem 5.3]{Sak}, \cite[Theorem 5.6]{Sak} and 
\cite[p.777, The proof of Proposition 2.2.b)]{Ge}. 
Hence the sequence
\begin{equation*}
0\to 
\HO^{n+1}_{\et}(
\kappa(x),
\mathbb{Z}/p^{r}(n)
)
\to
\HO^{n+1}_{\et}(
k(
\mathcal{O}_{\mathbb{P}_{A}^{m}, x}^{h}
),
\mathbb{Z}/p^{r}(n)
)
\to
\HO^{n+2}_{x}(
\mathbb{P}_{A}^{m}, \mathbb{Z}/p^{r}(n)
)
\end{equation*}
is exact for $x\in (\mathbb{P}_{A}^{m})^{(1)}$.
Let $X\to Y$ be a morphism between smooth schemes over $\spec(A)$.
Since we have
\begin{equation*}
\HB^{n+1}(Z)
=
\Gamma\left(
Z,
R^{n+1}\epsilon_{*}
\mathbb{Q}/\mathbb{Z}(n)_{\et}
\right)
\end{equation*}
for a smooth scheme $Z$ over $\spec(A)$ by
\cite[Proposition 3.7]{Sak},
we can define a homomorphism
\begin{equation*}
\HB^{n+1}(Y)\to  
\HB^{n+1}(X)
\end{equation*}
by the morphism (\ref{adipul}) and \cite[p.732, Theorem 3.5.1]{SaL}.
Let the residue field $\kappa(x)$ of a point 
\begin{math}
x\in (\mathbb{P}_{A}^{m})^{(1)}
\end{math}
be the fraction field of $\mathbb{P}_{A}^{m-1}$
and the residue field $\kappa(y)$ of a point 
\begin{math}
y\in (\mathbb{A}_{A}^{m-1})^{(m-1)}
\end{math}
be $k(A)$. Then the composite of the homomorphism 
\begin{equation}\label{AAm}
\HB^{n+1}(A)
\to
\HB^{n+1}(\mathbb{A}_{A}^{m-1}) 
\end{equation}
and the homomorphism 
\begin{equation*}
\HB^{n+1}(\mathbb{A}_{A}^{m-1})
\to
\HB^{n+1}(\kappa(y))    
\end{equation*}
agrees with the natural map
\begin{equation*}
\HB^{n+1}(A)\to 
\HB^{n+1}(k(A))
\end{equation*}
by \cite[p.538, Proposition 4.2.8]{SaP} and \cite[p.734, Corollary 3.5.3]{SaL}. 
So the homomorphism (\ref{AAm})
is injective by \cite[Theorem 4.6]{Sak} and the homomorphism
\begin{equation*}
\HO^{n+1}_{\et}(
A, \mathbb{Z}/p^{r}(n)
)
\to
\HO^{n+1}_{\et}(\kappa(x),
\mathbb{Z}/p^{r}(n)
)
\end{equation*}
is also injective by 
\cite[Theorem 4.6]{Sak} and \cite[Proposition 3.4]{Sak}. 
Hence the natural map (\ref{HBHas}) is injective.

By Corollary \ref{HBP1}, \cite[Theorem 5.3]{Sak}
and \cite[Theorem 5.6]{Sak}, 
the homomorphism (\ref{HBAr}) is an isomorphism.
This completes the proof.
\end{proof}
\begin{cor}\upshape\label{Ppbm2}
Let $A$ be a henselian discrete valuation ring
of mixed characteristic $(0, p)$. 
Suppose that $A$ contains $p$-th roots of unity.
Then we have an isomorphism
\begin{align*}
\HO_{\Nis}^{s}\left(
\mathbb{P}^{m}_{A}, R^{n+1}\alpha_{*}\mathbb{Z}/p(n)
\right)
\simeq
\begin{cases}
\HO_{\et}^{n-s+1}
(A, \mathbb{Z}/p(n-s))
& (0\leq s\leq m)  \\
0
& (m<s)
\end{cases}
\end{align*}
for 
where 
$\alpha: (\mathbb{P}_{A}^{m})_{\et}\to(\mathbb{P}_{A}^{m})_{\Nis}$ 
is the canonical map of sites
and $\alpha_{*}$ is the forgetful functor.
Moreover, the sequence
\begin{align}\label{Pje}
0\to&    
\bigoplus_{x\in (\mathbb{P}_{A}^{m})^{(0)}}
\HO^{n+2}_{x}\left(
(\mathbb{P}^{m}_{A})_{\et}, 
\mathbb{Z}/p(n)
\right)
\to
\bigoplus_{x\in (\mathbb{P}_{A}^{m})^{(1)}}
\HO^{n+3}_{x}\left(
(\mathbb{P}^{m}_{A})_{\et}, \mathbb{Z}/p(n)
\right) \nonumber \\
\to&
\bigoplus_{x\in (\mathbb{P}_{A}^{m})^{(2)}}
\HO^{n+4}_{x}\left(
(\mathbb{P}^{m}_{A})_{\et}, \mathbb{Z}/p(n)
\right)
\to\cdots
\end{align}
is exact.
\end{cor}
\begin{proof}\upshape
Let $k$ be the residue field of $A$,
$i: \mathbb{P}^{m}_{k}\to\mathbb{P}^{m}_{A}$
the inclusion of the closed fiber
and 
$j: \mathbb{P}^{m}_{k(A)}\to\mathbb{P}^{m}_{A}$
the inclusion of the generic fiber. 
Then we have an quasi-isomorphism
\begin{equation*}
\tau_{\geq n+2}R\alpha_{*}j_{!}\mu_{p}^{\otimes n}
\xrightarrow{\simeq}
\tau_{\geq n+2}R\alpha_{*}\mathbb{Z}/p(n)
\end{equation*}
by Proposition \ref{compi} and \cite[p.777, The proof of Proposition 2.2.b)]{Ge}.
Since $A$ contains $p$-th roots of unity, we have
\begin{equation*}
\tau_{\geq n+2}R\alpha_{*}j_{!}\mu_{p}^{\otimes n}
\simeq
\tau_{\geq n+2}R\alpha_{*}\mathbb{Z}/p(n-1).
\end{equation*}
Let $R$ be the henselization of $\mathbb{P}^{m}_{A}$ at a point of 
$i(\mathbb{P}^{m}_{k})$ and 
$\pi$ a prime element of $A$.
Then we have an isomorphism
\begin{equation*}
\HO^{k}_{\et}(R, \mathbb{Z}/p(n))
\simeq  
\HO^{k}_{\et}(R/(\pi), \mathbb{Z}/p(n))
\end{equation*}
for $k\geq n+1$
by Proposition \ref{compi} and the isomorphism (\ref{mpi}). 
Hence we have a distinguished triangle
\begin{equation}\label{hji}
\cdots
\to
\tau_{\geq n+1}R\alpha_{*}\mathbb{Z}/p(n-1)
\to
\tau_{\geq n+1}R\alpha_{*}\mathbb{Z}/p(n)
\to
i_{*}\left(
\tau_{\geq n+1}R\alpha_{*}\mathbb{Z}/p(n)
\right)
\to 
\cdots
\end{equation}
by
\cite[p.776, Proposition 2.2]{Ge} and \cite[pp.75--76, II, Remark 3.13]{M}.
Let $X$ be a smooth scheme over 
the spectrum of a regular local ring $B$
with $\operatorname{dim}(B)\leq 1$
and $\epsilon: X_{\et}\to X_{\Zar}$ 
the canonical map of sites. 
Then we have an isomorphism
\begin{equation*}
\HO_{\Zar}^{t}\left(
X, 
\tau_{\geq n+1}R\epsilon_{*}\mathbb{Z}/p(n)
\right)
\simeq
\HO_{\Nis}^{t}\left(
X, 
\tau_{\geq n+1}R\alpha_{*}\mathbb{Z}/p(n)
\right)
\end{equation*}
for $t\geq n+1$ by 
\cite[p.781, Proposition 3.6]{Ge}. Hence we have
\begin{equation}\label{T2}
\HO^{t}_{\Nis}\left(
\mathbb{P}^{m}_{A},
\tau_{\geq n+1}R\alpha_{*}\mathbb{Z}/p(n-1)
\right)
=0
\end{equation}
for $t\geq n+1$ by the distinguished triangle (\ref{hji}) and Theorem \ref{P1h}.
Since we have a distinguished triangle
\begin{equation*}
\cdots
\to
R^{n}\alpha_{*}\mathbb{Z}/p(n-1)
[-n]
\to
\tau_{\geq n}R\alpha_{*}\mathbb{Z}/p(n-1)
\to
\tau_{\geq n+1}R\alpha_{*}\mathbb{Z}/p(n-1)
\to 
\cdots,
\end{equation*}
we have an isomorphism
\begin{align*}
\HO_{\Nis}^{s}\left(
\mathbb{P}^{m}_{A}, R^{n+1}\alpha_{*}\mathbb{Z}/p(n)
\right)
\simeq
\begin{cases}
\HO_{\et}^{n-s+1}
(A, \mathbb{Z}/p(n-s))
& (0\leq s\leq m)  \\
0
&
(m<s)
\end{cases}
\end{align*}
%
by the equation (\ref{T2}) and Theorem \ref{P1h}.
Since we have a distinguished triangle
%
\begin{align*}
\cdots
&\to
R^{n+1}\alpha_{*}\mathbb{Z}/p(n-1)
[-(n+1)]
\to
\tau_{\geq n+1}R\alpha_{*}\mathbb{Z}/p(n-1)
\\
&\to
\tau_{\geq n+2}R\alpha_{*}\mathbb{Z}/p(n-1)
\to 
\cdots
\end{align*}
%
and a quasi-isomorphism
\begin{equation*}
\tau_{\geq n+2}R\alpha_{*}\mathbb{Z}/p(n-1)    
\simeq
\tau_{\geq n+2}R\alpha_{*}\mathbb{Z}/p(n), 
\end{equation*}
we have
\begin{equation*}
\HO^{s}_{\Nis}
(\mathbb{P}^{m}_{A}, 
R^{n+1}\alpha_{*}\mathbb{Z}/p(n-1))
=0
\end{equation*}
for any integer $s\geq 0$
by the equation (\ref{T2}).
Hence the sequence (\ref{Pje}) 
is exact by Lemma \ref{Zeh}.
This completes the proof.
\end{proof}
\begin{prop}\upshape\label{AHV}
Let $k$ be a field. 
Then we have
\begin{align*}
\HO^{j}_{\Zar}\left(
\mathbb{A}^{i}_{k}, \mathcal{H}^{s}(\mathbb{Z}/m(n))
\right)
=0
\end{align*}
for $j>0$ and $s\leq n$ in the following cases:
\begin{itemize}
\item[(i)] In the case where $m$ is
a prime number which is prime to $\operatorname{char}(k)$. 
\item[(ii)] In the case where
$\operatorname{char}(k)=p>0$
and $m=p^{r}$.
\end{itemize}
\end{prop}
\begin{proof}
First we prove that the statement in the case of
(i) is true by induction on $n$.

Assume that $n=0$. Then we have
\begin{equation*}
\HO^{j}_{\Zar}\left(
\mathbb{A}^{i}_{k}, \mathcal{H}^{0}(\mathbb{Z}/m(0))
\right)   
=
\HO^{j}_{\Zar}\left(
\mathbb{A}^{i}_{k}, \mathbb{Z}/m(0)
\right)  
=
\HO^{j}_{\Zar}\left(
k, \mathbb{Z}/m(0)
\right)
=0
\end{equation*}
for $j>0$ by \cite[p.781, Corollary 3.5]{Ge}
and \cite[p.786, Corollary 4.4]{Ge}.

Assume that the statement in the case of (i)
is true for $n\leq N$.
We prove that the statement is true
for $n=N+1$ in the case of (i).

By a standard norm argument, 
the proof of the statement in the case (i)
is reduced to the case where
$m$-th roots of unity is contained in $k$ and $s=n$.
So we assume that $k$ contains $m$-th roots of unity.
By \cite[p.774, Theorem 1.2]{Ge}, we have
an isomorphism
\begin{equation}\label{BLe}
\tau_{\leq n}R\epsilon_{*}
\mu_{m}^{\otimes n}
\simeq
\mathbb{Z}/m(n)_{\Zar}
\end{equation}
for any integer $n\geq 0$.
Since $k$ contains $m$-th roots of unity, 
we have an isomophism
\begin{equation*}
\mathcal{H}^{n}
\left(
\mathbb{Z}/m(n)_{\Zar}
\right)    
\simeq
\mathcal{H}^{n}
\left(
\mathbb{Z}/m(N+1)_{\Zar}
\right) 
\end{equation*}
for $n\leq N+1$ by (\ref{BLe}).
Hence we have
\begin{equation}\label{Bnl}
\HO_{\Zar}^{n+j}\left(
\mathbb{A}^{i}_{k},
\tau_{\leq n-1}\mathbb{Z}/m(n)
\right)    
=0
\end{equation}
for $j>0$ and $n=N+1$
by the spectral sequence
\begin{equation*}
E^{s, t}_{2}
=\HO^{s}_{\Zar}(
\mathbb{A}^{i}_{k}, \mathcal{H}^{t}(\tau_{\leq n-1}\mathbb{Z}/m(n))
)
\Rightarrow
E^{s+t}
=\HO^{s+t}_{\Zar}(
\mathbb{A}^{i}_{k}, 
\tau_{\leq n-1}\mathbb{Z}/m(n)
)
\end{equation*}
for $n=N+1$ and the assumption.
Since we have a distinguished triangle
\begin{equation}\label{dt}
\cdots\to
\tau_{\leq n-1}\mathbb{Z}/m(n)_{\Zar}
\to
\mathbb{Z}/m(n)_{\Zar}
\to
\mathcal{H}^{n}(\mathbb{Z}/m(n)_{\Zar})[-n]
\to\cdots
\end{equation}
by \cite[p.786, Corollary 4.4]{Ge},
the statement in the case of (i) 
follows from (\ref{Bnl}) and 
\cite[p.781, Corollary 3.5]{Ge}.

Next we prove that the statement in the case of (ii) is true.
Then we have a quasi-isomorphism
\begin{equation}\label{vpn}
\mathbb{Z}/p^{r}(n)_{\Zar}
\simeq
\mathcal{H}^{n}(\mathbb{Z}/p^{r}(n)_{\Zar})[-n]   
\end{equation}
by \cite[Theorem 8.3]{G-L}. Hence we have
\begin{equation*}
\HO^{j}_{\Zar}\left(
\mathbb{A}^{i}_{k}, \mathcal{H}^{n}(\mathbb{Z}/p^{r}(n))
\right) 
=
\HO^{n+j}_{\Zar}\left(
\mathbb{A}^{i}_{k}, \mathbb{Z}/p^{r}(n)
\right) 
=
\HO^{n+j}_{\Zar}\left(
k, \mathbb{Z}/p^{r}(n)
\right)
=0
\end{equation*}
for $j>0$ by (\ref{vpn}) and \cite[p.781, Corollary 3.5]{Ge}.
This completes the proof.
\end{proof}
\begin{cor}\upshape\label{Afex}
Let $A$ be a regular local ring with
$\operatorname{dim}(A)\leq 1$. 
Then the sequence
\begin{align}\label{Ag}
0&\to  
\HO^{s}_{\Zar}(A, \mathbb{Z}/m(n))
\to
\HO^{s}_{\Zar}(
k(\mathbb{A}^{i}_{A}),
\mathbb{Z}/m(n)
) \nonumber 
\\
&\to
\bigoplus_{x\in (\mathbb{A}_{A}^{i})^{(1)}}
\HO^{s-1}_{\Zar}(\kappa(x), \mathbb{Z}/m(n-1)) 
\to\cdots
\end{align}
is exact for $s\leq n$ in the following cases:
\begin{itemize}
\item[(i)] In the case where $m$ is
a prime number which is prime to $\operatorname{char}(A)$. 
\item[(ii)] In the case where
$\operatorname{char}(A)=p>0$
and $m=p^{r}$.
\end{itemize}
\end{cor}
\begin{proof}\upshape
The exactness of (\ref{Ag}) except at the first
two terms follows directly from Lemma \ref{Zeh} and
Proposition \ref{AHV}.
We prove the exactness of (\ref{Ag}) at the first 
two terms.
We have a distinguished triangle
\begin{equation*}
\cdots\to
\tau_{\leq s-1}(\mathbb{Z}/m(n)_{\Zar})
\to
\mathbb{Z}/m(n)_{\Zar}
\to
\tau_{\geq s}(\mathbb{Z}/m(n)_{\Zar})
\to \cdots
\end{equation*}
and
\begin{equation*}
\HO^{t}_{\Zar}(\mathbb{A}^{i}_{A},
\tau_{\leq s-1}(\mathbb{Z}/m(n)))=0    
\end{equation*}
for $t\geq s$ by Proposition \ref{AHV}.
Hence we have
\begin{equation*}
\HO^{s}_{\Zar}(\mathbb{A}^{i}_{A}, \mathbb{Z}/m(n))
=\HO^{0}_{\Zar}(\mathbb{A}^{i}_{A}, \mathcal{H}^{s}(\mathbb{Z}/m(n)))
\end{equation*}
and so the sequence
\begin{equation*}
0\to  
\HO^{s}_{\Zar}(A, \mathbb{Z}/m(n))
\to
\HO^{s}_{\Zar}(
k(\mathbb{A}^{i}_{A}),
\mathbb{Z}/m(n)
)
\to
\bigoplus_{x\in (\mathbb{A}_{A}^{i})^{(1)}}
\HO^{s-1}_{\Zar}(\kappa(x), \mathbb{Z}/m(n-1)) \nonumber  
\end{equation*}
is exact by \cite[p.774, Theorem 1.2.5]{Ge}. This completes the proof.
\end{proof}
\begin{cor}\upshape
Let $A$ and $m$ be same as in
Corollary \ref{Afex}.
Then we have
\begin{equation*}
\HO^{j}_{\Zar}(
\mathbb{P}^{i}_{A}, 
R^{s}\epsilon_{*}\mathbb{Z}/m(n)
)
=
\HO_{\et}^{s-j}(
A,
\mathbb{Z}/m(n-j)
)
\end{equation*}
for $s\leq n$ and $j\leq i$. 
\end{cor}
\begin{proof}\upshape
Since
\begin{equation*}
(\mathbb{P}_{A}^{i})^{(j)}
=(\mathbb{P}_{A}^{i-1})^{(j-1)}
\oplus
(\mathbb{A}_{A}^{i})^{(j)},
\end{equation*}
we have
\begin{equation*}
\HO^{j}_{\Zar}(
\mathbb{P}^{i}_{A}, 
R^{s}\epsilon_{*}
\mathbb{Z}/m(n)
)
=\HO_{\Zar}^{0}(
\mathbb{P}^{i-j}_{A},
R^{s-j}\epsilon_{*}
\mathbb{Z}/m(n-j)
)
\end{equation*}
by Corollary \ref{Afex}. Since $\mathbb{P}^{i}_{A}$
is covered by open sets
$\mathbb{A}^{i}_{A}$, we have
\begin{equation*}
\HO_{\Zar}^{0}(
\mathbb{P}^{i-j}_{A},
R^{s-j}\epsilon_{*}
\mathbb{Z}/m(n-j)
)
=
\HO^{s-j}_{\et}(
A, \mathbb{Z}/m(n-j)
)
\end{equation*}
by Corollary \ref{Afex}.
This completes the proof.
\end{proof}
\begin{prop}\upshape\label{A1e}
Let $k$ be a field of characteristic $p>0$.
Then the sequence
\begin{align*}
0
&\to
\HO^{1}_{\et}\left(
\mathbb{A}^{i}_{k},
\nu_{r}^{n}
\right)
\to
\HO^{1}_{\et}\left(
k(\mathbb{A}_{k}^{i}),
\nu_{r}^{n}
\right)
\to
\bigoplus_{
x\in (\mathbb{A}_{k}^{i})^{(1)}
}
\HO_{x}^{2}((\mathbb{A}_{k}^{i})_{\et}, \nu_{r}^{n})
\\
&\to 
\bigoplus_{
x\in (\mathbb{A}_{k}^{i})^{(2)}
}
\HO_{x}^{3}((\mathbb{A}_{k}^{i})_{\et}, \nu_{r}^{n})
\to\cdots 
\end{align*}
is exact for $r>0$.
\end{prop}
\begin{proof}\upshape
We have a distinguished triangle
\begin{equation*}
\cdots
\to \mathbb{Z}/p^{r}(n)_{\Zar}
\to R\epsilon_{*}\mathbb{Z}/p^{r}(n)_{\et}
\to R^{n+1}\epsilon_{*}\mathbb{Z}/p^{r}(n)_{\et}[-n-1]
\to \cdots
\end{equation*}
by \cite[p.774, Theorem 1.2.2]{Ge}. Hence we have
\begin{equation*}
\HO^{n+1}_{\et}(\mathbb{A}_{k}^{i}, \nu^{n}_{r})
=\HO^{0}\left(
\mathbb{A}_{k}^{i}, R^{n+1}\epsilon_{*}\mathbb{Z}/p^{r}(n)
\right)    
\end{equation*}
and
\begin{equation*}
\HO^{j}
\left(
\mathbb{A}_{k}^{i}, R^{n+1}\epsilon_{*}\mathbb{Z}/p^{r}(n)
\right)   
=0   
\end{equation*}
for $j>0$ by 
\cite[p.786, Corollary 4.4]{Ge}
and 
\cite[Expos\'{e} X, Th\'{e}or\`{e}me 5.1]{SGA4}.
\end{proof}

The following is a generalization of a result of S.Yuan (cf. 
\cite{Y}, \cite[p.153, IV, Exercise 2.20 (d)]{M}):
\begin{prop}\upshape\label{infh}
For a field $k$ of positive characteristic $p>0$, there is an exact sequence
\begin{equation}\label{A1h}
0\to 
\HO^{1}_{\et}(\mathbb{A}^{1}_{k}, \nu_{r}^{n})
\to
\HO^{1}_{\et}(\tilde{K}, \nu_{r}^{n})
\to
\HO^{1}_{\et}(k, \nu_{r}^{n-1})
\to 0
\end{equation}
where $\tilde{K}$ is the henselization of $k(\mathbb{P}^{1}_{k})$ at the point at
infinity.
\end{prop}
\begin{proof}
By Proposition \ref{A1e} and 
Corollary \ref{HBP1}, 
the sequence
\begin{equation*}
0\to 
\HO^{1}_{\et}(k, \nu_{r}^{n})
\to
\HO^{1}_{\et}(\mathbb{A}^{1}_{k}, \nu_{r}^{n})
\to
\HO^{2}_{\infty}\left(
(\mathbb{P}^{1}_{k})_{\et}, \nu_{r}^{n}
\right)
\end{equation*}
is exact. Moreover, the sequence
\begin{equation}\label{hep}
0\to 
\HO^{1}_{\et}(k, \nu_{r}^{n})
\to
\HO^{1}_{\et}(\tilde{K}, \nu_{r}^{n})
\to
\HO^{2}_{\infty}\left(
(\mathbb{P}^{1}_{k})_{\et}, \nu_{r}^{n}
\right)
\to 0
\end{equation}
is exact by \cite[Theorem 5.3]{Sak}, 
\cite[Expos\'{e} X, Th\'{e}or\`{e}me 5.1]{SGA4}
and \cite[p.93, III, Corollary 1.28]{M}. 
Hence the first map in the sequence (\ref{A1h}) 
is injective by the snake lemma.
Put
\begin{equation*}
C=  
\operatorname{Im}\Bigl(
\HO^{1}_{\et}(\mathbb{A}^{1}_{k}, \nu_{r}^{n})
\to
\HO^{2}_{\infty}\left(
(\mathbb{P}^{1}_{k})_{\et}, \nu_{r}^{n}
\right)
\Bigr).
\end{equation*}
Since we have
\begin{equation*}
\HO^{2}_{\et}(\mathbb{A}_{k}^{1}, \nu_{r}^{n})
=0
\end{equation*}
by \cite[Expos\'{e} X, Th\'{e}or\`{e}me 5.1]{SGA4}, 
the sequence
\begin{equation*}
0\to 
C \to
\HO^{2}_{\infty}\left(
(\mathbb{P}^{1}_{k})_{\et}, \nu_{r}^{n}
\right)
\to
\HO^{2}_{\et}\left(
\mathbb{P}^{1}_{k}, \nu_{r}^{n}
\right)
\to 0
\end{equation*}
is exact by the definition of $C$. 
Moreover, 
we have an isomorphism
\begin{equation*}
\HO^{2}_{\et}\left(
\mathbb{P}^{1}_{k}, \nu_{r}^{n}
\right)
\simeq
\HO^{1}_{\et}(k, \nu_{r}^{n-1})   
\end{equation*}
by \cite[p.14, Corollaire 2.1.15]{Gr}.
We have a commutative diagram
\begin{equation}\label{ti}
\xymatrix{
& 
& \HO^{1}_{\et}(\mathbb{P}_{k}^{1}, \nu_{r}^{n})
\ar[r]\ar[d]
& \HO^{1}_{\et}(\mathbb{A}_{k}^{1}, \nu_{r}^{n})
\ar[r]\ar[d]
& C \ar[r]\ar[d]
& 0 \\
& 0 \ar[r]
&
\HO^{1}_{\et}(k, \nu_{r}^{n})  
\ar[r]
&
\HO^{1}_{\et}(\tilde{K}, \nu_{r}^{n})  
\ar[r]
&
\HO^{2}_{\infty}\left(
(\mathbb{P}^{1}_{k})_{\et}, \nu_{r}^{n}
\right)\ar[r]
&
0
}    
\end{equation}
where the sequences are exact
by the definition of $C$ and (\ref{hep}). 
Since the composite of the natural map 
\begin{equation*}
\HO^{1}_{\et}(k, \nu_{r}^{n})
\to
\HO^{1}_{\et}(\mathbb{P}^{1}_{k}, \nu_{r}^{n})
\end{equation*}
and
the left map in the diagram (\ref{ti}) is the identity map,
the left map in the diagram (\ref{ti}) is surjective.
Hence the statement follows from the snake lemma. 
\end{proof}
\begin{rem}\upshape
As mentioned in \cite[p.178, Introduction]{SaR}, 
the $p$-adic \'{e}tale Tate twists do not satisfy homotopy
invariance. 
Indeed, we can show this fact as follows:
Let $A$ be a henselian discrete valuation ring of
mixed characteristic $(0, p)$, 
$k$ the residue field of $A$,
$K$ the field of fractions of $A$ and
$n\geq 1$ an integer. 
Let 
$j:\mathbb{A}^{1}_{K}\to\mathbb{A}^{1}_{A}$
be the open immersion.
Suppose that $k$ is a separably closed field
and $A$ contains $p$-th roots of unity. 

Assume that the $p$-adic \'{e}tale Tate twists satisfy homotopy invariance.
Then we have
\begin{equation*}
\HO^{1}_{\et}(
\mathbb{A}^{1}_{k}, \nu_{1}^{n})
=
\HO^{n+2}_{\et}(
\mathbb{A}^{1}_{A},
j_{!}\mu_{p})
=0
\end{equation*}
by Proposition \ref{compi}
and \cite[Expos\'{e} X, Th\'{e}or\`{e}me 5.1]{SGA4}.
On the other hand,
\begin{equation*}
\HO^{1}_{\et}(
\mathbb{A}^{1}_{k}, \nu_{1}^{n})
\neq 0
\end{equation*}
if $[k:k^{p}]\geq p^{n}$
by \cite[p.48, Th\'{e}or\`{e}me 1.2]{G-O} and Proposition \ref{infh}.
Hence the $p$-adic \'{e}tale Tate twists do not satisfy homotopy
invariance in general.
\end{rem}

%
%

%

%

%

\begin{thebibliography}{99}
\bibitem{A-B-B-G}
\textsc{A.Auel, A.Bigazzi, C.B\"{o}hning and H.-C.Graf von Bothmer},
Universal triviality of the Chow group 
of $0$-cycles and the Brauer group,
Int.Math.Res.Not.2021, No 4,
2479-–2496 (2021).
\bibitem{SGA4}
\textsc{Artin, M., A. Grothendieck, and J.-L. Verdier.},
Th\'{e}orie des Topos et Cohomologie \'{E}tale des Sch\'{e}mas
(SGA 4), Tome 3.
Lecture Notes in Mathematics, vol. 305. Heidelberg, Berlin:
Springer-Verlag, 1973.
\bibitem{B-K}
\textsc{S.Bloch and K.Kato},
$p$-adic \'{e}tale cohomology. 
Inst. Hautes \'{E}tudes Sci. Publ. Math. No. 63 (1986), 107--152.
\bibitem{C-H-K}
\textsc{
Colliot-Th\'{e}l\`{e}ne, Jean-Louis; 
Hoobler, Raymond T.; Kahn, Bruno},
The Bloch-Ogus-Gabber theorem,
Snaith, Victor P. (ed.), Algebraic K-theory. 
Papers from the 2nd Great Lakes conference, Canada, March 1996, 
in memory of Robert Wayne Thomason. 
Providence, RI: 
American Mathematical Society. Fields Inst. Commun. 16, 31--94 (1997).
\bibitem{G-O}
\textsc{O.Gabber and F.Orgogozo},
Sur la $p$-dimension des corps,
Invent.Math.174,
No.1,
47--80
(2008).
\bibitem{Ge}
\textsc{T.Geisser},
Motivic cohomology over Dedekind rings,
Math. Z. 248 (2004), no. 4, 773–-794. 
\bibitem{G-L}
\textsc{T.Geisser and M.Levine},
The $K$-Theory of fields in characteristic $p$,
Invent. Math. 139 (2000), no. 3, 459--493.
\bibitem{Ge-L2}
\textsc{T.Geisser and M.Levine}, 
The Bloch-Kato conjecture and a theorem of Suslin-Voevodsky,
J. Reine Angew. Math. 530,
55--103 (2001).
\bibitem{Gr}
\textsc{M.Gros},
Classes de Chern et classes de cycles en cohomologie de Hodge-Witt 
logarithmique,
M\'{e}m. Soc. Math. Fr., Nouv. S\'{e}r. 21, 87 p. (1985). 
\bibitem{I}
\textsc{L.Illusie},
Complexe de de Rham-Witt et cohomologie cristalline,
Ann. Sci. \'{E}cole Norm. Sup. (4) 12 (1979),
no. 4, 501--661.
\bibitem{H}
\textsc{R.Hartshorne},
\textit{Residues and duality}.
Lecture Notes in Math, No. 20,
Springer-Verlag, Berlin, 1966.
\bibitem{Hy}
\textsc{O.Hyodo},
A note on {$p$}-adic \'{e}tale cohomology in the semistable
reduction case,
Invent. Math. 91
(1988),
no.3,
543--557.
\bibitem{L}
\textsc{M.Levine},
$K$-theory and motivic cohomology of schemes, 
Preprint, 
1999.
\bibitem{Ma}
\textsc{H.Matsumura},
Commutative ring theory, 
Translated from the Japanese by M. Reid. 
Cambridge Studies in Advanced Mathematics, 8. 
Cambridge University Press, Cambridge, 1986.
\bibitem{Mc}
\textsc{J.McCleary},
A user’s guide to spectral sequences. 2nd ed,
Cambridge Studies in Advanced Mathematics. 58, 
Cambridge University Press, 2001, Cambridge.
\bibitem{M}
\textsc{J.S.Milne},
Etale Cohomology,
Princeton Math. Ser. 
Princeton University Press, Princeton, N.J., 
1980.
\bibitem{N}
\textsc{J.Neukirch; A.Schmidt;
K.Winberg},
Cohomology of Number Fields. 2nd ed,
Grundlehren der Mathematischen Wissenschaften 323.Berlin,
2008.
\bibitem{Ni}
\textsc{Y.Nisnevich},
The completely decomposed topology on schemes and 
associated descent spectral sequences in algebraic K-theory, 
Algebraic $K$-theory: Connections with geometry and topology, 
Proc. Meet., Lake Louise/Alberta (Can.) 1987, NATO ASI Ser., Ser. C 279, 
241--342 (1989).
\bibitem{O}
\textsc{S.Otabe},
On the mod $p$ unramified cohomology of varieties having universally trivial Chow group,
Manuscripta Math.171 (2023),
no.1--2, 215--239.
\bibitem{Sak}
\textsc{M.Sakagaito},
On a generalized Brauer group in mixed
characteristic cases,
J. Math. Sci. Univ. Tokyo 27, 
No.1, 29--64 (2020). 
\bibitem{Sak4}
\textsc{M.Sakagaito},
On \'{e}tale hypercohomology of henselian regular local rings with values
in $p$-adic \'{e}tale Tate twists,
arXiv:2002.04797v6.
\bibitem{Salt}
\textsc{D.Saltman},
The Brauer group and the center of
generic matrices,
J. Algebra 97, 53-67 (1985).
\bibitem{SaL}
\textsc{K.Sato},
Logarithmic Hodge-Witt sheaves on normal crossing varieties,
Math. Z. 257, 
No. 4, 
707--743 (2007).
\bibitem{SaP}
\textsc{K.Sato},
$p$-adic \'{e}tale Tate twists and arithmetic duality.  
With an appendix by Kei Hagihara. 
Ann. Sci. \'{E}cole Norm. Sup. (4) 40 (2007), 
no. 4, 519--588. 
\bibitem{SaR}
\textsc{K.Sato}, 
Cycle classes for $p$-adic \'{e}tale Tate twists and the image of $p$-adic regulators, 
Doc. Math.18 (2013), 
177--247.
\bibitem{Sh}
\textsc{A.Shiho},
On logarithmic Hodge-Witt cohomology of regular schemes,
J. Math. Sci., Tokyo 14, No. 4, 567--635 (2007).
\bibitem{SP}
\textsc{The Stacks project authors}.
\textit{The stacks project}. 
https://stacks.math.columbia.edu, 2023.
\bibitem{V}
\textsc{V.Voevodsky},
On motivic cohomology with $\mathbf{Z}/l$-coefficients,
Ann. of Math. (2) 174 (2011), no. 1, 401--438.
\bibitem{Y}
\textsc{S.Yuan},
On the Brauer groups of local fields,
Ann.Math.(2) 82,
434--444 (1965).
\end{thebibliography}
\end{document}